\documentclass[11pt, reqno, a4]{amsart}

\usepackage{textcmds} 

\usepackage[margin=3cm]{geometry}

\usepackage{graphics,graphicx}  

\usepackage{textcmds}  
\usepackage{amsmath, amssymb, amsfonts, amstext, verbatim, amsthm, mathrsfs}
\usepackage[mathcal]{eucal}
\usepackage{microtype}
\usepackage[all]{xy}
\usepackage[modulo]{lineno}
\usepackage[usenames]{color}
\usepackage{aliascnt}
\usepackage{enumitem}
\usepackage{xspace}
\usepackage{amsfonts}
\usepackage{amssymb}
\usepackage[centertags]{amsmath}
\usepackage{amsthm}
\usepackage[margin=3cm]{geometry}
\usepackage{dsfont}
\usepackage{bm}
\usepackage{color}
\usepackage{subfigure}
\usepackage{amsmath}
\usepackage{array}
\usepackage[all]{xy}

\usepackage{parskip}

\let\oldtocsection=\tocsection

\let\oldtocsubsection=\tocsubsection

\renewcommand{\tocsection}[2]{\hspace{0em}\oldtocsection{#1}{#2}}
\renewcommand{\tocsubsection}[2]{\hspace{1em}\oldtocsubsection{#1}{#2}}

\usepackage[backref,colorlinks=true,linkcolor=blue,citecolor=blue,urlcolor=blue,citebordercolor={0 0 1},urlbordercolor={0 0 1},linkbordercolor={0 0 1}]{hyperref} 

\newtheorem{thm}{Theorem}[section]

\newtheorem{theorem}[thm]{Theorem}

\newtheorem{proposition}[thm]{Proposition}

\newtheorem{lemma}[thm]{Lemma}

\newtheorem{corollary}[thm]{Corollary}

\newtheorem{definition}[thm]{Definition}

\newtheorem{remark}[thm]{Remark}

\newtheorem*{thm*}{Theorem}
\newtheorem*{cor*}{Corollary}
\newtheorem*{prop*}{Proposition}

\newcommand{\Z}{\mathbb{Z}}
\newcommand{\A}{\mathcal{A}}
\newcommand{\B}{\mathcal{B}}

\newcommand{\R}{\mathbb{R}}

\newcommand{\C}{\mathbb{C}}

\renewcommand{\P}{\mathbb{P}}

\newcommand{\Diff}{\mathit{Diff}}

\newcommand{\bdm}{\begin{displaymath}}
\newcommand{\edm}{\end{displaymath}}

\newcommand{\bq}{\begin{equation}}
\newcommand{\eq}{\end{equation}}

\numberwithin{equation}{section}

\title{Families of monotone Lagrangians in Brieskorn--Pham hypersurfaces}
\author{Ailsa Keating }

\begin{document}

\begin{abstract}
We present techniques, inspired by monodromy considerations, for constructing compact monotone Lagrangians in certain affine hypersurfaces, chiefly of  Brieskorn--Pham type. We focus on dimensions 2 and 3, though the constructions generalise to higher ones. 
The techniques give significant latitude in controlling the homology class, Maslov class and monotonicity constant of the Lagrangian, and a range of possible diffeomorphism types; they are also explicit enough to be amenable to calculations of pseudo-holomorphic curve invariants. 

Applications include infinite families of monotone Lagrangian $S^1 \times \Sigma_g$ in $\C^3$, distinguished by soft invariants for any genus $g \geq 2$; and, for fixed soft invariants, a range of infinite families of Lagrangians in Brieskorn--Pham hypersurfaces. These are generally distinct up to Hamiltonian isotopy. In specific cases, we also set up well-defined  counts of Maslov zero holomorphic annuli, which distinguish the Lagrangians up to compactly supported symplectomorphisms. Inter alia, these give families of exact monotone Lagrangian tori which are related neither by geometric mutation nor by compactly supported symplectomorphisms. 

\end{abstract}

\maketitle

\tableofcontents

\section{Introduction}

We present techniques  for constructing families of compact monotone Lagrangians, including exact ones, in affine varieties. Our prefered setting will be  Brieskorn--Pham hypersurfaces, i.e.~affine varieties of the form $\{ z_0^{a_0} + \ldots + z_m^{a_m} =1 \} \subset \C^{m+1}$, though  our constructions will carry over to e.g.~any affine variety which contains a suitable (truncated) Brieskorn--Pham hypersurface as a Stein submanifold. 
The article focuses on complex dimensions 2 and 3, though the techniques extend to give constructions in higher dimensions, which will briefly be discussed.

\subsection{Construction techniques}

Loosely speaking, the Lagrangians are built via an iterative process, increasing dimension one at a time. 
Suppose you start with a compact Lagrangian $L$ in, say, $X = \{ z_0^{a_0} + \ldots + z_{m-1}^{a_{m-1}} =1 \}$. Consider $Y = \{   z_0^{a_0} + \ldots + z_{m-1}^{a_{m-1}} + z_m^{a_m} =1 \}$, and a Lefschetz fibration $\pi: Y \to \C$, with smooth fibre $X$, given by $\epsilon(z_0, \ldots, z_{m-1}) + z_m$, where $\epsilon$ is a small generic linear deformation. We use properties of Dehn twists in vanishing cycles in $X$ (which follow from monodromy-type considerations) to construct in $Y$ embedded Lagrangians of the form $L \times S^1$, fibred over \emph{immersed} copies of $S^1$ in the base of $\pi$. We are also often able to get connected sums of such Lagrangians, via Polterovich surgery. The fact that the $S^1$ are immersed, rather than embedded, allows for considerable adaptability of the constructions; and the fact that the Lagrangians are (essentially) fibred over the base will make them tractable computationally.

\subsection{Soft invariants}
The techniques are explicit enough to allow us to control `soft' invariants such as the homology class of the Lagrangian, its Maslov class, and its monotonicity constant. Combined with results from geometric group theory, soft invariants give:

\begin{thm*}(Theorem \ref{thm:infinite_family_C3}.)
Fix $g \geq 2$, and any monotonicity constant $\kappa$. Then there exist infinitely many monotone Lagrangian $S^1 \times \Sigma_g$ in $\C^3$, distinct up to Lagrangian isotopy or symplectomorphism.
\end{thm*}

Experts may note that at least some of these can be arranged to have non-trivial Hamiltonian monodromy group, see Section \ref{sec:Hamiltonian_monodromy}. 

\subsubsection*{Discussion}
The diffeomorphism classes of closed orientable 3-manifolds admitting monotone Lagrangian embeddings into $\C^3$ are understood, by work of Evans and Kedra \cite[Theorem B]{Evans-Kedra} building on Fukaya \cite{Fukaya} and Damian \cite{Damian}: if $L$ is such a manifold, then $L$ is diffemorphic to  $S^1 \times \Sigma_g$, where $\Sigma_g$ is a surface of genus $g$; 
moreover, the $S^1$ factor in a monotone $S^1 \times \Sigma_g \subset \C^3$ must have Maslov index two.
At least one monotone embedding of $S^1 \times \Sigma_g$ exists for each $g$ (\cite[Proposition 12]{Evans-Kedra} and \cite[Corrigendum]{Evans-Kedra-corrigendum}).  Our constructions give examples with all possible (necessarily even) Maslov classes for $\Sigma_g$; by work of Waldhausen \cite{Waldhausen1, Waldhausen2, Waldhausen3} applied to $\text{Diff} (S^1 \times \Sigma_g)$, these classes are enough to distinguish the Lagrangians whenever $g \geq 2$. 

For the $g=1$ case, Auroux \cite{Auroux} showed that there are infinitely many distinct monotone Lagrangian tori in $\C^3$. This contrasts with the two- (and one-)dimensional case: it is widely expected that there are only two closed monotone Lagrangians in $\C^2$: the Clifford and Chekanov tori (see related results in \cite{DimitroglouRizell, DimitroglouRizell_Goodman_Ivrii}).  Auroux's tori all have the same `soft' invariants -- e.g.~they automatically all have Maslov class $(2,0,0) \in H^1(T^3; \Z) \cong \Z^3$ up to action by $SL(3,\Z)$; his proof uses counts of holomorphic discs with Maslov index two to tell them apart. $\Box$

In all our examples the count of Maslov index two discs will be zero. On the other hand, at least in the setting where the $a_i$ are sufficiently large, we can use calculations of other `hard' invariants to distinguish Lagrangians with the same soft invariants,  up to different possible types of equivalence depending on which hard invariant we use. 

\subsection{Lagrangian Floer theory}

Our convention will be that an `arbitrary' choice of Maslov class for a would-be Lagrangian $L$ is any class $\mu \in H^1(L;\Z)$ which satisfies the obvious restrictions imposed  by the topology of $L$: namely, that $\mu$ must pair to an even number with any class in $\pi_1(L)$ which preserves a choice of local orientation; and to an odd number with any class that reserves it. 
Suppose Lagrangians $L$ and $L'$ are diffeomorphic; 
we say that their Maslov classes 
are the same if they agree under any equivalence $H^1(L, \Z) \cong H^1(L', \Z)$ induced by a diffeomorphism from $L$ to $L'$. 

\subsubsection{Complex dimension two}

\begin{theorem} \label{thm:floer_2d} 
Let $\Sigma$ be a connected sum of tori and Klein bottles. 
If $r$ is  sufficiently large, then for any possible Maslov class and monotonicity constant, 
 we can construct  an infinite family of homologous monotone Lagrangian $ \Sigma$s in $\{ x^2 + y^4 + z^r = 1 \}$, distinct up to Hamiltonian isotopy, with that Maslov class and monotonicity constant.  

\end{theorem}

The proof uses the Lagrangian Floer cohomology of these Lagrangians with a reference Lagrangian sphere. In particular, the conclusions remain true under exact symplectic embeddings of  $\{ x^2 + y^4 + z^r= 1 \}$ (truncated e.g.~to have contact type boundary) into larger Liouville domains. Also, in many cases we can compute the Floer cohomology between two different members of one of the above infinite families; for suitable choices of rank one local systems, it can have arbitrarily large rank (Theorem \ref{thm:no_mutations} and Remark \ref{rmk:no_mutations}). This implies the following:

\begin{thm*} (see Theorem \ref{thm:no_mutations}.)
Consider one of the infinite families of monotone Lagrangian tori given by Theorem \ref{thm:floer_2d}. No two tori in this family can be related by a sequence of geometric mutations. 
\end{thm*}
This constrasts with most known constructions of interesting families of Lagrangian tori -- for further discussion, and more details on geometric mutation, see Section \ref{sec:no_mutations}. 

\subsubsection{Complex dimension three}

\begin{theorem}\label{thm:floer_3d}
Let $L = \#_{i=1}^l \left( S^1 \times ( \#_{g_i} T^2 \#_{g'_i} K ) \right) $, where $T^2$ is a torus and $K$ a Klein bottle, and $l, g_i, g'_i$ arbitrary. If $r$ and $s$ are sufficiently large, then for any possible Maslov class and monotonicity constant, 
 we can construct  an infinite family of homologous monotone Lagrangian $L$s in  $\{ x^2 + y^4 + z^r + w^s = 1 \}$, distinct up to Hamiltonian isotopy,  with that Maslov class and monotonicity constant.  
\end{theorem}

As before, the conclusion remains true under exact symplectic embeddings of (suitable large compact subsets of) $\{ x^2 + y^4 + z^r +w^s= 1 \}$ into larger Liouville domains. For calculations of Floer cohomology groups between members of a fixed infinite family, see Section \ref{sec:floer_3d_family}.

\subsubsection{Extensions}

Proceeding iteratively gives statements in higher dimensions;
see Proposition \ref{prop:floer_higher_dim}. 
We flag that Theorems \ref{thm:floer_2d} and \ref{thm:floer_3d} and Proposition \ref{prop:floer_higher_dim} give affine varieties with monotone Lagrangians, including tori, with arbitrarily high minimum Maslov number. (Experts may note that their homology class is primitive.) This contrasts with the case of $\C^m$, where it is known to be heavily restricted: for instance, Oh \cite{Oh} showed that if $L$ is a compact monotone embedded Lagrangian in $\C^m$, then $1 \leq N_L \leq m$, where $N_L$ is the minimal Maslov number of $L$; and Damian \cite{Damian} proved a number results for compact monotone Lagrangians $L$ in a monotone symplectic manifold $M$ such that every compact subset of $M$ is displaceable through a Hamiltonian isotopy (e.g.~$M = \C^m$): for instance, we always have $1 \leq N_L \leq m+1$, and, if $L$ is moreover aspherical, then $N_L \in \{  1, 2 \}$ (necessarily 2 in the orientable case -- for tori a number of further proofs are available \cite{FOOO, Buhovsky, Cieliebak-Mohnke, Irie}).

In some circumstances our techniques give monotone Lagrangians with different diffeomorphism types -- see Sections \ref{sec:further_diffeo_types} and \ref{sec:further_diffeo_higher_dim}. For instance, we get $\Sigma_g$ bundles over $S^1$ with non-trivial, finite order monodromy in $\{ z_0^{2} + z_1^{4} + z_2^{a_2}  + z_3^{a_3} =1 \}$ for sufficiently large $a_2$ and $a_3$, including in infinite families -- constrasting with the aforementioned constraints on the topology of compact orientable monotone Lagrangians in $\C^3$. 

\subsubsection*{Relation to other works} 

There are very interesting recent constructions by Oganesyan \cite{Oganesyan1, Oganesyan2} and by Oganesyan and Sun \cite{Oganesyan-Sun} of monotone Lagrangia submanifolds in $\C^n$, typically for large $n$, using ideas from toric geometry and building on work of Mironov \cite{Mironov}; with the exception of tori, this gives Lagrangians with different diffeomorphim types from the ones that are considered in this paper. (The reader may also be interested in Mikhalkin's constructions of Lagrangian submanifolds in symplectic toric varieties \cite{Mikhalkin}.)
In dimension two, there are interesting constructions of monotone Lagrangian embeddings of some non-orientable compact surfaces in $\C\P^2$ and $\C\P^1 \times \C\P^1$ in \cite{Abreu-Gadbled}. 
While in the final stages of writing this article the author learnt about the results of Casals and Gao \cite{Casals-Gao}; in particular, \cite[Corollary 1.8]{Casals-Gao} gives infinite families of smoothly isotopic exact higher genus Lagrangians surfaces  distinct up to Hamiltonian isotopy in Weinstein manifolds which are homotopic to $S^2$ (but do not contain Lagrangian spheres).

\subsection{Hard invariants: holomorphic annuli counts}
With string theoretic motivations in mind \cite{Hori-Vafa}, symplectic topologists have used counts of Maslov index two holomorphic discs with boundary on the monotone Lagrangian to tell such Lagrangians apart in a wide range of cases, starting e.g.~with \cite{Cho, Cho-Oh, Auroux}; this naturally distinguishes monotone Lagrangians up to symplectomorphism rather than merely Hamiltonian isotopy. On the other hand, this disc count vanishes if the Lagrangians are exact or have minimal Maslov number greater than two. Even if it is two, there are diffeomorphism types where one expects the count of Maslov two discs to be zero for topological reasons: it should follow from the arguments in e.g.~\cite{Fukaya, Irie} that Maslov 2 counts vanish on a Lagrangian $L$ whenever the fundamental class of $L$ is not in the image of the evaluation map from the homology of a non-trivial component of its free loop space; one then gets vanishing counts, for instance for $S^1 \times \Sigma_g$, from standard arguments about uniqueness of geodesics in spaces of negative curvature. 

Stepping back, holomorphic discs can be thought of as the simplest of open Gromov--Witten invariants; holomorphic annuli, the next simplest. (This also ties back to the physics perspective, viewing the former as a first-order invariant and the latter as a second-order one; for instance, in the setting of \cite{GLM} the former should correspond to hypermultiplets and the latter to vector multiplets, see e.g.~Section 4.1 therein.) This motivates us to consider
counts of Maslov zero holomorphic annuli; in dimensions 2 and 3 we give settings in which these are well-defined invariants, which moreover distinguish some of our Lagrangians:

\begin{thm*}(see Theorems \ref{thm:main_2d} and \ref{thm:main_3d})
\underline{Dim 2}: For any sufficiently large $r$, we can construct an infinite family of homologous monotone Lagrangian tori in $X_r= \{ x^3 + y^3 + z^r  =1 \}$, with fixed arbitrary Maslov class and monotonicity constant, distinct up to compactly supported symplectomorphisms of $X_r$. 

\underline{Dim 3}: Fix $g$. For any sufficiently large $r$ and $s$, we can construct an infinite family of homologous monotone Lagrangian $S^1 \times \Sigma_g$ in $Y_{r,s} = \{ x^3 + y^3 + z^r + w^s =1 \}$, with fixed arbitrary Maslov class and monotonicity constant, distinct up to compactly supported symplectomorphisms of $Y_{r,s}$.   

\end{thm*}

Note both of the above statements include the monotone exact case. Added motivation is that in complex dimension two, squares of Dehn twists act as the identity on homology. They are by now a well-established tool for producing infinite families of homologous Lagrangians none of which are Hamiltonian isotopic, following \cite{Seidel_knotted_spheres}. Given the wealth of Lagrangian spheres in Milnor fibres, including Brieskorn--Pham hypersurfaces, it is particularly relevant to be able to show that Lagrangians \emph{cannot} be related by Dehn twists. Combined with \ref{thm:no_mutations}, we get tori which are related neither by mutations nor by symplectomorphisms.

\subsubsection*{Technical discussion}  A trade-off is that these invariants don't behave well with respect to embeddings, as defining them requires us to compactify the ambient symplectic manifold. (On the other hand, a compactification or other modification is inevitable: by monotonicity the Maslov zero annuli would otherwise have to be constant.) 
For technical reasons we end up working with $x^3+y^3+z^r+w^s$ rather than  $x^2+y^4+z^r+w^s$; note $x^2+y^4$ is the singularity $A_3$ while $x^3+y^3$ is $D_4$. The theorem stated immediately above (and Theorems \ref{thm:main_2d} and \ref{thm:main_3d}) extends to Milnor fibres of a few other singularities, see Section \ref{sec:extensions_limitations}. 

Recall that if a moduli space of pseudo-holomorphic annuli is regular, then after quotienting out by reparametrisation its dimension is simply its Maslov index, by \cite[Theorem 1.2]{Liu_thesis}. However, in dimension greater than 3, disc bubbling a priori gets in the way of having a well-defined invariant.

Even in dimension at most 3, it can be difficult in general to extract well-defined invariants from Maslov zero holomorphic annuli counts: as well as bubbling, one needs to worry about the fact that the abstract moduli space of holomorphic annuli has boundary. 
In the present work, we carefully restrict ourselves to a geometric set-up in which the primary analytical difficulties can readily be ruled out. In particular, the `modulus infinity' boundary of the abstract moduli space is avoided by asking that both boundary curves lie in homologically non-trivial classes; and the `modulus zero' one by using a displacement of the Lagrangian off itself, so that the boundary curves formally lie on different components of a Lagrangian link -- this is why all of the Lagrangians in Theorems \ref{thm:main_2d} and \ref{thm:main_3d} have trivial cotangent bundle. Extra care is then taken to rule out holomorphic disc bubbling.
One could view our results as motivation for studying how to get invariants in more general cases, as has been done in 
\cite{Ekholm-Shende} in a similar setting.

\subsection*{Structure of the paper}
Section \ref{sec:preliminaries} gives background on Brieskorn--Pham hypersurfaces; in particular, in Section \ref{sec:rotations} we give explicit descriptions of their total monodromy. Section \ref{sec:building_blocks} gives constructions of monotone Lagrangians surfaces, and explains how to calculate their soft invariants. Section \ref{sec:C^3} is dedicated to monotone Lagrangians in $\C^3$, including the proof of Theorem \ref{thm:infinite_family_C3} stated above. A wider range of constructions of 3-dimensional monotone Lagrangians is given in Section \ref{sec:S^1xXi_g}, with the rest of Section \ref{sec:annuli} devoted to defining and evaluating counts of Maslov index zero annuli for these spaces, including a proof of Theorems \ref{thm:main_2d} and \ref{thm:main_3d} stated above. Floer--theoretic properties in dimensions 2 and 3, and consequences thereof, are given in Section \ref{sec:Floer_theory_properties}, including proofs of Theorems \ref{thm:floer_2d} and \ref{thm:floer_3d} . Finally, Section \ref{sec:higher_dim} briefly covers extensions to higher dimensions.

\subsection*{Acknowledgements} 
This project is inspired by earlier attempts, joint with Mohammed Abouzaid, to generalise Auroux' result \cite{Auroux} to Lagrangian $S^1 \times \Sigma_g$s in $\C^3$. In particular, many thanks to Abouzaid for suggesting using holomorphic annuli, and for discussions of technical difficulties which might arise when trying to get a well-defined invariant out of them. 

The author learnt Theorem \ref{thm:split_diffeo} from Henry Wilton, and Lemma \ref{lem:non_orientable_constraint} from Oscar Randal-Williams. 
Many thanks also to
Roger Casals for discussions of \cite{GLM} and feedback on earlier versions of this project; 
Georgios Dimitroglou-Rizell for discussions related to Remark \ref{rmk:Georgios}; 
 Tobias Ekhlom for discussions relating to counts of holomorphic annuli; 
Jonathan Evans for discussions relating to his joint work \cite{Evans-Kedra};
 Melissa Liu for explanations of her thesis \cite{Liu_thesis};
 and Ivan Smith for feedback on an earlier version of this article.

The author was partially supported by a Junior Fellowship from the Simons Foundation, NSF grant DMS--1505798, by NSF grant DMS--1128155
whilst at the Institute for Advanced Study, and by a Title A Fellowship from Trinity College, Cambridge.


\section{Preliminaries on Brieskorn-Pham hypersurfaces}\label{sec:preliminaries}

\subsection{Lefschetz (bi-)fibrations on  Brieskorn--Pham hypersurfaces}

Fix integers $a_0, \ldots$, $a_{m} \geq 1$. Let $X_{\mathbf{a}}$ be the hypersurface given by
$$
X_{\mathbf{a}}=\left\{
\sum_{i=0}^{m} z_i^{a_i} =1
\right\}. 
$$
This carries the structure of an exact symplectic manifold, inherited from $\C^{m+1}$. As the polynomial $\sum_{i=0}^{m} z_i^{a_i}$ is weighted homogeneous, its only singularity is at the origin. In particular, $X_\mathbf{a}$ is a representative, as an exact symplectic manifold, of the Milnor fibre of the singularity $ \sum_{i=0}^{m} z_i^{a_i}$, with half-infinite conical ends glued to its boundary.

Many of our explicit constructions will involve two families of hypersurfaces, for which we use dedicated notation, namely, for fixed integers $r, s \geq 1$:
\begin{align}
& X_r = \{ (x,y,z) \in \C^3 \, | \, x^3 + y^3 + z^r = 1 \} \\
& Y_{r,s} = \{ (x,y,z,w) \in \C^4 \, | \, x^3 + y^3 + z^r +w^s= 1 \}.
\end{align}

For notational convenience, let $a_m=b$, and label the hypersurface accordingly as $X_{\mathbf{a}, b}$. 
 Morsifying $\sum_{i=0}^{m-1} z_i^{a_i}$ 
and projecting to $z_m$ realises $X_{\mathbf{a}, b}$ (appropriately cut off) as the total space of a Lefschetz fibration 
$\Pi: X_{\mathbf{a}, b} \to \C$, with smooth fibre $X_{\mathbf{a}}=\{ \sum_{i=0}^{m-1} z_i^{a_i} =1 \}$. 

This has a built-in $\Z /b$ symmetry: multiplying $z_m$ by a $b$th root of unity gives a symplectomorphism of $X_{\mathbf{a}, b}$ which preserves fibres of $\Pi$, and induces an automorphism of its base $\C$ given by a rotation by $2 \pi /b$ in the origin. Let $\varpi = \prod_{i=0}^{m-1} (a_i-1) $. The fibration $\Pi$ has $b \varpi$ vanishing cycles in total: ordered clockwise with a $\Z /b$ symmetric choice of vanishing paths, $v_1, v_2, \ldots, v_\varpi, v_1, v_2, \ldots, v_\varpi, \ldots , v_1, v_2, \ldots, v_\varpi$ 
(see Figure \ref{fig:base_rotation}).
We will later use the observation that the base of the Lefschetz fibration on $X_{\mathbf{a}, kb}$ can naturally be divided into $k$ cyclically symmetric sectors, such that the total space of the restriction of $\Pi$ to each sector is $X_{\mathbf{a}, b}$. 

We will again want special notation for the cases we will make the most use of, as follows. 

In the case of $X_r$,  we call this fibration $\Pi_r: X_r \to \C$. It has  smooth fibre the thrice-punctured elliptic curve $\{ x^3 + y^3 =1 \}$, i.e.~the Milnor fibre of the two-variable $D_4$ singularity, which we will denote by $M$. For a suitable choice of vanishing paths, the critical points split into $r$ groups of size four, each giving the same vanishing cycles; these four cycles correspond to the `standard' $D_4$ configuration of vanishing cycles on $M$. See Figure \ref{fig:vcycles_on_Xr2}.

\begin{figure}[htb]
\begin{center}
 \includegraphics[scale=0.35]{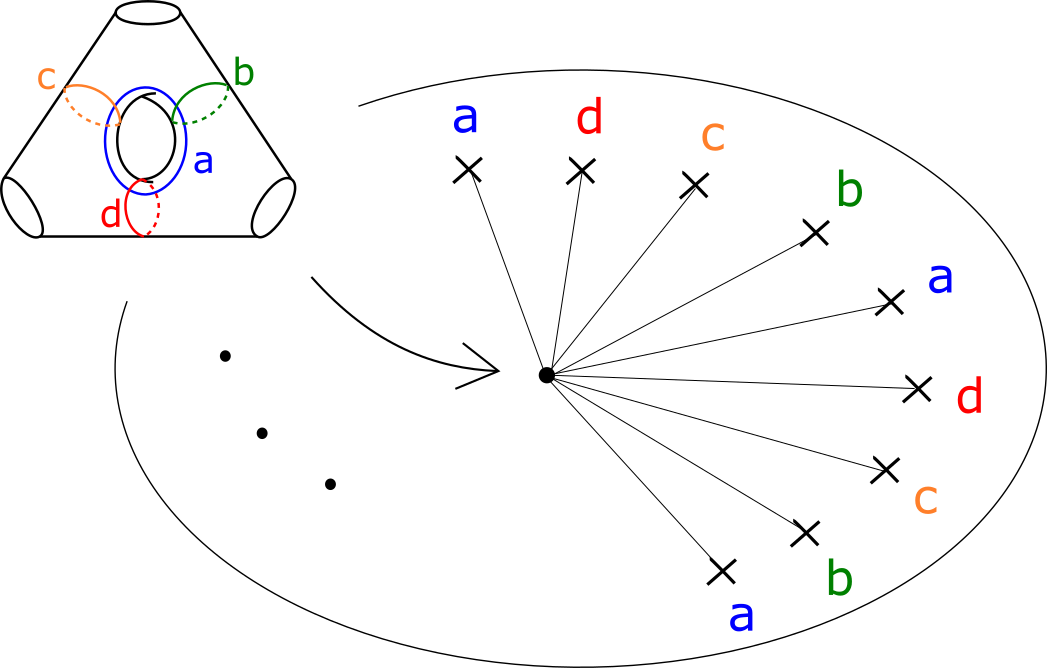}
\caption{Lefschetz fibration $\Pi_r$ on $X_r$. The marked points in the base are singular values; the segments joining them to the central point are vanishing paths, and the corresponding vanishing cycles are given in the fibre above the central smooth point. 
}
\label{fig:vcycles_on_Xr2}
\end{center}
\end{figure}

In the case of $Y_{r,s}$, we call this fibration $P_{r,s}: Y_{r,s} \to \C$. By construction, the smooth fibre is $X_r$. There is a total of $4(r-1)s$ critical points, and a $\Z/s$ cyclic symmetry, with vanishing cycles grouped into collections of size $4(r-1)$.  
We will sometimes take advantage of the fact that the smooth fibre of $P_{r,s}$ is $X_r$ to think of $Y_{r,s}$ as the total space of a bifibration $(P_{r,s}, \Pi_r)$.

\subsection{Deformations and parallel transport}\label{sec:deformations}

We'll make repeated use of the following well-known fact.

\begin{lemma}\label{lem:add_base_form}
There is a global isotopy of $X_{\mathbf{a}, b}$ which takes the standard K\"ahler form $\omega$ to $\omega+k \Pi^\ast \omega_0$, where $k \geq 0$ is a constant, and $\omega_0$ is the standard symplectic form on the base of $\Pi$. 
\end{lemma}

\begin{proof}
$\Pi$ is given by a polynomial map to $\C$ of the form $z_m + \epsilon(z_0, \ldots, z_{m-1})$, where $\epsilon$ is (for instance) a linear map with arbitrary small coefficients. 
For $t \in [0,k]$ the function $ t |\Pi(\mathbf{z})|^2+\sum_{i=0}^m |z_i|^2$ is a K\"ahler potential; this gives an interpolation of exact symplectic forms between $\omega$ and $\omega + k \pi^\ast \omega_0$. Let $\omega_t = \omega + k \pi^\ast \omega_0$, $t \in [0,k]$. 
Let's use the standard metric on $X_{\mathbf{a}, b}$ (induced by the one on $\C^{m+1}$) to estimate growth rates of differential forms and vector fields. The natural primitive to $\omega_t$, say $\theta_t$, grows linearly with the distance to the origin; on the other hand, $\omega_t$ is bounded (as, say, a family of bilinear forms applied to the unit sphere with respect to the metric). Thus the Moser vector field associated to the path $\omega_t$ and $\theta_t$ grows linearly with the distance to the origin. In particular, the vector field is integrable everywhere, which completes the proof. 
\end{proof}

Recall that on any Lefschetz fibration, there's a parallel transport map on smooth fibres, determined by taking the symplectic orthogonal of the tangent space of the fibre inside the tangent space of the total space. In order for this to be well-defined, technical care is required near the boundary of the fibres, which we haven't yet worried about beyond noting that the Milnor fibres of $\sum_{i=0}^m z_i^{a_i}$, with a half-infinite conical end attached, agrees with $X_\mathbf{a}$.

\begin{lemma}  \label{lem:fibration_J}

Consider $\Pi: X_{\mathbf{a}, b} \to \C$ as above, and let $\omega$ be the standard K\"ahler symplectic form on $X_{\mathbf{a}, b}$. Then for any  $s$, there exists $R$ such that for all $a \in D_s(0) \subset \C$, $\Pi^{-1}(a) \pitchfork S_{R'}(0)$ for all $R' \geq R$.
Moreover there exists $\tilde{R} > R$ and an exact symplectic form $\tilde{\omega}$ on $\pi^{-1}(D_s(0))$ such that 
\begin{itemize}
\item $\omega = \tilde{\omega}$ on $\pi^{-1}(D_s(0)) \cap B_R(0)$;
\item $\tilde{\omega}$ is a product outside of a sufficiently large bounded set, roughly points at distance greater than $\tilde{R}$ to zero. More precisely, there exists an open set $N \subset \Pi^{-1} (D_s(0))$ such that $(\Pi^{-1}D_s(0)) \backslash N$ is bounded, and parallel transport induces a symplectomorphism from $(N, \tilde{\omega})$ to 
$$
\Big( 
\big( \Pi^{-1}(0) 
\backslash
 \bar{B}_{\tilde{R}}(0) \big) \times D_s(0), 
\omega|_{ \Pi^{-1}(0)} \oplus \omega_0 
\Big)
$$
intertwining $\Pi$ and the projection map to $\C$, where $\omega_0$ is a symplectic form on $D_s(0) \subset \C$ compatible with the standard complex structure. Denote by $\rho$ the pullback of the distance to zero function on $\Pi^{-1}(0)$.

\item $\tilde{\omega}$ agrees with $\omega$ when restricted to any fibre of $\Pi$.

\item The form $\tilde{\omega}_k = \tilde{\omega} + k \Pi^\ast \omega_0$, for $k$ a non-negative constant, is also a symplectic form. Moreover, for all sufficiently large $k$, there is an $\tilde{\omega}_k$--compatible almost-complex structure on $\Pi^{-1}(D_r(0))$, say $\tilde{J}_k$, agreeing with a product  for $\rho > \tilde{R}+1$, with the standard $J$ for $\rho < \tilde{R}$, such that $\tilde{J}_k$ preserves vertical tangent spaces, and such that
$\Pi$ is $(\tilde{J}_k, J_0)$--holomorphic, where $J_0$ is the standard complex structure on $D_r(0)$. 
\end{itemize}

Moreover, for any compact set $L$ in the domain of definition of $\omega_k$, there is a Moser isotopy $\iota$ such that $\iota^\ast \omega = \tilde{\omega}_k$ on $L$.

\end{lemma}

\begin{proof}
The statement in the opening paragraph (existence of $R$ so as to ensure transversality of $\Pi^{-1}(a)$ and $S_{R'}(0)$ for all $a, R'$ as described) goes back to Milnor \cite[Corollary 2.8]{Milnor_singularities}. 
The point of the rest is to have a common technical set-up for a Lefschetz fibration, often knows as a `trivial horizontal boundary' -- see e.g. \cite[Section 15]{Seidel_book}. To obtain this set-up, one can proceed as follows.

Given $a \in D_s(0)$, consider the symplectic parallel transport along a straight segment from $a$ to $0$, say $\psi_a$.
Note that away from critical points of $\Pi$, parallel transport is always well defined, as similar considerations to the proof of Lemma \ref{lem:add_base_form} show that the relevant vector field is integrable (indeed, it has at worst polynomial growth). 
In particular,  outside of $\Pi^{-1} (a) \cap B_{R'}(0)$, where $R'$ may be quite a bit larger than $R$, 
$\psi_a$ is always defined, and 
a symplectomorphism onto its image. 
(In fact, we see that for a path outside of $B_r(0)$ with $r$ large enough such that it contains all the critical values, parallel transport along that path is defined on the entire (non-truncated) fibre.)
 Wlog assume that $R'$ works for all $a \in D_s(0)$. Now fix $\tilde{R}$ such that $\Pi^{-1}(0) \backslash B_{\tilde{R}}(0)$ is contained in $\psi_a \big( \Pi^{-1}(a) \backslash B_{R'}(0) \big)$ for all $a \in D_s(0)$.  

Let $$N = \bigcup_{a \in D_s(0)} \psi_a^{-1} \big( \Pi^{-1}(0) \backslash B_{\tilde{R}}(0) \big),$$ and define $\Psi$ by:
\begin{eqnarray}
\Psi: N  & \to  & \big( \Pi^{-1}(0) \backslash B_{\tilde{R}}(0) \big) \times D_s(0) \\
p  &\mapsto & (\psi_{\Pi(p)} (p), \Pi(p)).
\end{eqnarray}
The map $\Psi$ is a diffeomorphism by construction. We will use $F_{a} = \Pi^{-1}(a) \cap N$ to denote a (subset of a) fibre.

Let's now construct $\tilde{\omega}$. 
Let $r, \theta$  be radial and angle coordinates on the base  $\C$, and $x_1, \ldots, x_{2m-2}$ be local coordinates on the central fibre $F_0$; together these pull back to local coordinates on $N$ via $\Psi$. Observe that there is a decomposition
$$
\omega = \omega_F + \omega_B + \omega_{\epsilon}
$$
where $\omega_F = \sum f_{ij} dx_i \wedge dx_j$, for some functions $f_{ij}$, and should be thought of as a fibre term; $\omega_B = h r dr \wedge d\theta$, for some positive function $h$, and should be thought of as a base term; and $\omega_{\epsilon} = \sum \epsilon_i dx_i \wedge d \theta$ consists of mixed fibre / base terms. 
By construction, $\omega_F = \Psi^\ast (\omega|_{\pi^{-1}(0)})$; in particular, the $f_{ij}$ are independent of $r$ and $\theta$, and $\omega_F$ is exact. Moreover, note that there are no terms of the form $dx_i \wedge dr$ in $\omega$. (This is because we are using parallel transport in the $r$ direction to define $\Psi$.) Further, by varying over coordinate charts for $F_0$, we can patch forms together to get $\omega_F$, $\omega_B$ and $\omega_{\epsilon}$ globally defined on $N$.

Say $\omega - \omega_F = d \alpha$, some one-form $\alpha$. By construction, $d \alpha = \beta \wedge d \theta$, for $\beta$ a one-form. Now $d \beta \wedge d \theta = 0 $, so, for fixed $\theta$, we can integrate $\beta$ on $K_\theta = \{ F_a \, | \, \text{arg}(a) = \theta \}$ to $b_\theta \in C^\infty (K_\theta)$ with $d b_\theta = \beta \in \Omega^1 (K_\theta)$. As $K_\theta$ is connected, $b_\theta$ is uniquely defined up to a constant. Now one can choose constants so that $b_\theta$ varies smoothly with $\theta \in S^1$, and extends over $r=0$ by the zero constant. Let $b \in C^{\infty} (N)$ be the resulting function, which is now uniquely determined; by construction we can choose $\alpha$ to be $b d \theta$.

As defined before, let $\rho$ be the distance to zero function on $\Pi^{-1}(0)$, pulled back via $\Psi$ to a function on $N$.
Let $\eta = \eta(\rho)$ be a smooth non-decreasing cut-off function on $[0, \infty)$ which is identically zero for $\rho \leq \tilde{R}$, and identically one for $\rho \geq \tilde{R}+1$.

Let $\omega_k = \Psi^{\ast}(\omega|_{\Pi^{-1}(0)}) \oplus k \Pi^{\ast} (\omega_0) \in \Omega^2 (N)$, where $\omega_0$ is the standard symplectic form on the base, and $k$ a positive constant. By construction, $\omega_k$ is a symplectic form on $N$, and agrees with $\omega$ when restricted to fibres.

Consider the exact two-form
$$
\tilde{\omega}_k =  \omega_F + d (\eta r^2 /2 d \theta ) + d ( (1-\eta) \alpha) + k \Pi^\ast (\omega_0)
$$
where $k \geq 0$ is a constant. As $d \eta$ only involves terms of the form $dx_i$, we see that 
\begin{equation}\label{eq:tildeomega_k}
\tilde{\omega}_k^{m} = m \omega_F^{m-1} \wedge \left(  \eta r d r \wedge d \theta +(1- \eta) d b \wedge d \theta + k \Pi^\ast (\omega_0)  \right)
\end{equation}
which is positive everywhere. Thus $\tilde{\omega}_k$ is symplectic. By construction, it's equal to $\omega+k \Pi^\ast (\omega_0)$ for $\rho \leq \tilde{R}$, and $\Psi^\ast (\omega|_{\Pi^{-1}(0)}) + (k+1) \Pi^\ast (\omega_0)$ for $\rho \geq \tilde{R}+1$; moreover, it agrees with $\omega$ when restricted to fibres. 
This is true of $\tilde{\omega}_k$ for all $k \geq 0$. In particular, $\tilde{\omega} = \tilde{\omega}_0$ works for the statement of the lemma.

Let us next check the claim about a Moser isotopy between $\omega$ and $\tilde{\omega}_k$.
A similar calculation to Equation \ref{eq:tildeomega_k} shows that the linear interpolation between $\tilde{\omega}_k$ and  $\omega$ symplectic for all time; in fact, the same would be true for $\tilde{\omega}_k$ and  $\omega_l = \omega + l \Pi^\ast \omega_0$. Consider the Moser vector field associated with the obvious choice of primitives for these. Its grow at worst polynomially with distance to the origin. Further, if $l \gg k$, the vector field points inwards along $\Pi^{-1}\{ |z|=s \}$; thus for any $L \subset \Pi^{-1}(D_s(0))$, we can integrate the Moser vector field to get an isotopy $\iota: L \to X_{\mathbf{a}, b}$ such that $\iota^\ast \omega_l = \tilde{\omega}_k$. The claim then follows from Lemma \ref{lem:add_base_form}.

In order to establish the final point, about almost-complex structures, we now want to find a suitable one  for any sufficiently large $k$, say $\tilde{J}_k$, compatible with $\tilde{\omega}_k$.

For $p \in N$, consider a basis for $T_pN$ given by 
taking a basis for $T_pF_{\Pi(p)}$ followed by one for its symplectic orthogonal $T_pF_{\Pi(p)}^{(\tilde{\omega}_k, \perp)}$. 
With respect to this basis, we want $\tilde{J}_k$ to be of the form
$$
\mathcal{J} = 
\begin{pmatrix}
 & & & \ast \\
 & J|_{T_{\phi_\eta (p)}F_{\Pi( \phi_\eta(p))}} & & \ast \\
  & & & \vdots \\
 & & & \ast \\
0 &\ldots & 0 & J_0 \\
\end{pmatrix}
$$
where $\phi_t$, some fixed $t \in [0,1]$, denotes the parallel transport along a straight line segment from $F_a$ to $F_{(1-t)a}$ for any $a \in D_s(0)$, and  the entries $\ast$ are to be determined. Note that any such matrix $\mathcal{J}$ satisfies $\mathcal{J}^2 = -I$ (irrespective of the values $\ast$); moreover, by construction it satisfies $D\pi \circ \mathcal{J} = J_0 \circ D \pi$ and $\mathcal{J} (TF_a) = TF_a$. 

Now notice that the condition $\tilde{\omega}_k (u, \tilde{J}_k v) = \tilde{\omega}_k (v, \tilde{J}_k u)$, for all $u,v \in T_pM$, uniquely determines each of the entries $\ast$. 
For $\rho \leq \tilde{R}$, we have arranged to have $\tilde{\omega}_k = \omega + k \Pi^\ast (\omega_0)$. Note that the standard $J$ is of the form $\mathcal{J}$, and is compatible with both $\omega$ and $\omega + k \Pi^\ast (\omega_0)$. This implies that $\tilde{J}_k = J$ for $\rho \leq \tilde{R}$. Moreover, for $\rho \geq \tilde{R}+ 1$, $\tilde{\omega}_k$ is a product, and it follows that $\tilde{J}_k$ is too (with all of the entries $\ast$ vanishing). Finally, notice that for any sufficiently large $k$, we have $\tilde{\omega}_k (u, \tilde{J}_k u) >0 $ for $u \neq 0$. This completes the proof. 
\end{proof}

\begin{remark}\label{rmk:bifibred_J}
Applying the preceeding lemma iteratively, we can arrange to have symplectic forms and almost complex structures which are standard on a large compact set (up to a positive pullback of the base symplectic form), and fibred with respect to the bifibration $(P_{r,s}, \Pi_r)$ outside a slightly larger compact set. 
\end{remark}


\subsection{Fractional boundary twists for Brieskorn-Pham Milnor fibres}\label{sec:rotations}

Recall that
$$
X_{\mathbf{a},b}=\left\{
\sum_{i=0}^{m-1} z_i^{a_i} + z_m^b =1
\right\}. 
$$
and that the Lefschetz fibration $\Pi:  X_{\mathbf{a}, b} \to \C$ is given by morsifying $\sum_{i=0}^{m-1} z_i^{a_i}$ 
and projecting to $z_m$. The smooth fibre is $X_{\mathbf{a}}=\{ \sum_{i=0}^{m-1} z_i^{a_i} =1 \}$. 
Recall that we let $\varpi = \prod_{i=0}^{m-1} (a_i-1) $, which means that the fibration $\Pi$ has $b \varpi$ vanishing cycles in total -- ordered clockwise, $v_1, v_2, \ldots, v_\varpi, v_1, v_2, \ldots, v_\varpi, \ldots , v_1, v_2, \ldots, v_\varpi$, as in Figure \ref{fig:base_rotation}.

We start by recalling a result about the total monodromy of these Milnor fibres.

\begin{lemma}
Let $\nu = \tau_{v_1} \tau_{v_2} \ldots \tau_{v_{\varpi}} \in \text{Symp}^c(X_{\mathbf{a}})$ be the total monodromy of $z_0^{a_0} + \ldots + z_{m-1}^{a_{m-1}}$. Then $\nu^{\text{lcm}(a_0, \ldots, a_{m-1})}$ is compactly Hamiltonian isotopic to a boundary Dehn twist on $X_{\mathbf{a}} \cap B_R(0)$,  say $\varrho$, defined using a periodic Reeb flow ($R$ just needs to be sufficiently large). In particular, $\varrho$ has support in a collar neighbourhood of the boundary of  $X_{\mathbf{a}} \cap B_R(0)$, which wlog is disjoint from all of $v_1, v_2, \ldots, v_{\varpi}$. Thus $\varrho$ commutes with each of $\tau_{v_1}, \ldots, \tau_{v_{\varpi}}$. 
\end{lemma}

\begin{proof}
The singularity $\sum_{i=0}^{m-1} z_i^{m_i}$ is weighted homogeneous with weight $\text{lcm}(a_0, \ldots, a_{m-1})$. Thus we get a periodic Reeb flow on the boundary of $\{z_0^{a_0} + \ldots + z_{m-1}^{a_{m-1}} = 1\} \cap B_R(0)$ (contactomorphic to the boundary of $\{ z_0^{a_0} + \ldots + z_{m-1}^{a_{m-1}} = 0\} \cap B_R(0)$), and $\nu^{\text{lcm}(a_0, \ldots, a_{m-1})}$ is Hamiltonian isotopic to the boundary Dehn twist which it induces -- see  the discussion in \cite[Section 4c]{Seidel00}.
\end{proof}

Assume for the rest of this section that $b$ is a multiple of $\text{lcm}(a_0, \ldots, a_{m-1})$, say $b = k \cdot \text{lcm}(a_0, \ldots, a_{m-1})$. Pick $r$ such that all of the critical values of $\Pi$ lie in $B_r(0)$. Define a relative mapping class $[\varphi] \in \pi_0 \Diff^{cpt} (\C, \text{Crit}(\Pi))$, fixing the $b\varpi$ critical values of $\Pi$ setwise,  by a counterclockwise (i.e.~positive) rotation by $2 \pi / k$ on $B_r(0)$, the identity outside $B_{r+1}(0)$, and a smoothing of the linear interpolation between the rotation and the identity on the annulus between the two.

Fix a representative $\varphi$ of $[\varphi]$ which is a symplectomorphism of the base. As an element of $\text{Symp}^{cpt} (\C)$ (forgetting the marked points), $\varphi$ is Hamiltonian isotopic to the identity. Fix such an isotopy, say $\{ \varphi_t \}_{t \in [0,1]}$, with $\varphi_t = \varphi$ on a neighbourhood of zero and $\varphi_t = \text{Id}$ on a neighbourhood of 1. For simplicity we take $\varphi_t$ to be rotationally symmetric, as an element of $\text{Symp}^{cpt}(\C)$, for all $t$. 

Following Lemma \ref{lem:fibration_J}, we fix a constant $s$ such that $D_s(0)$ contains all of the critical values of $\Pi: X_{\mathbf{a}, b} \to \C$, and the supports of the $\varphi_t$. Given $s$, we also fix constants $R$, $\tilde{R}$, and symplectic forms $\tilde{\omega}_k$, $k \geq 0$ ($\tilde{\omega}_0 = \tilde{\omega}$) as in Lemma \ref{lem:fibration_J}, and a large set $K$ lying over ${D}_s(0)$ and compact in the vertical direction, such that there is a symplectomorphism $\Psi: \Pi^{-1}({D}_s(0)) \backslash K \to \big( \Pi^{-1}(0) \backslash B_{R}(0) \big) \times D_s(0)$ such that the product symplectic form pulls back to $\tilde{\omega}$. (With the notation of Lemma \ref{lem:fibration_J}, $\Pi^{-1}({D}_s(0)) \backslash K = N$.)  There is a Moser isotopy such that $\omega$ pulls back to $\tilde{\omega}_k$ on $K$. Loosely speaking, $K$ is  capturing all of the topology of $X_{\mathbf{a}, b}$. Also, we will use the fact that symplectic parallel transport of fibres of $\Pi$ with respect to $\tilde{\omega}$ is flat for $\tilde{R} \leq \rho \leq \tilde{R}+1$ (as $\tilde{\omega}$ is a product).

\begin{proposition}\label{prop:Phi_definition}
The map $\varphi$ induces a compactly supported symplectomorphism $\Phi$ of $X_{\mathbf{a}, b}$ which, up to a compactly supported Moser isotopy, has the following properties:

\begin{itemize}
\item $\Phi$ is the identity away from $\Pi^{-1}(D_s(0))$, and on the set identified with 
$$\Pi^{-1}(0) \cap \{ \tilde{R}+2 \leq  \rho  \} \times D_s(0).$$ In particular, $\Phi$ has compact support. 

\item The following diagram commutes:
\begin{equation}\label{eq:Phi}
\xymatrix{
K \ar[r]^\Phi \ar[d]_\Pi & K  \ar[d]^\Pi \\
\C \ar[r]_\varphi & \C
}
\end{equation}

\item On the set identified with  $\big( \Pi^{-1}(0) \cap \{ \tilde{R}+1 \leq  \rho \leq \tilde{R}+2 \} \big) \times D_s(0)$, we have that
\begin{equation}\label{eq:isotope_back}
\Phi (y,a) = (y, \varphi_{\rho - (\tilde{R}- 1)} a).
\end{equation}

\end{itemize}

The map $\Phi$ is uniquely defined up to compactly supported Hamiltonian isotopy. 

\end{proposition}

\begin{remark}\label{rmk:fibred_symplecto}
Even though $\Phi$ is only fibred with respect to $\Pi$ over a large compact set, we will sometimes refer to it, somewhat abusively, as a `fibred symplectomorphism'. 
\end{remark}

\begin{proof}
Without loss of generality we  work with $\tilde{\omega}$ rather than $\omega$.
Fix a smooth reference point $\star \in D_s(0)$ away from the support of the $\varphi_t$. Fix a system of smooth paths from $\star$ to all other points of $D_s(0)$. We can arrange for these to vary smoothly away from a choice of cuts, which we can take to run cyclically between absolutely ordered critical values (first to second, second to third, etc -- but not from the final one back to the first), and from the final critical value all the way out of $D_s(0)$. Moreover, we choose them so that they foliate the complement of the cuts. See Figure \ref{fig:cut_choices}.  We also arrange for the paths between $\star$ and points outside the support of the $\varphi_t$ not to enter the support of the $\varphi_t$. 

\begin{figure}[htb]
\begin{center}
\includegraphics[scale=0.30]{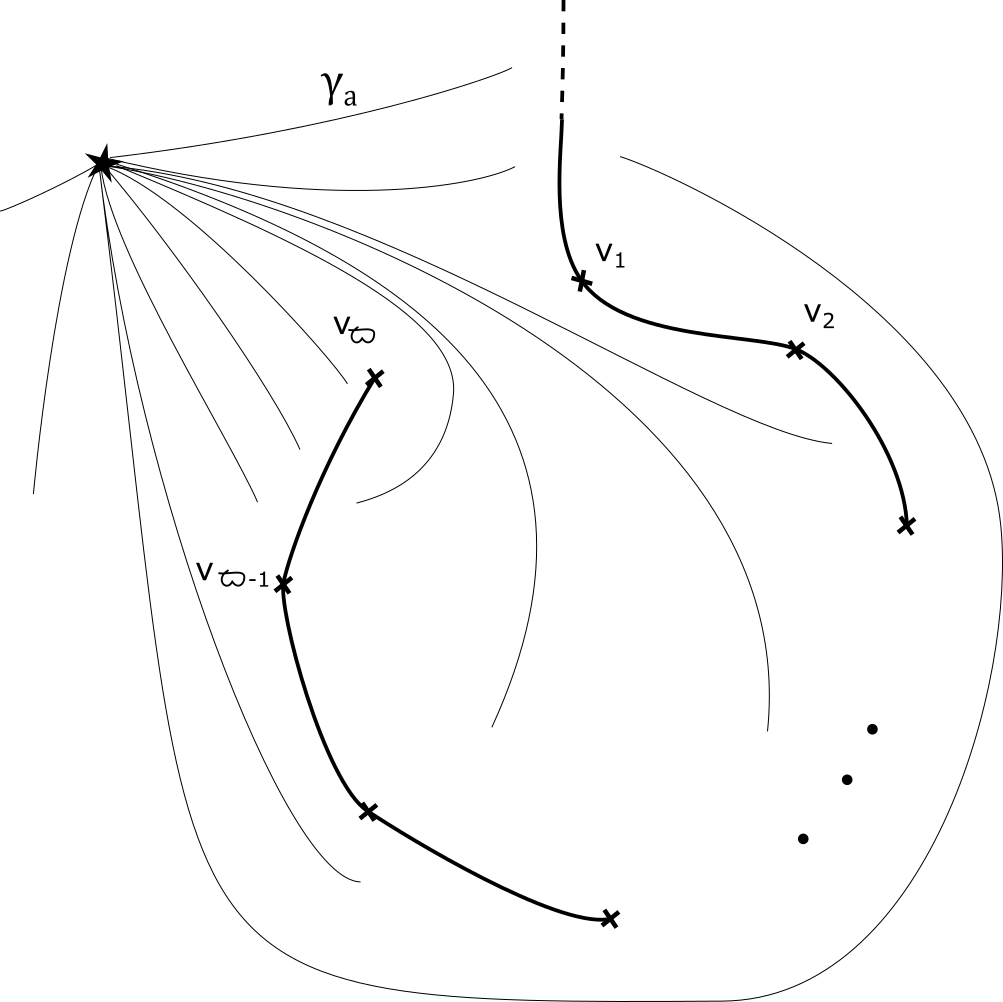}
\caption{Base of $\Pi: X_{\mathbf{a},b} \to \C$, with its $b \varpi$ critical values, and our choice of cuts (thicker segments), and smooth system of paths (thinner segments) from $\star$ to points $z \in \C$, say $\gamma_z$, used in the proof of Proposition \ref{prop:Phi_definition}.
}
\label{fig:cut_choices}
\end{center}
\end{figure}

Let $\gamma_a$ be the paths from $\star$ to $a$, and let $F_a = \Pi^{-1}(a) \cap K$.  

By assumption, on $K$, if $\Phi$ is defined it must be given by a collection of maps $\Phi_z: F_z \to F_{\varphi(z)}$. 
Moreover, we require that $\Phi_z = Id$ for $z \notin D_{s+1}(0)$. 
Given a path $\gamma: [0,1] \to \B \backslash \text{Crit}(\Pi)$ with $\gamma(0)= \star$, $\gamma(1) = z$, let $\sigma_\gamma: F_\star \to F_z$ be the symplectic parallel transport map induced by $\gamma$. 
Now notice that \emph{if} the map $\Phi$ is defined, then we must have that $\sigma_{\varphi(\gamma)}^{-1} \circ \Phi_z \circ \sigma_{\gamma}$ is isotopic to $\Phi_\star$, i.e.~the identity $\text{Id}$, so $\Phi_z$ is isotopic to $\sigma_{\phi(\gamma)} \circ \sigma_{\gamma}^{-1}$. In particular, if $\Phi$ is well-defined, this should be independent of the choice of $\gamma$. Let us check that this is the case for our system of paths: we will show that for each of the paths $\gamma_1, \ldots, \gamma_{b \varpi}$ on Figure \ref{fig:base_rotation}, from $\star$ to $0$, the symplectomorphism $\Gamma_i = \rho_{\varphi(\gamma_i)} \circ \rho_{\gamma_i}^{-1}$ is independent of $i = 1, \ldots,b \varpi$ up to Hamiltonian isotopy. (This will be enough to construct a well-defined symplectomorphism, which in turn will imply the statement for all paths.)

\begin{figure}[htb]
\begin{center}
\includegraphics[scale=0.30]{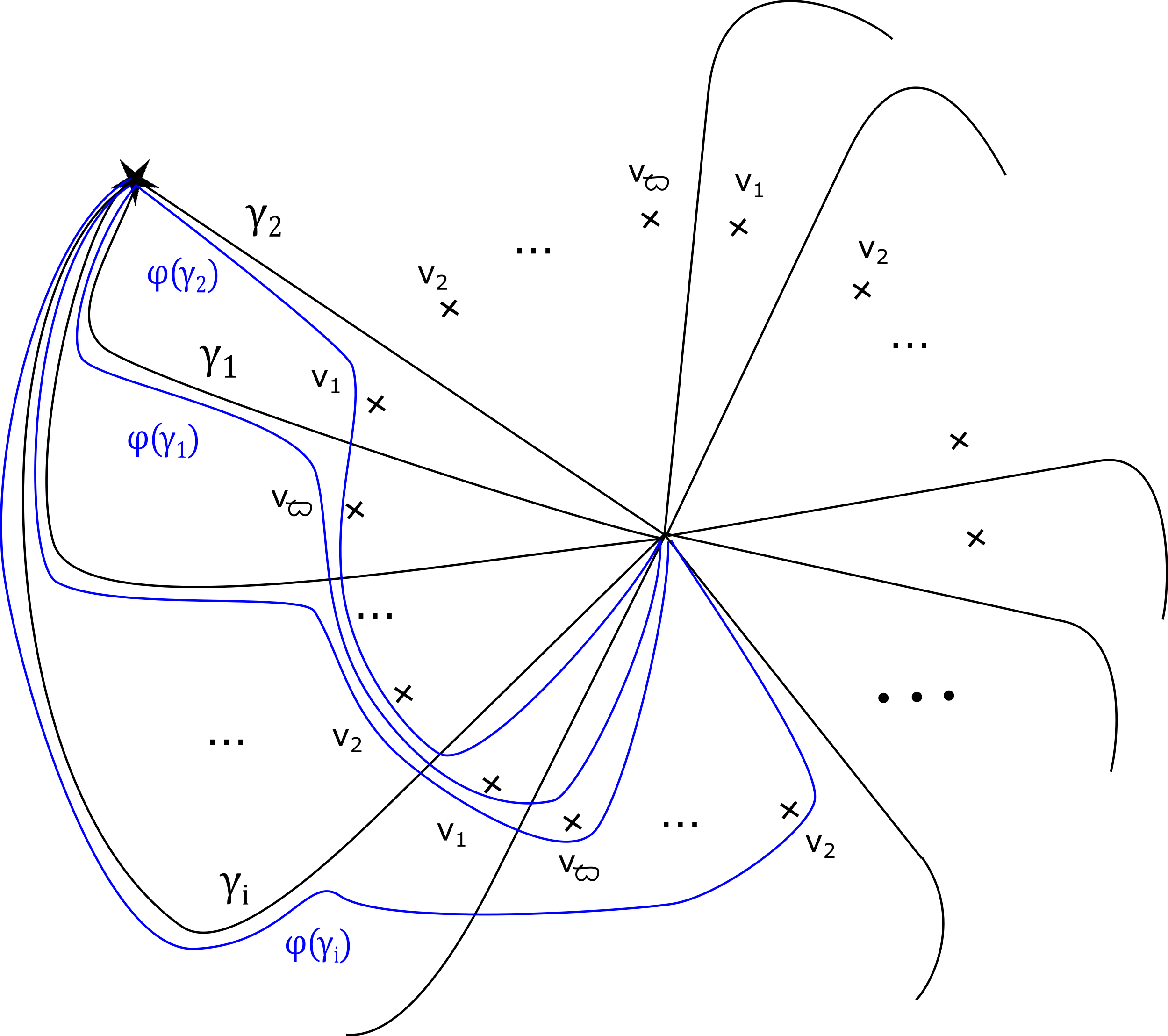}
\caption{Base of $\Pi: X_{\mathbf{a}, b} \to \C$, with a total of $b \varpi$ critical values, together with the paths $\gamma_1, \ldots, \gamma_{b \varpi}$ and their images under $\varphi$, as used in the proof of Proposition \ref{prop:Phi_definition}.}
\label{fig:base_rotation}
\end{center}
\end{figure} 

We can read off the monodromy factorisation of each $\Gamma_i$ directly from Figure \ref{fig:base_rotation}:
$$
\Gamma_i = \big(\tau_{v_i} \tau_{v_{i+1}} \ldots, \tau_{v_{i+\varpi-1}} \big)^{\text{lcm}(a_0 , \ldots, a_{m-1})} = \varrho
$$
where indices are taken modulo $\varpi$, and we are using the fact that $\varrho$ commutes with each of the $\tau_{v_i}$ for the second equality. Thus $\Gamma_i$ is independent of $i$.

We now use this to define a fibred map $\tilde{\Phi}: K \to K$ lifting $\varphi$ such that:
\begin{itemize}
\item $\tilde{\Phi}$ intertwines $\Pi$, and is given by fibrewise symplectomorphisms $\tilde{\Phi}_z$ (including for critical fibres, away from the critical points);

\item Away from a thickening of the cuts, $\tilde{\Phi}_z = \sigma_{\phi(\gamma_z)} \circ \sigma^{-1}_{\gamma_z}$;

\item For a point $a$ on a cut (but not a critical point), there are two choices, say $\tilde{\Phi}_a^l$ and $\tilde{\Phi}_a^r$; we know these to be Hamiltonian isotopic, so pick a Hamiltonian isotopy from one to the other and realise this by symplectomorphisms along a segment across  the thickened cut (in particular, this is a smooth family of symplectomorphisms with two parameters: the distance traveled across the thickened cut, and $a$, where $a$ varies between the two consecutive critical values on the cut). 

\item $\tilde{\Phi}_z$ is the identity outside of the support of the $\varphi_t$.

\end{itemize}

In order to extend this over the critical fibres, we need to be careful with our choices of Hamiltonian isotopies over cuts (as in general there is no reason to expect the space of all Hamiltonian maps on a fibre to be simply connected). A `hands on' way of resolving this in this particular case is as follows:

Using our system of paths, the map on the central fibre is given by $\Gamma_1$. Fix a Hamiltonian isotopy $h_t$, $t \in [0,1]$, from $\Gamma_1$ to $\varrho$.

Now notice that if we used instead a path going across the first cut (the one between critical points of type $v_\varpi$ and $v_{\varpi-1}$), the map on the central fibre would be given by $\Gamma_\varpi$. This means that picking a Hamiltonian isotopy to use to define $\tilde{\Phi}_z$ over the first cut amounts to picking a Hamiltonian isotopy between $\Gamma_\varpi$ and $\Gamma_1$. As the monodromy about the critical point $v_{\varpi}$ is $\tau_{v_\varpi}$, choosing the following Hamiltonian isotopy  allows us to extend the maps $\tilde{\Phi}_z$ over that first critical point:
$$
G_{1,t} = 
\begin{cases}
 \tau_{v_{\varpi}}h_{2t} \tau_{v_{\varpi}}^{-1} & \text{for } t \in [0,1/2] 
\\
h_{2-2t} & \text{for } t \in [1/2,1] 
\end{cases}
$$
(This is well-defined at $t = 1/2$ as the $\tau_{v_j}$ have support disjoint from that of our model for $\nu^{\text{lcm}(a_0, \ldots, a_{m-1})}$, i.e.~$\rho$.) In order to extend over the next critical point (of type $v_{\varpi-1}$), we need to pick a suitable Hamiltonian isotopy between $\Gamma_{\varpi-1}$ and $\Gamma_1$, for instance:
$$
G_{2,t} = 
\begin{cases}
 \tau_{v_{\varpi-1}} \tau_{v_{\varpi}}h_{2t} \tau_{v_{\varpi}}^{-1}  \tau_{v_{\varpi-1}}^{-1}  & \text{for } t \in [0,1/2] 
\\
h_{2-2t} & \text{for } t \in [1/2,1] 
\end{cases}
$$
Proceeding iteratively, for the final cut (which stretches off out of the support of $\varphi$), we use the Hamiltonian isotopy between $\Gamma_{1}$ and $\Gamma_1$ given by 
$$
G_{b \varpi,t} = 
\begin{cases}
 (\tau_{v_{1}} \ldots \tau_{v_{\varpi}})^{b} h_{2t} (\tau_{v_{1}} \ldots \tau_{v_{\varpi}})^{-b} & \text{for } t \in [0,1/2] 
\\
h_{2-2t} & \text{for } t \in [1/2,1] 
\end{cases}
$$
Now use the fact that $b$ is a multiple of $\text{lcm}(a_0, \ldots, a_{m-1})$ to deform $G_{b \varpi, t}$ (rel.~endpoints) to the constant isotopy. This allows us to ensure that $\tilde{\Phi}_z$ is the identity outside of the support of the $\varphi_t$.

As $\tilde{\omega}$ is a product over $D_s(0)$ for $\tilde{R} \leq \rho \leq \tilde{R}+1$, $\tilde{\Phi}_z$ is given by $(p,z) \mapsto (p, \varphi(z))$ on that region (using the identification $\Psi$). This means that we're free to extend the collection $\tilde{\Phi}_z$ to a compactly supported diffeomorphism $\tilde{\Phi}$ of the total space simply by requiring that equation \ref{eq:isotope_back} hold (without isotopy) for $\tilde{\Phi}$. 

Recall $\tilde{\omega}_k = \tilde{\omega} + k \Pi^\ast \omega_0$. 
Consider the family of closed two-forms, for $\lambda \in [0,1]$ and $k \geq 0$;
$$
\tilde{\omega}_{k,\lambda} = \lambda (\tilde{\omega}_k) + (1-\lambda) \tilde{\Phi}^\ast ( \tilde{\omega}_k )
$$
We claim that for sufficiently large $k$, $\tilde{\omega}_{k,\lambda}$ is a symplectic form for any $\lambda$. 
The claim follows from noticing the following:

\begin{itemize}

\item For $\rho \leq \tilde{R}+1$, $\tilde{\Phi}^\ast \Pi^\ast \omega_0 = \Pi^\ast \phi^\ast \omega_0 = \Pi^\ast \omega_0$. On the other hand, $\tilde{\Phi}$ preserves the restriction of $\tilde{\omega}$ (or equally of $\tilde{\omega}$) to each fibre. It follows that for sufficiently large $k$, $\omega_{k, \lambda}$ is certainly symplectic on this region.

\item For $ \tilde{R}+1 \leq \rho \leq \tilde{R}+2$,  $\tilde{\Phi}^\ast \tilde{\omega}_{k} - \tilde{ \omega}_{k}$ is of the form 
$d \rho \wedge  d r$, where $r$ is a radial coordinate on the base $\C$. 
(Here we use the assumption that the $\varphi_t$ are rotationally invariant.) Now note that this term doesn't contribute to $\tilde{\omega}_{k,\lambda}^m$. 

\item For $\tilde{R}+2 \leq \rho$, as $\tilde{\Phi} = \text{Id}$, $\tilde{\omega}_{k,\lambda} = \tilde{\omega}_k$ for all $\lambda$. 

\end{itemize}

This then implies that we can perform a compactly supported Moser isotopy to deform $\tilde{\Phi}$ to a symplectomorphism with respect to the symplectic form $\tilde{\omega} + k \Pi^\ast \omega_0$, for sufficiently large $k$. (The support of the isotopy is contained in that of $\tilde{\Phi}$, so we needed worry about issues of compactness / being able to integrate the Moser vector field.) 
We can now conjugate this with a Moser isotopy between $\tilde{\omega}_k$ and $\omega$ to get a symplectomorphism with respect to the original $\omega$.
\end{proof}

From the discussion in \cite[Section 2.5]{Keating_tori}, we know that the matching cycles 
$$V^1_1, V^1_2, \ldots, V^1_{\varpi}, V^2_1, V^2_2, \ldots, V^2_{\varpi}, \ldots,  V^{b-1}_1, \ldots, V^{b-1}_{\varpi}$$
as given by Figure \ref{fig:distinguished_matching_cycles}, are a distinguished collection of vanishing cycles in $X_{\mathbf{a},b}$ considered as the Milnor fibre of the singularity $z_0^{a_0} + \ldots + z_{m-1}^{a_{m-1}} + z_m^b$.

\begin{figure}[htb]
\begin{center}
\includegraphics[scale=0.65]{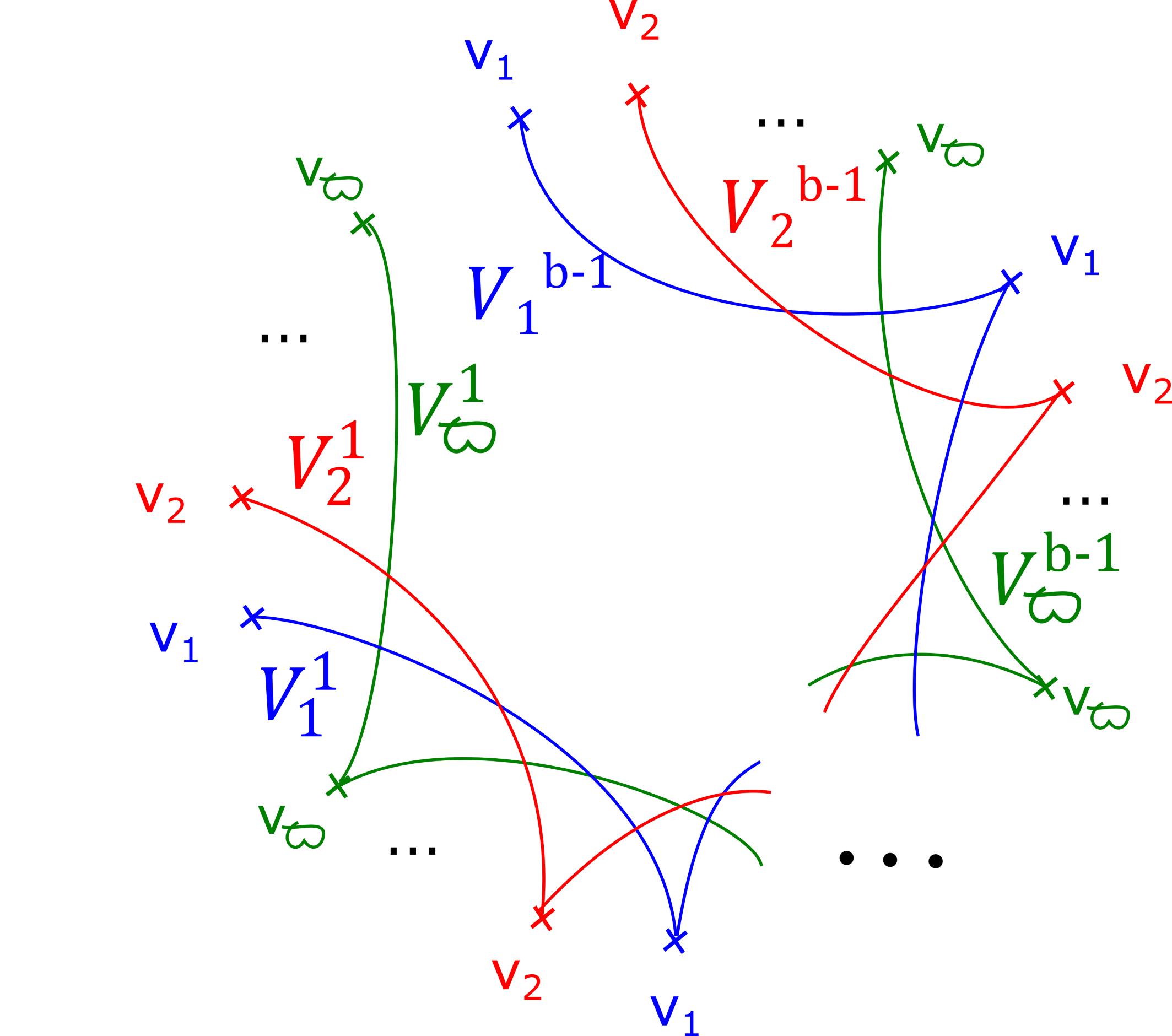}
\caption{The matching cycles $V^1_1, \ldots, V^1_{\varpi},  \ldots,  V^{b-1}_1, \ldots, V^{b-1}_{\varpi}$.}
\label{fig:distinguished_matching_cycles}
\end{center}
\end{figure} 

\begin{proposition}\label{prop:Phi_Dehntwists}
Up to compactly supported Hamiltonian isotopy, the following symplectomorphisms are equal:
$$
\Phi = \Big( \tau_{V^{1}_1}  \tau_{V^{1}_2} \ldots \tau_{V^{1}_{\varpi}}
 \tau_{V^{2}_1}  \tau_{V^{2}_2} \ldots \tau_{V^{2}_{\varpi}}
 \ldots  \tau_{V^{b-1}_1}  \tau_{V^{b-1}_2} \ldots
 \tau_{V^{b-1}_{\varpi}}
\Big)^{\text{lcm}(a_0, \ldots, a_{m-1})}.
$$
\end{proposition}

\begin{proof}
Let 
$$\tau = \tau_{V^{1}_1}  \tau_{V^{1}_2} \ldots \tau_{V^{1}_{\zeta}} \tau_{V^{2}_1}  \tau_{V^{2}_2} \ldots \tau_{V^{2}_{\varpi}} \ldots  \tau_{V^{b-1}_1}  \tau_{V^{b-1}_2} \ldots \tau_{V^{b-1}_{\varpi}}.$$
The map $\tau$ is presented as a fibred symplectomorphism, in the sense of Remark \ref{rmk:fibred_symplecto}.
See Figure \ref{fig:map_tau}. It is the monodromy of the singularity $z_0^{a_0} + \ldots + z_{m-1}^{a_{m-1}} $. Note that after a compactly supported Hamiltonian isotopy (induced by one of the base relative to the critical points), we can take $\tau$ to be $\Z / b$--symmetric; the action on the central fibre (which is fixed set-wise) is precisely $\nu$, the monodromy of the singularity $z_0^{a_0} + \ldots + z_{m-1}^{a_{m-1}}$, i.e.~$\tau_{v_1} \tau_{v_2} \ldots \tau_{v_{\varpi}} \in \text{Symp}_{\text{cpt}} X_{\mathbf{a}}$. In particular, $\tau^{\text{lcm}(a_0, \ldots a_{m-1})}$ and $\Phi$ agree on the central fibre. 

\begin{figure}[htb]
\begin{center}
\includegraphics[scale=0.30]{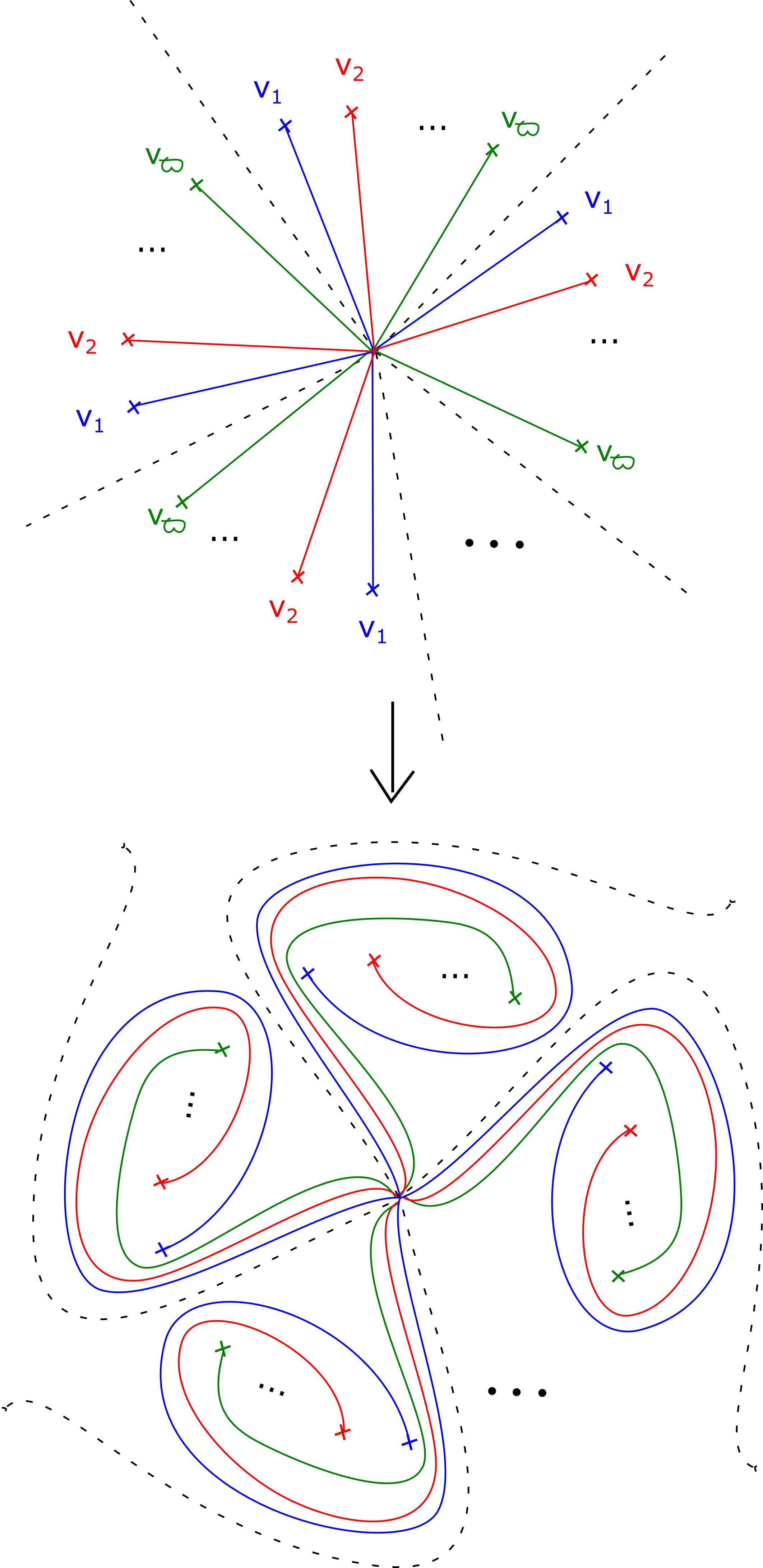}
\caption{The automorphism of the base of $\Pi$ which describes the map $\tau$.}
\label{fig:map_tau}
\end{center}
\end{figure} 

Now consider the fibred symplectomorphism $\Phi^{-1} \circ \tau^{\text{lcm}(a_0, \ldots a_{m-1})}$.
 After Hamiltonian isotopy, this is given by Figure \ref{fig:map_tau^lcm}; in particular, the half-lines $\gamma_1, \ldots, \gamma_b$ (relabelled compared with Figure \ref{fig:base_rotation}) are preserved point-wise, as are the fibres above them. 
Thus $\Phi^{-1} \circ \tau^{\text{lcm}(a_0, \ldots a_{m-1})}$ can be decomposed as a composition of cyclically symmetric, compactly supported symplectomorphisms with disjoint support, each contained in a `sector' between $\gamma_i$ and $\gamma_{i+1}$. Let us focus on one such sector, say between $\gamma_1$ and $\gamma_2$; let $\Upsilon$ be the compactly supported symplectomorphism of the sector given by restricting $\Phi^{-1} \circ \tau^{\text{lcm}(a_0, \ldots a_{m-1})}$ . Notice that the total space of that (sub) Lefschetz fibration is $\{ z_0^{a_0} + \ldots + z_{m-1}^{a_{m-1}} + z_m =1 \} \cong \C^m$, with the map to $\C$ simply given by morsifying $z_0^{a_0} + \ldots + z_{m-1}^{a_{m-1}}$ and projecting to $z_m$. 
It now follows that $\Upsilon$ is Hamiltonian isotopic to the identity as a compactly supported symplectomorphism of $\C^m$, albeit not as a fibred map with respect to the Lefschetz fibration: smoothly `turn off' the Morsification of $z_0^{a_0} + \ldots + z_{m-1}^{a_{m-1}}$ to get a single critical point (the projection is now given by just mapping to $z_m$), and it's now immediate that the `twisting' which defines $\Upsilon$ can be unravelled, inducing a Hamiltonian isotopy to the identity. As this can be done in each sector, the conclusion follows.
\end{proof}

\begin{figure}[htb]
\begin{center}
\includegraphics[scale=0.40]{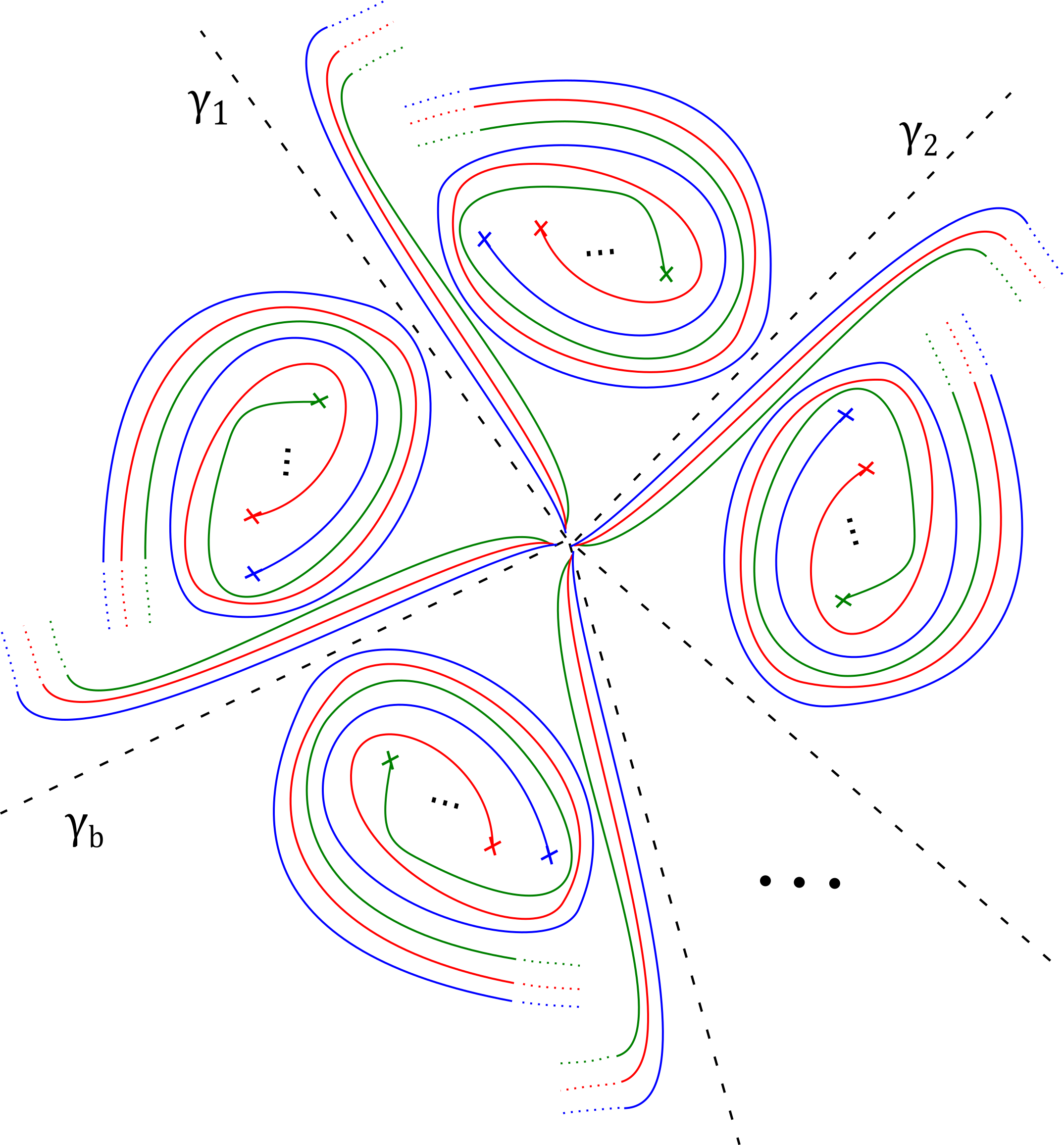}
\caption{The automorphism of the base of $\Pi$ which describes the map $\Phi^{-1} \circ \tau^{\text{lcm}(a_0, \ldots, a_{m-1})}$. (Note the indices on the $\gamma_i$ have been renamed for simplicity.) Each vanishing path in the base rotates about the relevant cluster of vanishing cycles $\text{lcm}(a_0, \ldots, a_{m-1})$ times.}
\label{fig:map_tau^lcm}
\end{center}
\end{figure} 

\begin{remark}

In the final step of the proof above, in the case $m=2$ one could also appeal to Gromov's theorem \cite{Gromov} that any compactly supported symplectomorphism of $\C^2$ is Hamiltonian isotopic to the identity.

\end{remark}

\begin{remark}
Proposition \ref{prop:Phi_Dehntwists} implies that  the symplectomorphism $\Phi^k$, which corresponds to a $2 \pi$ rotation of the base, is Hamiltonian isotopic to $\nu^{k \circ \text{lcm}(a_0, \ldots, a_{m-1})} = \nu^b$. On the other hand, by assumption, $b = \text{lcm}(a_0, \ldots, a_{m-1},b)$, so this agrees with the boundary Dehn twist constructed using the periodic Reeb blow in \cite[Section 4c]{Seidel00}. 
\end{remark}

\begin{remark} \label{rmk:weighted_homogeneous}
For simplicity, we've chosen to restrict ourselves to the case of Brieskorn--Pham singularities. However, the discussion e.g.~in \cite[Section 4c]{Seidel00} applies more broadly to weighted homogeneous singularities. In particular, this means that the constructions (and conclusions) of this section apply more broadly to any singularity of the form $f(z_0, \ldots, z_{m-1}) + z_m^{b}$, where $f$ is a weighted homogeneous singularity. 
\end{remark}


\subsection{Monotone Lagrangians and Maslov indices}\label{sec:Maslov_background} \

Let $L \subset (X^{2m}, \omega)$ be a closed Lagrangian submanifold of a symplectic manifold. Let $LGr(2m)$ be the Grassmanian of Lagrangian $m$ planes in $\R^{2m}$. Recall $\pi_1 (LGr(2m)) = \pi_1 (U(m) / O(m)) = \Z$. An element $\beta \in \pi_2 (X, L)$ induces a trivialization of $T X |_{\partial \beta}$, and a class $\left[ TL|_{\partial \beta} \right] \in \pi_1 (LGr(2n))$. This is called the \emph{Maslov index} of $\beta$, denoted $\mu(\beta)$. (This should not be confused with the total monodromy of a Lefschetz fibration, also conventionally denoted $\mu$, though we have avoided that notation in this article.) 
Whenever $L$ is orientable, $\mu(\beta) \in 2 \Z$. 
If $2c_1(X) = 0$, we can fix a trivialisation of $(\Lambda^mT^*X)^{\otimes 2} \simeq \C \times X$; given a Lagrangian $L$, this determines the homotopy class of a map $L \to \C^\ast$; the induced class in $H^1(L, \R)$ is called the \emph{Maslov class} of $L$. Dually, via the standard identification $\text{Hom}(H_1(L; \Z), \Z) \cong H^1 (L, \R)$, each class in $H_1(L; \Z)$ has a Maslov index (in particular, the Maslov index of a disc only depends on its boundary).

\begin{definition}
$L$ is monotone if there exists $\kappa >0$ such that for all $\beta \in \pi_2(X,L)$,
\bq 
[ \omega ] (\beta ) = \kappa \cdot \mu (\beta).
\eq \label{eq:monotone}
\end{definition}

If $L$ is a Lagrangian in $X_r$, respectively $Y_{r,s}$, we have
\begin{align}
& \pi_2 (X_r, L) \cong \pi_2(X_r) \oplus \pi_1 (L) \cong \Z^{4(r-1)} \oplus \pi_1 (L) \\
& \pi_2 (Y_{r,s}, L) \cong \pi_1(L)
\end{align}
where $4(r-1)$ is the Milnor number of the singularity $f_r$. 
On the $\pi_1 (L)$ term, both maps from $\pi_2(X,L)$ to $\R$ in Equation \ref{eq:monotone} factor through $H_1(L)$. 
Moreover, all classes in the image of $\pi_2(X)$, which are represented by Lagrangian spheres, have symplectic area zero. In particular, the symplectic area of any class in $\pi_2(X,L)$ is determined by the homology  class of its boundary in $H_1(L)$. 

We briefly review some relevant background concerning Maslov indices as well as tools for computations which appear later. 
Recall that given any hypersurface singularity $f$ in $m+1$ variables, its Milnor fibre $Z_f$ has trivial tangent bundle: $[Z_f , BU(m) ] = [ \vee_{i=1, \ldots  \text{Mil} (f)} S_i, BU(m) ] $, where the $S_i$ are vanishing cycles and $\text{Mil}(f)$ is the Milnor number of $f$; now use the fact that for each $S_i$, $T(T^\ast S_i)$ is a trivial $U(m)$ bundle.

In the $m=1$ case, using $[Z_f , U(1) ] = [ \vee_{i=1, \ldots  \text{Mil} (f)} S_i, U(1) ] $, we see that  a choice of trivialisation of $TZ_f$ is determined up to homotopy by the Maslov classes of all of the vanishing cycles. 

Now fix trivialisations of $TX_r$ and $TY_{r,s}$ as $U(2)$ and $U(3)$--bundles, say $\sigma_r: TX_r \cong X_r \times \C^2$ and $\sigma_{r,s}: TY_{r,s} \cong Y_{r,s} \times \C^3$. Let $L$ be a Lagrangian in $X_r$ or $Y_{r,s}$. Using the trivialisations, any path $\rho: S^1 \to L$ induces a path $\tilde{\rho}: S^1 \to LGr (\R^{2m})$, $m=2,3$; the class $\tilde{\rho} (S^1) \in \pi_1 (LGr(\R^{2m}))$ is the Maslov index of $\rho$.  
Note that this class is independent of the choice of $\sigma_r$ and $\sigma_{r,s}$: $\sigma_r$, the trivialisation of $TX_r$, is essentially  unique, as the difference between two trivialisations is given by a class in 
$$
[ X_r, U(2) ] = [ \vee_{i=1,\ldots , 4(r-1)} S^2, U(2) ] = \{1 \}.
$$
On the other hand, $\sigma_{r,s}$, the trivialisation of $TY_{r,s}$, is not unique, as it determined up to a class in 
$$
[ Y_{r,s}, U(3) ] = [ \vee_{i=1,\ldots , 4(r-1)(s-r)} S^3, U(3) ] = \Z^{4(r-1)(s-r)},
$$ 
where $4(r-1)(s-1)$ appears as the Milnor number of the singularity $f_{r,s}$ -- but as $S^3$ is simply connected this doesn't affect Maslov indices.

Let $Z$ denote either $X_r$ or $Y_{r,s}$, and let $m=2$ or $3$ be its dimension. To calculate Maslov indices, we fix a `reference' Lagrangian plane $\R^m \subset \C^n$, and pull it back to Lagrangian planes $\mathcal{L}_p \subset TZ$ at each point $p \in Z$ via $\sigma_r$ or $\sigma_{r,s}$. 
Given an oriented path $\{ TL_p \}_{p \in \rho(S^1)}$, we count the non-transverse intersections with $\mathcal{L}_p$. The Maslov index $\mu(\rho)$ is given by the sum of the signed dimensions of these non-transverve intersections. 

For $X_r$, we can use the following. 
Suppose we are given a disc $B \subset \C$ in the base of $\Pi_r$, containing no critical points. Then we may assume that the trivialization of $TX_r$ restricts to the product of trivialisations of the fibre $M$ and the base $B$, and take reference Lagrangian planes given by the product of two reference Lagrangian lines, one in the tangent bundle to the fibre and one in the tangent bundle to the base. For the base, we pick e.g.~a constant horizontal line. A trivialisation of $TM$ is determined up to homotopy by the Maslov indices of  the curves $a, b, c, d$ (labelled as before following Figure \ref{fig:vcycles_on_Xr2}); we pick one where these are all zero; for instance, we may take our reference Lagrangian lines to be as in Figure \ref{fig:reference_Lagrangian_line2}.  (In a suitable identification with a thrice-punctured square with sides glued in pairs, these tangent lines all have slope one.)

\begin{figure}[htb]
\begin{center}
\includegraphics[scale=0.35]{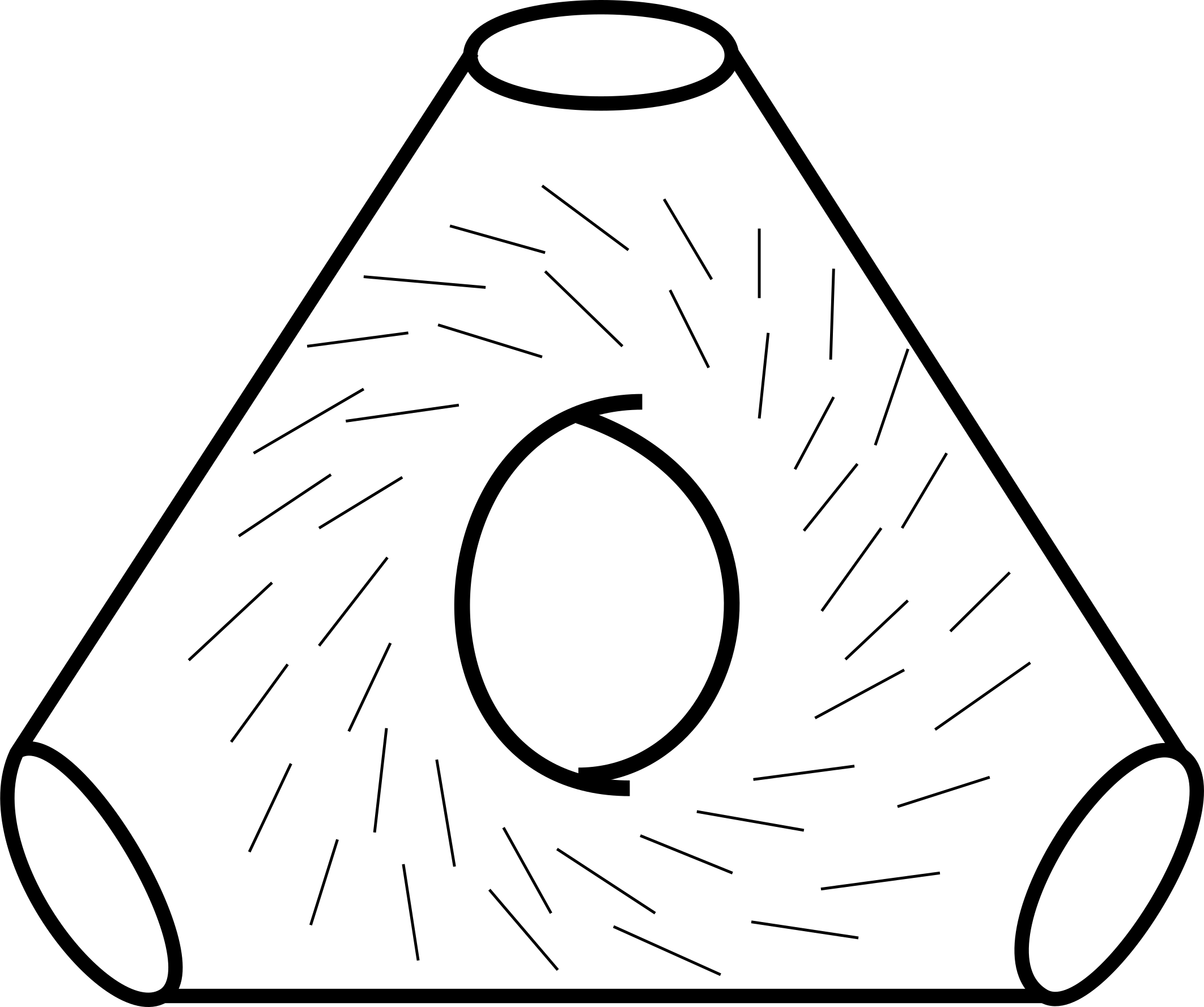}
\caption{Reference Lagrangian lines in $T\Sigma$. 
}
\label{fig:reference_Lagrangian_line2}
\end{center}
\end{figure}

The trivialisation of $TM$ is invariant in Dehn twists in $a, \ldots, d$ up to isotopy, so our choices extend to give a trivialisation of $T\Pi^{-1}_r (\C \backslash D_\epsilon ( \text{Crit}(\Pi_r)) )$, where $D_\epsilon ( \text{Crit}(\Pi_r))$ is a small neighbourhood of the critical values of $\Pi_r$. Further, as we have chosen the curves $a, \ldots, d$ to have Maslov index zero, our choice of trivialisation can be extended (up to homotopy) over the critical fibres (recall that $a,b,c$ and $d$ are also the vanishing cycles for $\Pi_r$).

For $TY_{r,s}$, we proceed similarly: using $P_{r,s}$, start with the product of our trivialisation of $TX_r$ with the obvious trivialisation of $TB$ for a disc $B \subset \C \backslash \text{Crit}(P_{r,s})$, and notice that it can be extended to the whole space.  

We will later use the following observation.
\begin{lemma}
Suppose $\Sigma_g \subset X_r$ is an orientable Lagrangian submanifold. Fix a decomposition of $\Sigma_g$ (as an abstract surface) as a connect sum of $g$ tori, say $T_1, \ldots, T_g$. Then there are bases of $H_1 (T_j; \Z)$, for each $j=1, \ldots, g$, such that for the basis of $H_1 (\Sigma_g, \Z)$ induced by the natural isomorphism 
\bq
H_1 (\Sigma_g ; \Z) = H_1 (T_1; \Z) \oplus \ldots \oplus H_1 (T_g; \Z)
\eq
the Maslov class of $\Sigma_g$  is equal to
\bq
(2n_1, 0, 2n_2, 0, \ldots, 2n_g, 0) \in Hom(H_1 (\Sigma_g;\Z); \Z)
\eq
for some integers $n_i \in \Z_{\geq 0}$. 
\end{lemma}
 
\begin{proof}
As $\Sigma_g$ is orientable, all entries for the Maslov class with respect to any basis are even, and wlog non-negative. In the case $\Sigma_g = T^2$, a class of the form $(2\alpha, 2\beta)$ with respect to some basis can be transformed to $(2\text{gcd}\{\alpha,\beta\}, 0)$ with respect to another one. The claim is then immediate.
\end{proof}


\section{Building blocks: monotone Lagrangian surfaces}\label{sec:building_blocks}

\subsection{Distinguished monodromy actions}

We will study Lagrangian tori and Klein bottles which  are fibred over immersed $S^1$s in the base of $\Pi_r$. As a preliminary, we calculate the images of certain curves under parallel transport along some distinguished arcs in the base of $\Pi_r$, and fix some notation.

We will care about six configurations, each involving four critical points. Two of them are given in Figures \ref{fig:bc2} and \ref{fig:bb2}; we will refer to these two configurations as being of types $BC$ and $BB$ respectively. Configurations $CC$, $BD$, $CD$ and $DD$ are defined similarly. (This shorthand will also be used on diagrams later when describing more complicated configurations, to aid legibility.) As the vanishing cycles $b$ and $c$ do not intersect, one could swap their order in Figure \ref{fig:bc2}; in particular, a configuration of type $CB$ would be the same as $BC$.

\begin{figure}[htb]
\begin{center}
\includegraphics[scale=0.40]{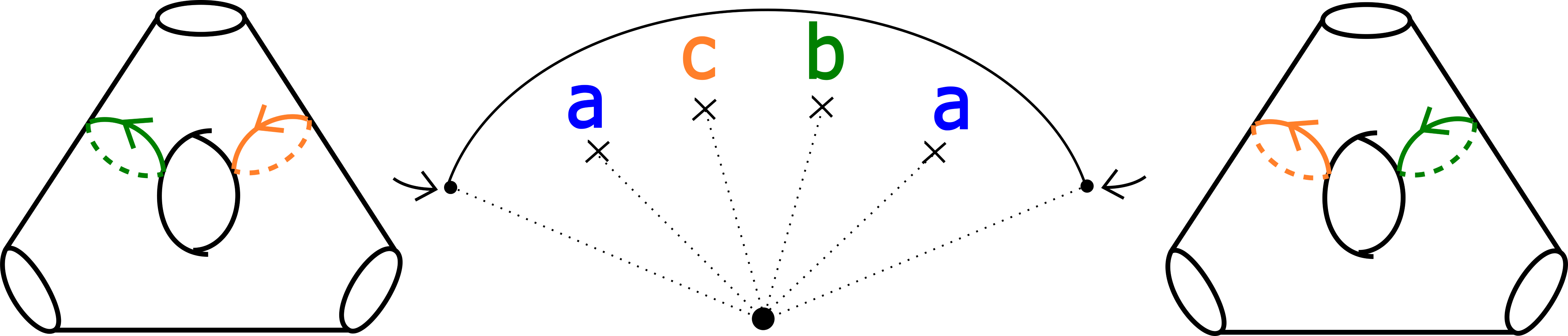}
\caption{Configuration of type $BC$. With our choices of comparison paths to the reference central fibre (the dotted segments), the oriented $b$ and $c$ curves get swapped under parallel transport along the given arc.
}
\label{fig:bc2}
\end{center}
\end{figure}

\begin{figure}[htb]
\begin{center}
\includegraphics[scale=0.40]{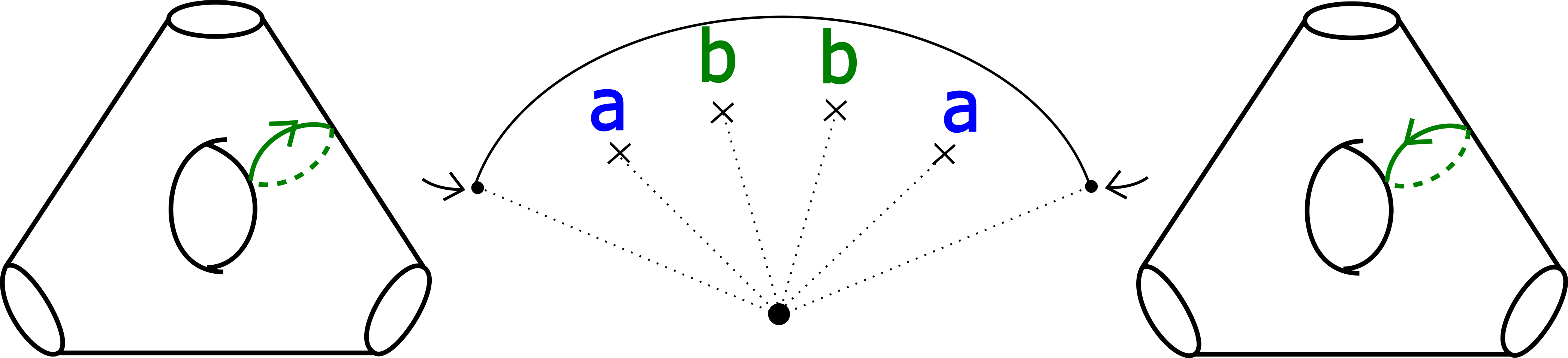}
\caption{Configuration of type $BB$. With our choices of comparison paths to the reference fibre, the curve $b$ is fixed setwise, and its orientation reversed.
}
\label{fig:bb2}
\end{center}
\end{figure}

\begin{remark}
Each of the configurations just involves an $A_3$ chain of vanishing cycles; in particular, they could be defined for different fibres, the simplest of which would be a twice-punctured elliptic curve, i.e. the fibre of the two-variable $A_3$ singularity $x^2+y^4$. 
\end{remark}

\subsection{The Lagrangian tori $T_{k,l,m}$}

Consider an immersed loop $\gamma \subset \C \backslash \text{Crit}(\Pi_r)$. Suppose we're given an exact Lagrangian $S^1$ in the fibre above a point $\star$ of $\gamma$, say $V \subset M_\star$. (In all the cases we will consider, we just take $V$ to be a vanishing cycle for $\Pi_r$.) We will be interested in special cases in which the image of $V$ under the total monodromy along $\gamma$ will happen to be Hamiltonian isotopic to $V$ itself; in such cases, taking the union of the images of $V$ under parallel transport along $\gamma$ yields a Lagrangian in the total space $X_r$, which, a priori, is immersed. If monodromy preserves the orientation of $V$, it is a torus, and otherwise, a Klein bottle; call this hypothetical Lagrangian $L$. (We will give a range of possible explicit constructions of such an $L$ further down in this section.)

In very special cases, we can arrange for $L$ to be embedded, as follows. Assume wlog that $\gamma$ is immersed with transverse double intersection points. We will give examples of curves $\gamma$ with the property that for each of their intersection points, the images under parallel transport of $V$ on each of the two segments do \emph{not} intersect in the fibre above the intersection point. To construct such curves, we will exploit the fact that the $b$ and $c$ curves do no intersect, together with the fact that a $BC$ `move' trades them (and similarly with $b$ and $d$, and $c$ and $d$).

Two simple examples are given in Figure \ref{fig:simple_tori2}.  By greedily counting elements needed and using Figure \ref{fig:vcycles_on_Xr2}, 
the left-hand one can be constructed in $(X_r, \Pi_r)$ for any $r \geq 6 $, and the right-hand one for any $r \geq 4 $.

\begin{figure}[htb]
\begin{center}
\includegraphics[scale=0.6]{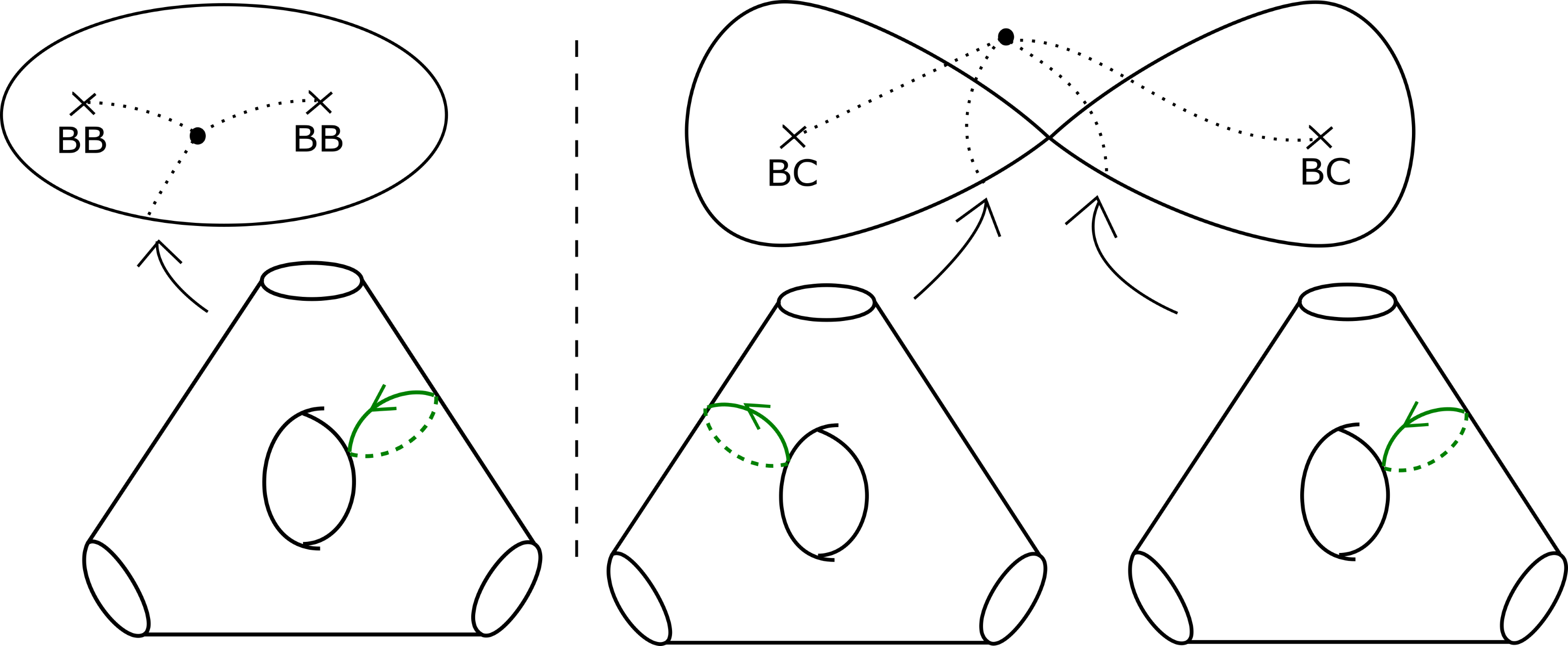}
\caption{Two examples of tori in $X_r$. We record the orientation of the meridional curve. As before, the dotted lines define paths to a smooth reference fibre. For legibility, we have moved this a little bit off-centre in the right-hand example.
}
\label{fig:simple_tori2}
\end{center}
\end{figure}

We will mostly interested in more sophisticated families of tori, such as the three-parameter one defined as follows.

\begin{definition}
Fix non-negative integers $k$, $l$, $m$. We define the immersed curve $\gamma_{k,l,m}$ in the base of $\Pi_r$ to be as given by Figure \ref{fig:fibred_torus_Tklm2}. As there are 7 basic configurations involved, each of which containing four critical points, it can certainly be drawn in the base of $\Pi_r$ for $r \geq 28$, and in fact one can check that $r \geq 18$ suffices. There is an embedded Lagrangian torus $T_{k,l,m} \subset X_r$, fibred over $\gamma_{k,l,m}$, also given by Figure \ref{fig:fibred_torus_Tklm2}. The figure also fixes an orientation of $\gamma_{k,l,m}$ for future reference. 

For $l=0$, our convention is to delete the obvious $BB$ configuration, namely, within the right-hand `lobe' of $\gamma_{k,l,m}$, the left-most of the two $BB$s.

\begin{figure}[htb]
\begin{center}
\includegraphics[scale=0.35]{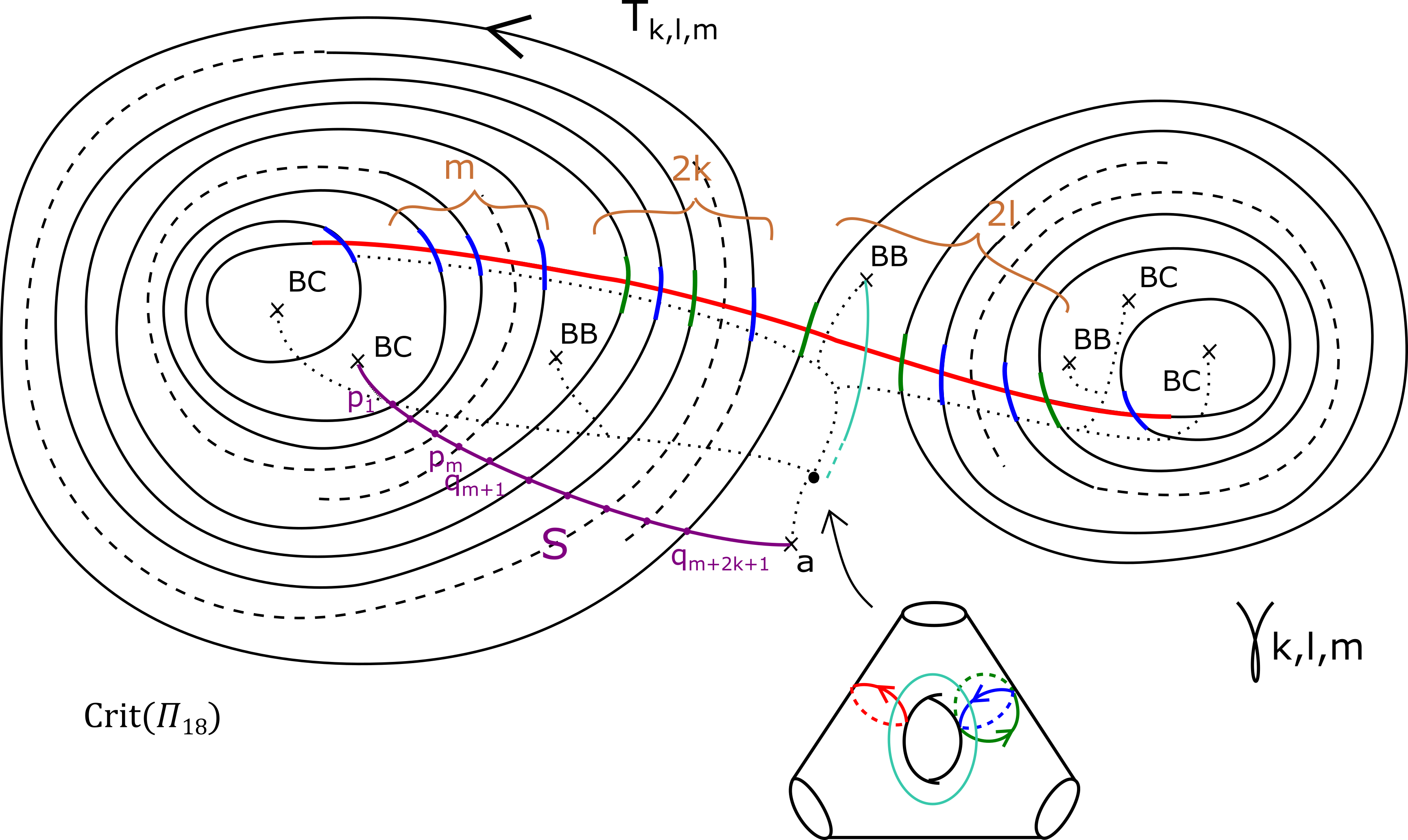}
\caption{The fibred torus $T_{k,l,m}$ above the immersed curve $\gamma_{k,l,m}$. The coloured sections of the base curve encode the oriented vanishing cycle living above it, with respect to the choices of references given by the small dotted lines. The purple curve represents a matching path between two type $a$ critical points, say $S$; the intersection points of $S$ and $T_{k,l,m}$ are labelled $p_1, \ldots, p_m, q_{m+1}, \ldots, q_{m+2k+1}$. This will be used in subsequent sections. The cyan line, which will also be used later, denotes two possible thimbles, both starting at an $a$ type critical point (of which there are two in a $BB$ configuration).
We take the immersed curve to lie in the base of $X_{18}$. 
}
\label{fig:fibred_torus_Tklm2}
\end{center}
\end{figure} 

\end{definition}

\begin{remark}\label{rmk:A_3}
To construct $T_{k,l,m}$ we didn't use the vanishing cycle $d$. In particular, we could instead work in the total space of a Lefschetz fibration with fibre a two-punctured elliptic curve (supporting an $A_3$ configuration of vanishing cycles), e.g.~$\{x^2 + y^4 + z^r =1 \}$. 
\end{remark}

\begin{remark}
One can use slight variations on these monodromy techniques to get an alternative construction of a monotone exact torus in the Milnor fibre of a simple elliptic singularity (for instance, the affine hypersurface $\{ x^2+y^4+z^4=1\}$) whose Floer cohomology with any vanishing cycle is zero, reproducing part of the results in \cite{Keating_tori}. 
\end{remark}

\begin{remark}
Let $I$ denote the immersed interval in the base just before it gets closed to an immersed $S^1$.
As mentioned above, the Lagrangian $I \times S^1$ given by using parallel transport might need to be modified by a Hamiltonian isotopy, in order to get the two ends to match up precisely. This will also be true with analogous constructions later, though we shall hereafter omit explicitly saying so, except for Maslov index calculations, to which the Hamiltonian isotopy will contribute.
\end{remark}

\subsection{Maslov indices, monotonicity and homology classes}

As before, let $L \subset X_r$ be a Lagrangian torus or Klein bottle fibred over an immersed oriented $S^1$, say $\gamma \subset \C \backslash \text{Crit}(\Pi_r)$. Pick a basis of $H_1(L; \Z)$ given by the ordered pair of:
\begin{enumerate}
\item any lift of $\gamma$ (with the induced orientation); and
\item the restriction of $L$ to the fibre of any point of $\gamma$ (with any orientation), i.e.~the class of the cycle $V$ which has been parallel transported.
\end{enumerate}

\begin{lemma}\label{thm:Maslov_indices}
Assume that 
there are choices of reference paths to a fixed (smooth) reference fibre such that 
the parallel transport of $V$ along $\gamma$ can be decomposed into a concatenation of basic configurations (i.e.~of types $BB$, $BC$, etc.), each traversed either positively or negatively. 

This is the case for instance for the examples of Figure \ref{fig:simple_tori2}, and the $T_{k,l,m}$ of Figure \ref{fig:fibred_torus_Tklm2}. Then, with respect to the basis given above, $L$ has Maslov class
$$(2t - \nu_+ + \nu_-\,, \, 0),$$ 
where
\begin{itemize}
\item $t$ is the total winding number of $\gamma$ (in other words, $\gamma$ has total curvature $2\pi t$);
\item $\nu_+$ is the number of basic configurations traversed positively (for our choice of orientation of $\gamma$);
\item $\nu_-$ is the number of basic configurations traversed negatively.
\end{itemize}
In particular, for any $m$, $T_{k,l,m}$ has Maslov class $(2(l-k), 0)$. 
\end{lemma}

\begin{proof} We use the set-up of Section \ref{sec:Maslov_background} to calculate Maslov indices. The claim about the Maslov index of `meridians' $V \subset M$ being zero is immediate, as they are vanishing cycles for the Lefschetz fibration $\Pi_r$.

For the other index, suppose first that we have a trivial fibration $M \times B \to B$, some disc $B \subset \C$, and that $\gamma \subset B$ is immersed. Fix  $V \subset M$ a vanishing cycle for $\Pi_r$, and let $L_\gamma$ be the Lagrangian given by parallel transporting $V$ along $\gamma$. Then $L_\gamma$ is an immersed Lagrangian, and the Maslov index of any lift of $\gamma$ is $2 t$, where $t$ is the total winding number of $\gamma$. 

More generally, the Maslov index of a lift of $\gamma$ given by concatenating basic configurations will be twice the total winding number of $\gamma$, adjusted for the effect of each of the basic configurations; we need to show that the contribution of each  basic configuration, positively traversed, is $-1$.

Consider the Klein bottle $K$ and the torus $T$ given in Figure \ref{fig:Maslov_basic_cases2}, associated to positively oriented embedded curves in the base.
To show that the contribution of a (positively traversed) $BB$ configuration is $-1$, it suffices to show that any lift of the base $S^1$ in $K$ has Maslov index one; to show that the contribution of a positively traversed $BC$ configuration is $-1$, it is enough to show that any lift of the base $S^1$ in $T$ has Maslov index zero. 
We shall prove the claim about $T$; the one about $K$ can be proved analogously. 

\begin{figure}[htb]
\begin{center}
\includegraphics[scale=0.50]{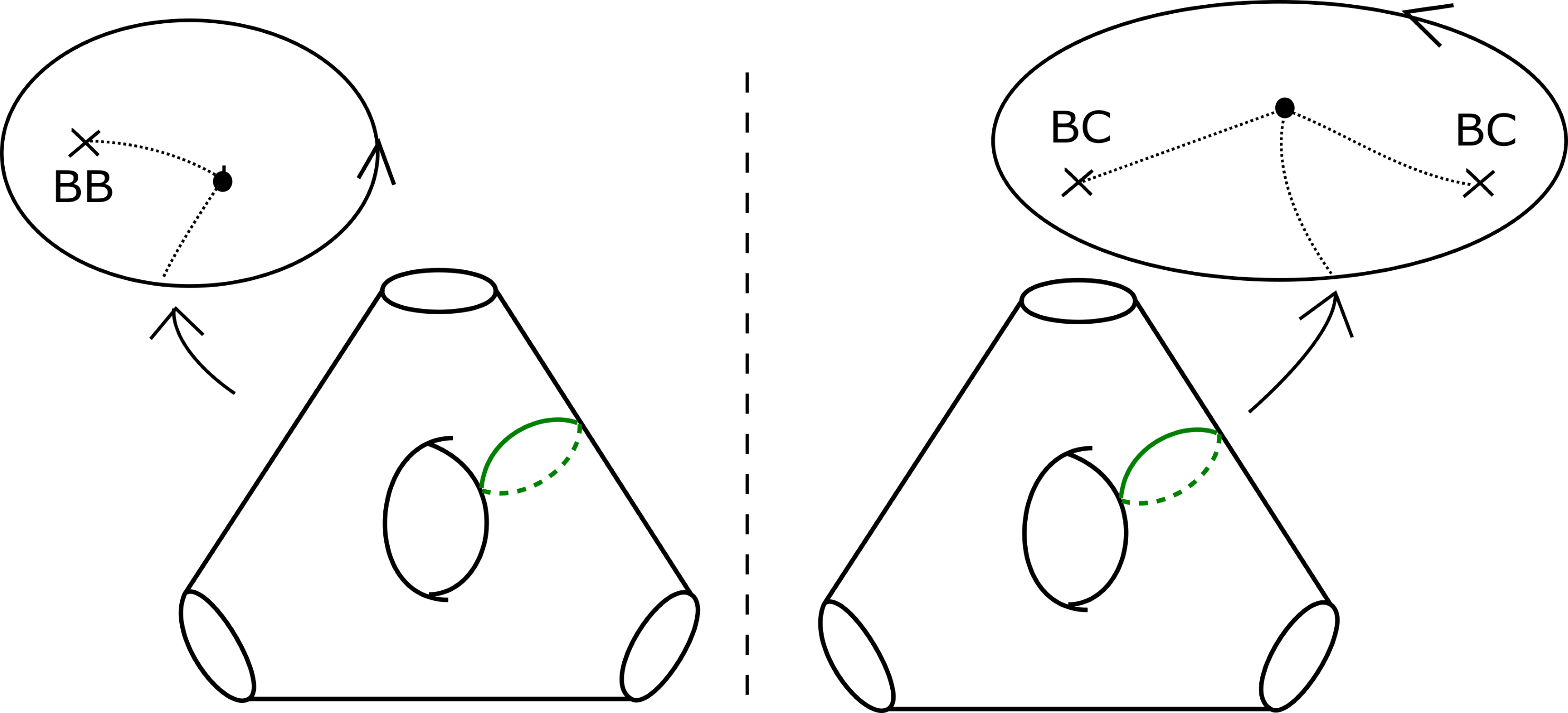}
\caption{Basic cases for the Maslov index computation: the Klein bottle $K$ (left) and torus $T$ (right).
}
\label{fig:Maslov_basic_cases2}
\end{center}
\end{figure} 

We will see that $T$ is Lagrangian isotopic to a Lagrangian torus obtained by performing Polterovich surgery on an ordered chain of four matching cycles, $B$,  $A_1$, $C$ and $A_2$, as given in Figure \ref{fig:Maslov_calculation2}. 
We use the convention of \cite[Appendix A]{Seidel_knotted_spheres}:
 in the case where two Lagrangian spheres $L_1$ and $L_2$ intersect transversally at a single point, the surgery $L_1 \# L_2$ is Lagrangian isotopic to $\tau^{-1}_{L_2} L_1 = \tau_{L_1} L_2$. 
Suppose $L_1$ and $L_2$ are two of the matching spheres at hand. To check our claim, one locally compares the different descriptions of $\tau^{-1}_{L_2} L_1 = \tau_{L_1} L_2$: both viewed as $ \tau_{L_1} L_2$ and as $\tau^{-1}_{L_2} L_1$ it can be described as a matching cycle, as in \cite[Figure 18.2]{Seidel_book}.
 Now consider Figure \ref{fig:Maslov_change2}. 
This shows  portions of $\tau^{-1}_B (A_1) = A_1 \# B$ and of $ \tau_{A_1} (B) = A_1 \# B$; gluing these together, one gets a different description of $A_1 \# B$ as a matching cycle, given in Figure \ref{fig:Maslov_change2andahalf}. Proceeding similarly at the intersections points of $(C, A_1)$, $(A_2, C)$ and $(B, A_2)$, one recovers $T$.

\begin{figure}[htb]
\begin{center}
\includegraphics[scale=0.6]{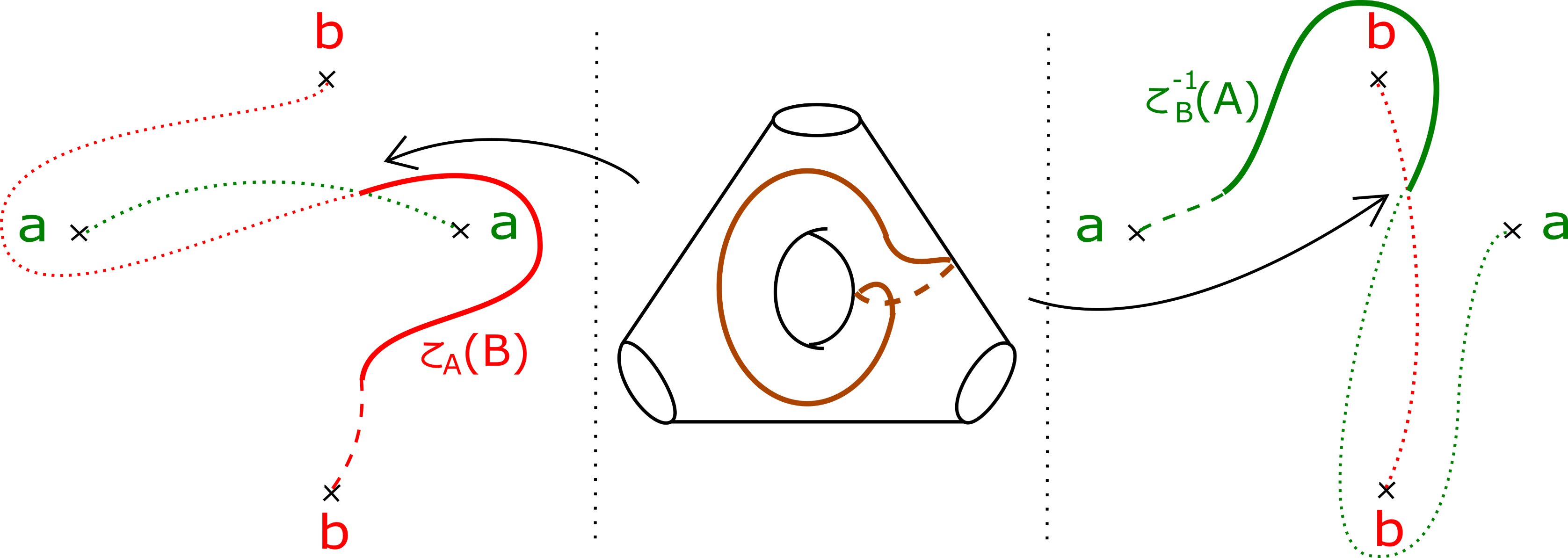}
\caption{Portions of $\tau^{-1}_B (A) = \tau_A (B) = A \# B$ used to compare the two descriptions of $T$. The brown cycle is the fibre above the point of the green, respectively  red, matching paths. 
}
\label{fig:Maslov_change2}
\end{center}
\end{figure}

\begin{figure}[htb]
\begin{center}
\includegraphics[scale=0.6]{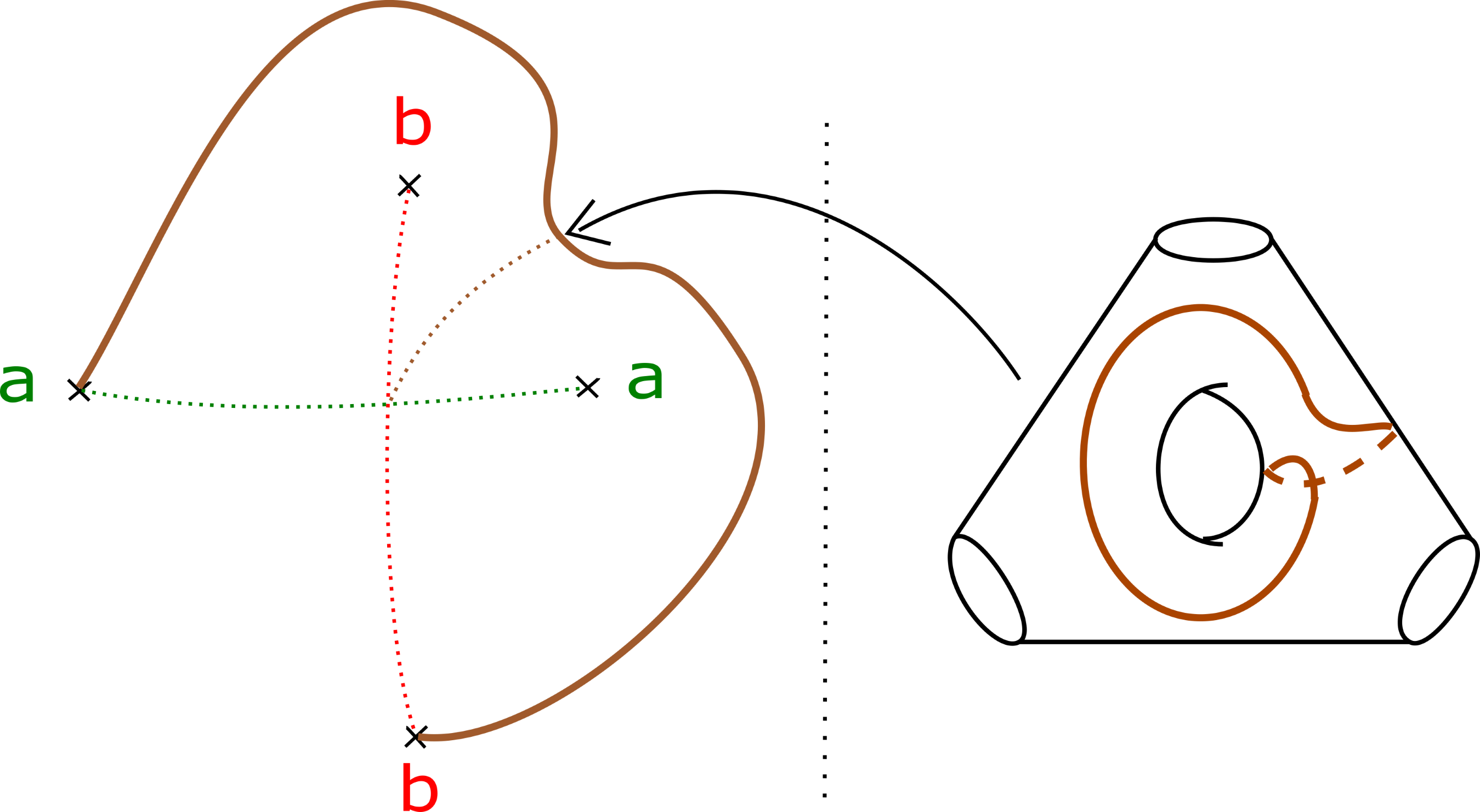}
\caption{Different description of $A \# B$ as a matching cycle. 
}
\label{fig:Maslov_change2andahalf}
\end{center}
\end{figure}

\begin{figure}[htb]
\begin{center}
\includegraphics[scale=0.35]{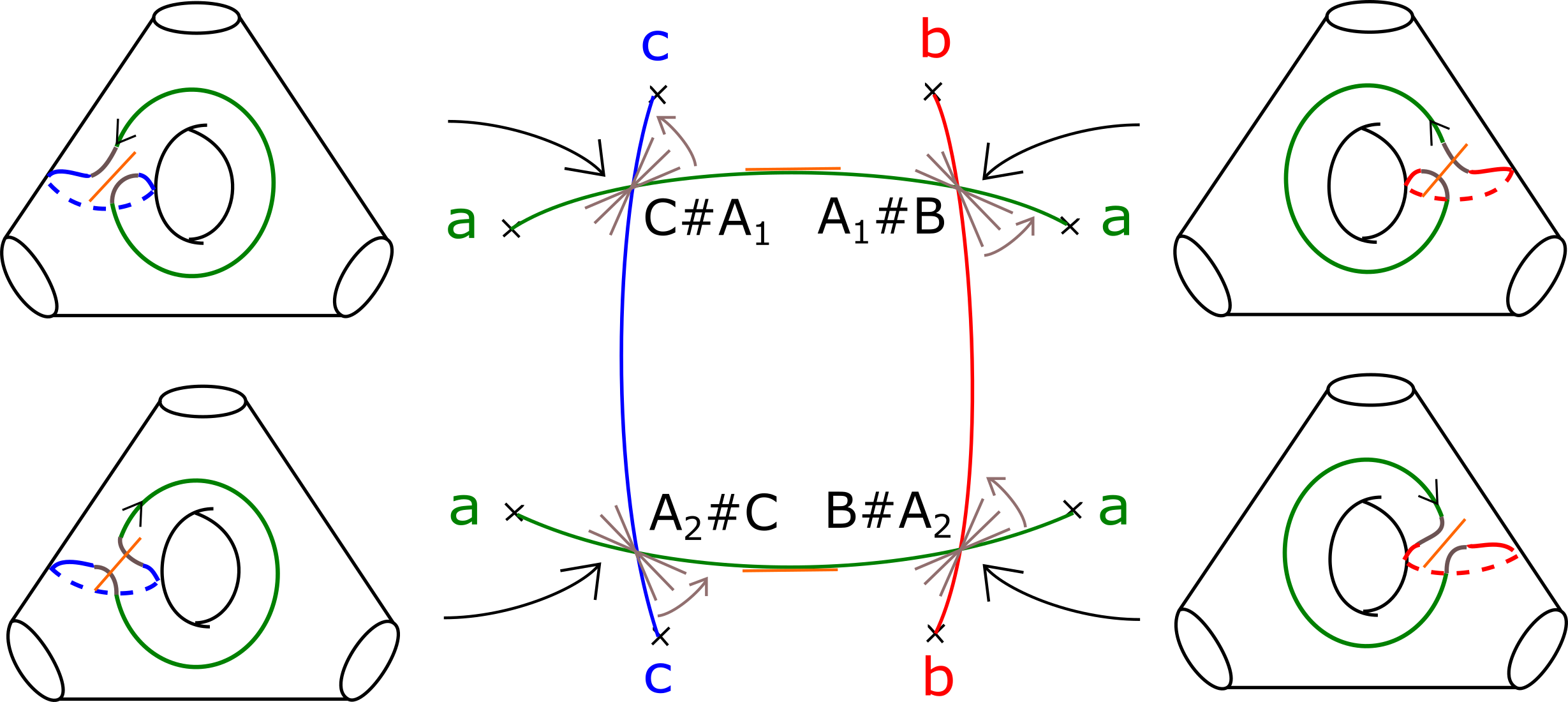}
\caption{A model for $T$ as the result of four Polterovich surgeries, at the intersections of the matching cycles $B$, $C$, $A_1$ and $A_2$.
}
\label{fig:Maslov_calculation2}
\end{center}
\end{figure} 
We calculate the Maslov index of a lift in $T$ of the base $S^1$ using the model for $T$ given by Polterovich surgery, and the choices of reference Lagrangian lines given at the end of Section \ref{sec:Maslov_background}. 
The path of Lagrangian planes in the base direction is given by following the matching paths, and the grey tangent directions at each of the surgery points. The orange segments give our choices of reference Lagrangian lines (in the fibre and base). Going around $S^1$, the reference Lagrangian line in the base is crossed twice,  both times positively. The reference Lagrangian in the fibre is crossed twice (at diagonally opposite surgery points), both times negatively.
Thus the required Maslvo index is zero.
\end{proof}

With the amount of information specified thus far, the paths $\gamma_{k,l,m}$ (and the associated Lagrangian submanifolds $T_{k,l,m}$) are only defined up to by a compactly supported isotopy of $\C$ relative to the critical values of $\Pi_r$. Any such isotopy
 lifts to a compactly supported isotopy of $X_r$. While this will not in general be a symplectic isotopy, it restricts to a Lagrangian isotopy $T_{k,l,m}$ and more generally of any Lagrangian $L$ fibred over an immersed path $\gamma \subset \C \backslash \text{Crit}(\Pi_r)$. 
We fix the Lagrangian isotopy class of $T_{k,l,m}$ by choosing a monotone representative, as follows.

\begin{lemma} Fix a constant $\kappa > 0$. There exists a  compactly supported isotopy of  $\C$ relative to the critical values of $\Pi_r$ such that the induced image of $T_{k,l,m}$ is monotone, with monotonicity constant $\kappa$.
\end{lemma}

\begin{proof} 
From Section \ref{sec:Maslov_background} it's enough to consider one disc for each of the generators of $H_1(T_{k,l,m})$. 
As the meridian curve on $T_{k,l,m}$, of Maslov index zero, is a vanishing cycle for $\Pi_r$, it thus bounds a Lagrangian disc in $X_r$, which in particular has symplectic area zero.
This means that for $T_{k,l,m}$ to be $\kappa-$monotone, we simply need to the symplectic area of any oriented disc with boundary a lift of $\gamma_{k,l,m}$ to be equal to $2 \kappa (l-k)$. Start with an oriented immersed disc with boundary $\gamma_{k,l,m}$, any pick any of its lifts. If it has signed area greater than $2 \kappa (l-k)$, we adjust by stretching outwards the outmost loop in the right lobe of $\gamma_{k,l,m}$; if it has area less than $2 \kappa (l-k)$, we instead stretch the outmost loop in the left lobe of $\gamma_{k,l,m}$. 
\end{proof}

Note that for fixed $\kappa$, $T_{k,l,m}$ is now determined up to Hamiltonian isotopy. 

We record the following property.

\begin{lemma}
Consider tori $T_{k,l,m}$ and $T_{k', l', m'}$ in $X_r$, with paths $\gamma_{k,l,m}$ and $\gamma_{k', l', m'}$ drawn using the same $BC$ and $BB$ configurations (so that the paths are almost superimposed, notwithstanding the different numbers of twists). 
They are homologous if and only if $m=m'$ and either $l, l' \geq 1$ or $l=l'=0$. (The latter condition is required simply because of our convention for $l=0$.)
\end{lemma}

\begin{proof}
To see that $T_{k,l,m}$ and $T_{k',l',m'}$ lie in different homology classes for $m \neq m'$, 
 consider the path given in purple in Figure \ref{fig:fibred_torus_Tklm2}, which we think of as a matching path between two critical points of type $a$ in the base of $\Pi_r$. Now notice that the associated vanishing cycle intersects $T_{k,l,m}$ transversally in the $m$ points $p_1, \ldots, p_m$, all with the same sign, and then with alternating signs at the $2k$ points $q_{m+1}, \ldots, q_{m+2k+1}$, which cancel.

Let's compare $T_{k,l,m}$ and $T_{k+1, l,m}$. One could calculate intersections with a basis for $H^2(X_r)$, given by vanishing cycles. Alternatively, note that the classes of $T_{k,l,m}$ and $T_{k+1, l,m}$ differ by the class of an immersed torus, given by joining up the two extra loops in the base of $T_{k+1, l,m}$. Along each of the two loops, the meridian $S^1$'s (that is, the classes of the fibres of the Lagrangian above points of the loops) are the same cycles, with opposite orientations. Thus it readily follows that this immersed torus is null-homologous, 
so $T_{k,l,m}$ and $T_{k+1,l,m}$ are homologous. Similarly, $T_{k,l,m}$ and $T_{k,l+1, m}$ are homologous if $l \geq 1$. 
\end{proof}

\subsection{Further tori}

We will want to consider variations on $T_{k,l,m}$.
The ones defined in Figure \ref{fig:fibred_linked_tori}, which we will call $R_{k,l,m}$ and $S_{n,p,q}$ ($k,\ldots, q \geq 0$), will be particularly useful. We'll care about the relative position of $R_{k,l,m}$ and $S_{n,p,q}$. Note that as drawn in Figure \ref{fig:fibred_linked_tori}, they do not intersect in $X_r$. Indeed, any intersection point in $X_r$ would project to an intersection point of the projections; restrict attention to those. Consider the fibre $M_{pt}$ above an intersection point of the two projections; now notice that we have constructed $R_{k,l,m}$ and $S_{n,p,q}$ in Figure \ref{fig:fibred_linked_tori} so that $R_{k,l,m}$ restricts to the vanishing cycle $b$ (with one or the other choice of orientation) on that fibre $M_{pt}$, whereas $S_{n,p,q}$ resticts either to $c$ (for half of the fibres above intersection points) or to $d$ (for the other half). As $b$ is disjoint from $c$ and $d$, it follows that $R_{k,l,m}$ and $S_{n,p,q}$ are disjoint Lagrangians. 
 Further, as there are twenty basic configurations involved in total, the whole picture certainly (crudely) fits in the basis of $\Pi_{80}$.

\begin{figure}[htb]
\begin{center}
\includegraphics[scale=0.3]{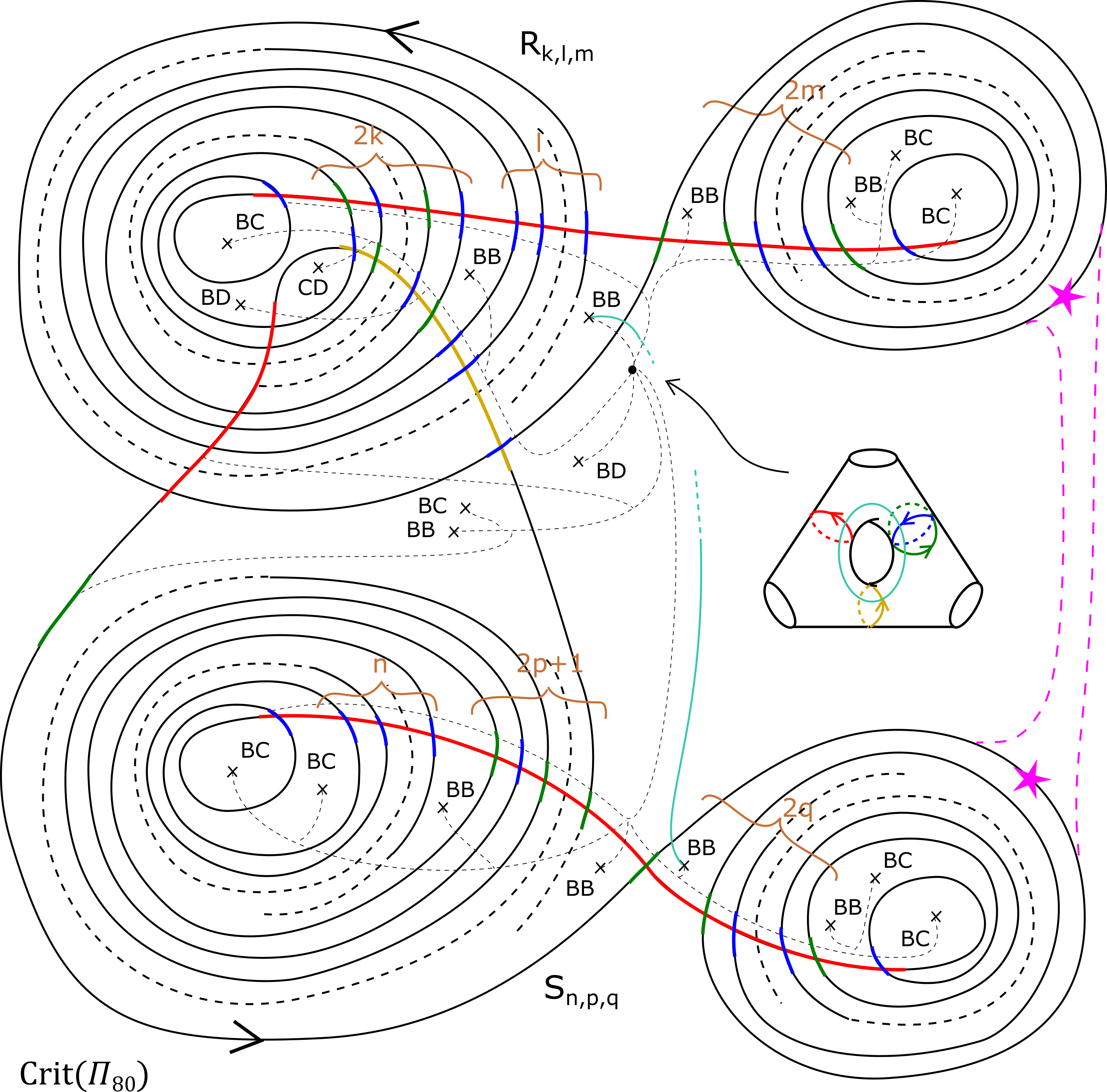}
\caption{The tori $R_{k,l,m}$ and $S_{n,p,q}$. 
The purple star will be used later to build further Lagrangians from  $R_{k,l,m}$ and $S_{n,p,q}$ -- for instance, note that they can be combined to form a single Lagrangian torus by joining them along the dashed purple lines.
As with $T_{k,l,m}$ the cyan lines denote thimbles starting at an $a$ type critical point (there are two possible ones for a $BB$ configuration); these will also be used later, for Polterovich surgery.
}
\label{fig:fibred_linked_tori}
\end{center}
\end{figure} 

It immediately follows from Lemma \ref{thm:Maslov_indices} that for $R_{k,l,m}$, the Maslov index of a lift of the base $S^1$ is $2(m-k-l)$, and that for $S_{n,p,q}$ it's $2(q-p-1)$. As before, by adjusting the area of the different lobes we can arrange for them to be monotone for any monotonicity constant, and their homology classes only depend on $l$, respectively $n$.

\begin{remark}\label{rmk:linking}
We will see in Section \ref{sec:annuli_counts_2d} that for any $(k,l,m)$ and $(n,p,q)$, these Lagrangians are linked with respect to the fibration, in the  following sense: there cannot exist a Hamiltonian isotopy (or indeed, a compactly supported symplectomorphism) such that the projection of their two images under the isotopy (or symplectomorphism) are disjoint. It will also follow from the arguments in that section that we get different links for e.g.~different pairs $(k,l)$. 
\end{remark}

\subsection{Higher genus}
 
The cyan thimbles of Figures \ref{fig:fibred_torus_Tklm2} and \ref{fig:fibred_linked_tori} intersect $T_{k,l,m}$, $R_{k,l,m}$ and $S_{n,p,q}$, respectively, transversally in a single point. We can patch these thimbles together to get matching paths, and perform Polterovich surgery at the intersection points of the associated matching cycles with copies of $T_{k,l,m}$, etc., 
to construct higher genus monotone Lagrangians in $X_r$ for sufficiently large $r$. See Figure \ref{fig:Lambda_g} for a genus $g$ Lagrangian which we will denote $\Lambda_g (T_{k_1, l_1, m_1}, \ldots, T_{k_g, l_g, m_g})$, and which can be realised in $X_{18g}$. One can make constructions using some $R_{k_i,l_i,m_i}$ or $S_{k_i,l_i,m_i}$ completely analogously.

\begin{figure}[htb]
\begin{center}
\includegraphics[scale=0.50]{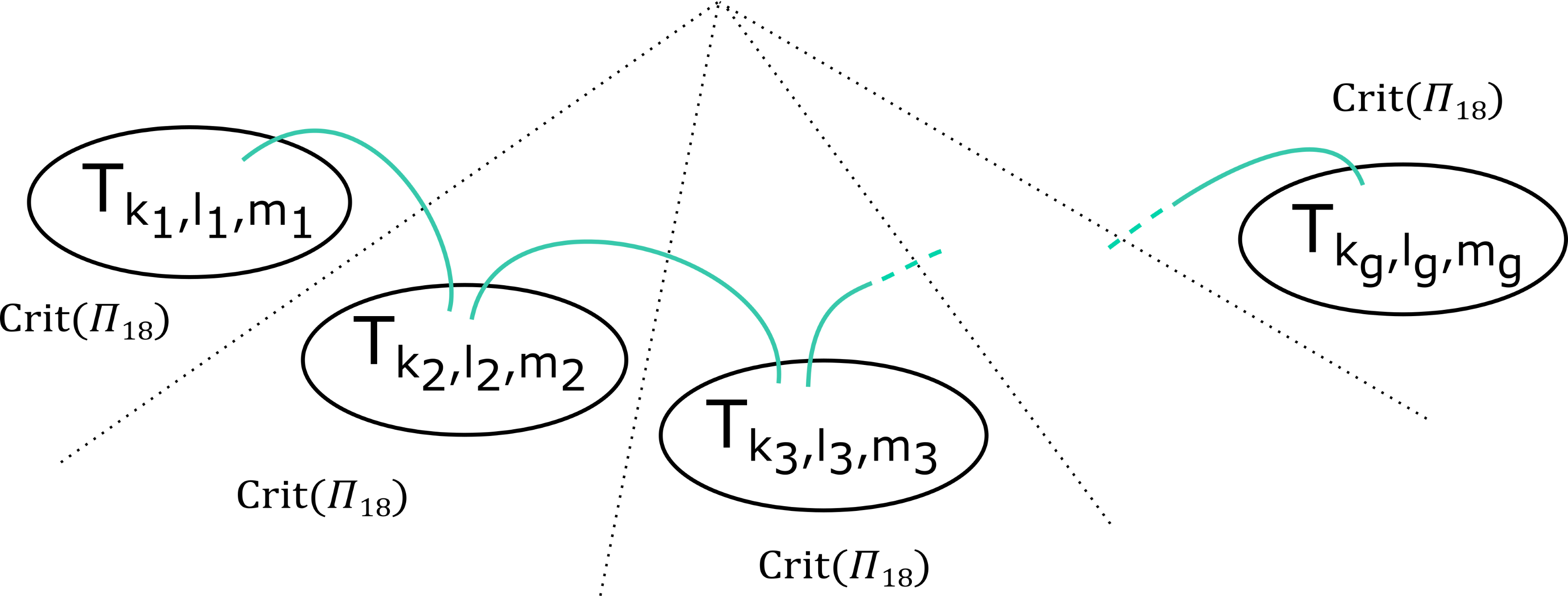}
\caption{ The genus $g$ monotone Lagrangian $\Lambda_g (T_{k_1, l_1, m_1}, \ldots, T_{k_g, l_g, m_g})$.
}
\label{fig:Lambda_g}
\end{center}
\end{figure}

The following readily follows. 

\begin{corollary}
Fix an integer $g \geq 1$ and a constant $\kappa > 0$. Suppose $r \geq 18g$. Then 
there exist Lagrangian surfaces of genus $g$ in $X_r$, say $L_g$ such that
\begin{itemize}
\item the $L_g$ are monotone with monotonicity constant $\kappa$; 
\item $L_g$ can lie in countably infinitely many homology classes;
\item in each of these homology classes, the Maslov class of $L_g$ can take any possible value.
\end{itemize}

\end{corollary}

We'll see how to tell these apart for fixed `soft' data as above in Sections \ref{sec:annuli} and \ref{sec:Floer_theory_properties}. 
Following Remark \ref{rmk:A_3}, note that these Lagrangians could even be realised in the `smaller' variety $X_{2,4,r}$ for the same bounds on $r$.

\subsection{Non-orientable examples}\label{sec:non-orientable}
The above constructions also readily give non-orientable Lagrangian submanifolds: replacing $2k$ with $2k \pm 1$, \emph{or} $2l$ with $2l \mp 1$, in the contruction of Figure \ref{fig:fibred_torus_Tklm2}, yields a Klein bottle. 
Lemma \ref{thm:Maslov_indices} shows that it has Maslov class
\bq
(2l-2k \mp 1, 0 )
\eq
with respect to the obvious basis, and can be arranged to be monotone. Moreover, by taking $r$ sufficiently large, one can obtain connected sums of Klein bottles (or a Klein bottle and several tori) of arbitrary length.

On the other hand, note that there are  topological constraint on non-orientable Lagrangians, going back to work of Givental \cite{Givental} for the case where the ambient manifold is $\C^2$, as follows.
\begin{lemma} \label{lem:non_orientable_constraint}
Suppose $X^{2n}$ is an exact symplectic manifold with trivial tangent bundle, and that there is a Lagrangian immersion $(S^1)^{n-2} \times\Sigma \looparrowright  X$, for a closed 2-manifold $\Sigma$. Then $\Sigma$ is either orientable, or the connected sum of an even number of $\R \P^2$s. 
\end{lemma}

\begin{proof}
This follows from a calculation of the total Steifel--Whitney classes of the tangent and normal bundles of $(S^1)^{n-2} \times\Sigma$, which must be both equal and inverse to each other.
\end{proof}

Note that by taking parallel copies of the same immersed $S^1$ curve (e.g.~in the right-hand side of Figure \ref{fig:simple_tori2}) and Polterovich surgering them with a fixed matching cycle, we can get embedded Lagrangian $\Sigma_g$ in $X_r$ for arbitrary $g$ so long as $r \geq 5$. (We need to take $ r \geq 5$ rather than $r \geq 4$ in order to also have the matching cycle for surgery.) Of course this sacrifices monotonicity. Similarly for connected sums of Klein bottles for $r \geq 4$. 

In particular, we readily get the following:

\begin{itemize}
\item So long as $r \geq 5$, we can get Lagrangians surfaces of arbitrarily high genus in $\{ x^2 + y^4 + z^r =1 \}$, which itself has an exact symplectic embedding into $X_r$ \footnote{This partially answers a question of Peter Kronheimer at the author's thesis defense.}; the bound on $r$ is sharp: $\{x^2 + y^4 + z^4 = 1\}$ can only contain Lagrangian tori or spheres as it has a negative semi-definite intersection form. (It is a parabolic, modality one singularity, see \cite[Chap.~2, \textsection 2.5]{Arnold_book}.) 

\item So long as $r \geq 9$,  all the diffeomorphism types of Lagrangian surfaces allowed by Lemma \ref{lem:non_orientable_constraint} can be realised in $\{ x^2 + y^4 + z^r =1 \}$. 

\item As $r$ increases we can gradually cover all possible Maslov types.

\end{itemize}

Of course, there is no reason for the later two bounds to be optimal.

\subsection{Examples with non-trivial automorphisms}\label{sec:finite_order_automorphisms}

As an aside, we briefly note that our techniques allow us to construct examples of monotone Lagrangians such that there are symplectomorphisms of $X_r$ which fix them set-wise but act non-trivially on e.g.~their homology. The most naive such construction is, for for odd $g$, to define a genus $g$ Lagrangian, say 
$\zeta_g (k_1, k_2, l_1, l_2, m_1, m_2)$, embedded inside $X_{18(g-1)}$, as in Figure \ref{fig:Zeta_g}: this is given by taking $(g-1)/2$ copies of $T_{k_1, l_1, m_1}$ and $T_{k_2, l_2, m_2}$, and attaching them via Polterovich surgery with $g-1$ different matching cycles. (Recall that to do a Polterovich surgery one needs to choose an ordering of the two Lagrangians involved; in this case we want to pick any orderings that are cyclicly symmetric.) 
Using the same techniques as in the proof of Lemma \ref{thm:Maslov_indices} (see in particular Figure \ref{fig:Maslov_calculation2}), we get that the Maslov index of the `first' longitude  -- the lift of the obvious cyclically symmetric curve in the base (which travels around the blue and red matching cycles, counterclockwise) -- is $2(3-g)$; we ensure that $\zeta_g (k_1, k_2, l_1, l_2, m_1, m_2)$ is monotone by suitably adjusting the areas of the `overlaps' between the blue and red matching paths in the base of $\Pi_r$, as sketched in Figure \ref{fig:Zeta_g}.

\begin{figure}[htb]
\begin{center}
\includegraphics[scale=0.60]{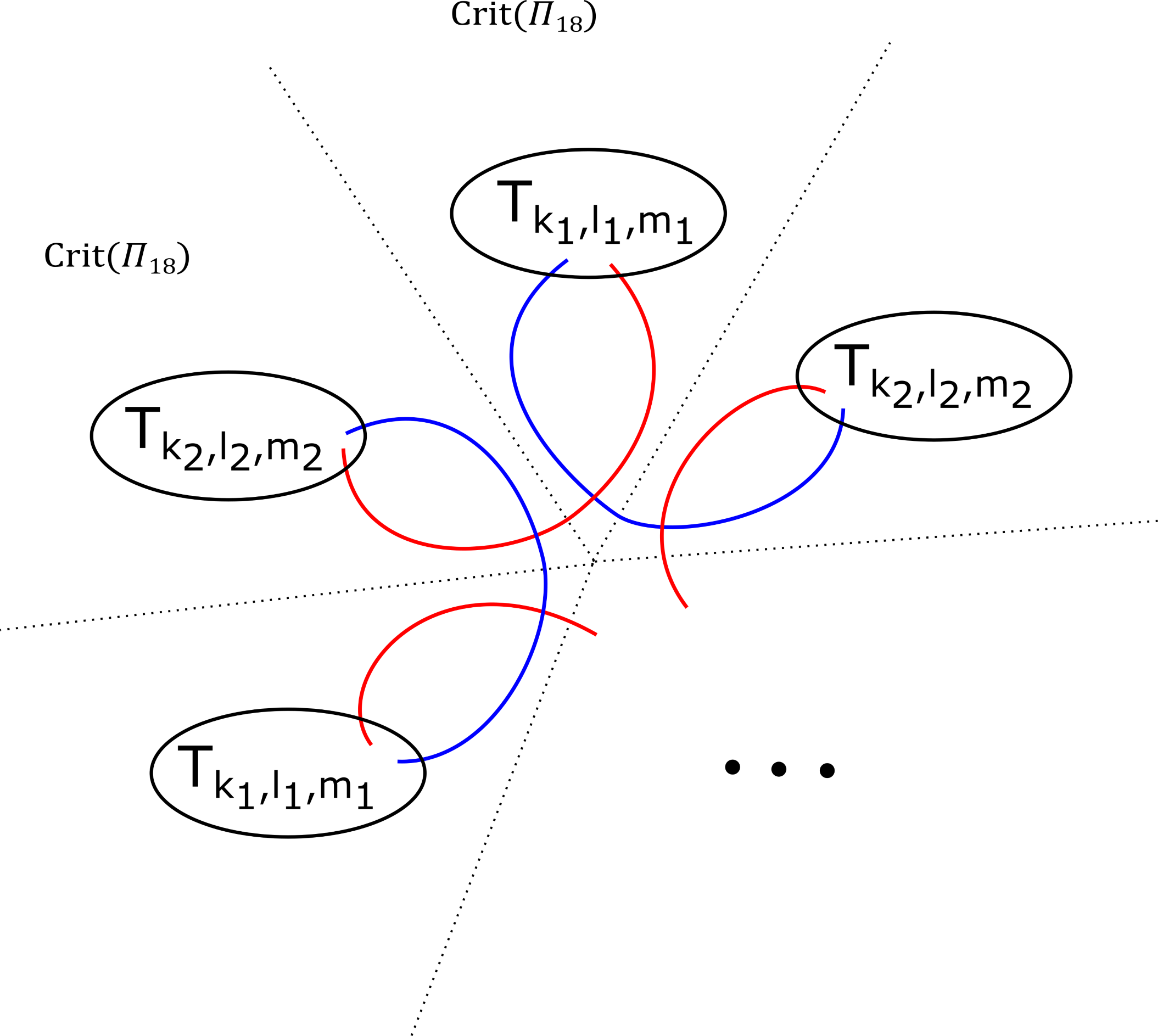}
\caption{ The genus $g$ monotone Lagrangian $\zeta_g (k_1, k_2, l_1, l_2, m_1, m_2)$ inside $X_r$, for $r = 18(g-1)$. The blue, resp.~red, curves correspond to matching cycles of types $b$, resp.~$c$, attached to the critical points of those two types in the obvious $BB$ configurations, namely the ones also used for the cyan (type $a$) matching cycles of Figure \ref{fig:Lambda_g} (with more details on Figure \ref{fig:fibred_torus_Tklm2}).}
\label{fig:Zeta_g}
\end{center}
\end{figure} 

Let $r = 18(g-1)$. 
Consider the automorphism $\rho$ of $X_{r}$ corresponding to the positive rotation of the base by angle $4\pi / (g-1)$, which is a power of the map $\Phi$ defined in Proposition \ref{prop:Phi_definition}. By Proposition \ref{prop:Phi_Dehntwists}, this is Hamiltonian isotopic to a product of Dehn twists in spheres in $X_r$, which are themselves vanishing cycles for the singularity $x^3 +y^3 + z^r$, say $\rho = \tau_{V_1} \ldots \tau_{V_{n_\rho}}$. Moreover, if $k_2 = \ldots = k_g$, and similarly for the $l_i$ and $m_i$, $\rho$ fixes $\zeta_g$ setwise, and acts as an order $(g-1)/2$ rotation pointwise.

\begin{remark}
The Maslov index calculation would be the same if we had used type $a$ matching cycles (i.e.~with the conventions of Figures \ref{fig:fibred_torus_Tklm2} and \ref{fig:Lambda_g}, cyan matching paths), though in that case there would be no ready way of adjusting things to ensure monotonicity. Dropping the monotonicity requirement, such a construction gives genus $g$ Lagrangians and a symplectomorphism of $X_r$ which fixes them setwise and acts as an order $g-1$ rotation pointwise for arbirary $g$. (The lowest genus case would need separate treatment: instead, one could for instance construct a genus 2 monotone Lagrangian with a rotation of order two, by using Polterovich surgery on two copies of $T_{k,l,m}$ inside $X_{r}$ for $r \geq 36$.)
\end{remark}

\begin{remark}
Following Remark \ref{rmk:A_3}, note that these constructions could instead have been realised in  $X_{2,4,r}$, as the type $d$ vanishing cycle is never needed.
\end{remark}



\section{Monotone Lagrangians in $\C^3$}\label{sec:C^3}

\subsection{Infinitely many monotone $S^1 \times \Sigma_g$}\label{sec:monotone_in_C^3}

We will use the following result from geometric group theory.

\begin{theorem}\label{thm:split_diffeo}\cite{Waldhausen1, Waldhausen2, Waldhausen3}
Let $\Sigma$ be a closed surface of negative Euler characteristic.
Suppose $f$ is a diffeomorphism of $S^1 \times \Sigma$. Then $f$ is isotopic to a product, i.e.~an element of $\Z/2 \oplus \text{Diff}(\Sigma)$,
where the first factor acts on $S^1$, and the second factor on $\Sigma$.

\end{theorem}

\begin{proof}
This was established in work of Waldhausen \cite{Waldhausen1, Waldhausen2, Waldhausen3}; for an account of the results in English, see e.g.~the exposition in \cite[Section 8.1, Theorem 4]{Orlik}.
\end{proof}

This now allows us to `upgrade' our constructions $\Lambda_g$ to get Lagrangians in $\C^3$, as follows.

\begin{theorem}\label{thm:infinite_family_C3}
Fix $g \geq 2$, and any monotonicity constant $\kappa$. Then there exist infinitely many monotone Lagrangian $S^1 \times \Sigma_g$ in $\C^3$, distinct up to any equivalence that preserves Maslov classes. (This includes Lagrangian isotopy, and almost-complex diffeomorphisms of $\C^3$ -- so in particular, symplectomorphisms.)
\end{theorem}

\begin{proof}
\underline{Step 1: Construction.}
Recall $Y_{r,1} = \{ x^3 + y^3 + z^r + w=1 \} \cong \C^3$, and that $P_{r,1}: Y_{r,1} \to \C$ has smooth fibre $X_r$. 
Given a monotone Lagrangian $\Lambda_g \subset X_r$, with monotonicity constant $\kappa$, our strategy will be to construct a monotone Lagrangian $S^1 \times \Lambda_g$ in $\C^3$ by taking a product with a suitable $S^1$ in the basis of $P_{r,1}$, away from the critical values.

The map $P_{r,1}$ has finitely many critical points and values; pick a disc $D$ in the base away from these; this can be chosen with arbitrary (finite) symplectic area. 
The fibration $P_{r,1}$ above $D$ is essentially trivial. More precisely, 
following the ideas of Section \ref{sec:deformations}, 
for any fixed compact subset $K \subset X_r$, we can assume that after a Moser isotopy, $P_{r,1}$ restricts to the trivial fibration $K \times D \to D$, where $K \times D$ is equipped with a product symplectic form.

   Pick an embedding $\gamma: S^1 \to D$. Up to Hamiltonian isotopy, the symplectic parallel transport map about $\gamma$ is trivial. In particular, by taking the union of the images of $\Lambda_g$ one gets a Lagrangian $ \Lambda_g \times S^1 \subset \C^3$, fibred over $\gamma$. By construction, the positively oriented curves $ \{ pt \}  \times S^1$ have Maslov index two. Let $D_\gamma \subset D$ be the (positively oriented) disc with boundary $\gamma$. This lifts to a family of discs $\{ pt \} \times D_\gamma \subset K \times D$, with boundary on $\{ pt  \} \times \gamma$, where $\{ pt \}$ varies in $K$.  Adjusting $\gamma$, one can arrange for these to have symplectic area $2\kappa$. Then, by construction, $S^1 \times \Lambda_g \subset \C^3$ is a monotone Lagrangian.

\underline{Step 2: Invariants.} Recall that $\Lambda_g$ depended on some choices:
$$
\Lambda_g = \Lambda_g (T_{k_1, l_1, m_1}, \ldots, T_{k_g, l_g, m_g}).
$$
Consider a basis for $H_1 (\Lambda_g)$ given by $(\alpha_1, \beta_1, \ldots, \alpha_g, \beta_g)$, where $\alpha_i$ is any lift of the base $S^1$ for $T_{k_i, l_i, m_i}$, and $\beta_i$ is the meridian $S^1$ for it (i.e.~a cycle on $M$, the thrice-punctured elliptic curve $\{ x^3 + y^3 = 1\}$); this is the natural generalisation of the basis considered in Lemma \ref{thm:Maslov_indices}; moreover, by Lemma \ref{thm:Maslov_indices}, with respect to the induced basis for $H_1 (S^1 \times \Lambda_g) \cong H_1(S^1) \oplus H_1 (\Lambda_g)$, $S^1 \times \Lambda_g$ has Maslov class given by the indices:
$$
(2, 2(l_1-k_1), 0, 2(l_2-k_2), 0, \ldots, 2(l_g-k_g), 0).
$$ 
Suppose we're also given $\Lambda_g'= \Lambda_g (T_{k_1', l_1', m_1'}, \ldots, T_{k_g', l_g', m_g'})$. Let $N = \text{gcd}\{l_1-k_1, \ldots, l_g-k_g\}$ and $N' = \text{gcd}\{l_1'-k_1', \ldots, l_g'-k_g'\}$; $2N$, respectively $2N'$, is the minimal Maslov number of $\Lambda_g$, resp.~$\Lambda'_g$. By Theorem \ref{thm:split_diffeo}, if $N \neq N'$, then there cannot be any map taking $S^1 \times \Lambda_g$ to $S^1 \times \Lambda_g'$ and preserving the Maslov class: irrespective of the choice of bases for $H_1(\Lambda_1; \Z)$ and $H_1(\Lambda'_g, \Z)$, we cannot get the two collections of indices to agree. In particular, this means that the two cannot be Lagrangian isotopic, and that there cannot exist any symplectomorphism of $\C^3$ taking one to the other.
\end{proof}

\begin{remark}\label{rmk:Georgios}
The following results from discussion with Georgios Dimitroglou-Rizell: first, note that the Lagrangian $\Lambda_g$ in the fibre $X_r$ is an isotropic subcritical surface in the whole of $\C^3$, and is exact in $\C^3$ whenever it is exact in $X_g$. 
At least in this case, our construction should be equivalent to the circle bundle construction by Audin, Lalonde and Polterovich \cite{ALP}: the $S^1 \times \Sigma_g$ can instead be thought of as the product of $\Sigma_g$ with a circle in its (trivial) symplectic normal bundle. (The circle may need to be taken to be very small, which can be remedied by an overall scaling -- in the non-exact case one would have to be more careful.)
Now recall that exact subcritical isotropics satisfy an h-principle (one can use e.g.~the h-principle for isotropic submanifolds in $J^1(\R^3)$ as given in \cite[Theorem 12.4.1]{Eliashberg-Mishachev}); thus two exact subcritical isotropics with the same soft invariants will be Hamiltonian isotopic, and, at the cost of sufficiently shrinking the circles, it should be possible to carry along the circle bundle Lagrangians with the Hamiltonian isotopy (and later perform an overall rescalling of $\C^3$ to make up for having shrunk the circles).

\end{remark}


\subsection{Extensions}

\subsubsection{Non-trivial Hamiltonian Lagrangian monodromy group}
\label{sec:Hamiltonian_monodromy}
In general, given a Lagrangian $L$ in a symplectic manifold $(X, \omega)$, we can study the group of time-1 Hamiltonian isotopies of $X$ which preserve $L$ set-wise, called the Hamiltonian Lagrangian monodromy group of $L$; often one simply studies its action on a flavour of homology of $L$. Mei-Lin Yau \cite{MLYau} studied this in the case of the standard (i.e.~Clifford) and Chekanov monotone tori in $\C^2$; in both cases, the Hamiltonian Lagrangian monodromy group for $\Z$--homology is $\Z/2$. In contrast, Hu, Lalonde and Leclerq \cite{Hu-Lalonde-Leclerc} showed that if $L$ is a weakly exact closed Lagrangian submanifold, i.e.~if $[\omega] (\pi_2(X, L))$ vanishes, then the Hamiltonian Lagrangian monodromy group acts trivially on the $\Z/2$--homology of $L$.

As an aside on our main results, we note that our techniques give examples of monotone Lagrangians in $\C^3$ with non-trivial Hamiltonian Lagrangian monodromy group. 
This builds on Section \ref{sec:finite_order_automorphisms}, as follows: consider a monotone Lagrangian of the form $S^1 \times \zeta_g$ ($g$ odd) in $\C^3 \cong Y_{r,1} = \{ x^3 + y^3 + z^r + w=1 \}$, given by the product of a copy of $\zeta_g$ (described in that section) in a smooth fibre of $P_{r,1}$ with an $S^1$ in a locally trivial part of the base -- i.e.~the same construction as in Section \ref{sec:monotone_in_C^3} above. 
Recall that there exists a symplectomorphism $\rho = \tau_{V_1} \ldots \tau_{V_{n_\rho}} $ of $X_r$ which fixes $\zeta_g$ setwise and acts as an order $(g-1)/2$ rotation pointwise. The $V_i$ are vanishing cycles for the singularity $x^3 + y^3 + z^r$, which implies that each of the $\tau_{V_i}$ is the monodromy of a path in the base of $P_{r,1}$ (avoiding the critical points). It then follows that the map $\text{Id} \times \rho: S^1 \times \zeta_g \to S^1 \times \zeta_g$ is induced by a Hamiltonian isotopy of $\C^3$. 

\subsubsection{Non-orientable examples}

Much of the constructions and arguments above extend to the non-orientable case. In particular, combining the constructions of Section \ref{sec:non-orientable} with the ideas of Section \ref{sec:monotone_in_C^3}, we get infinitely many monotone Lagrangian $S^1 \times \Sigma$ in $\C^3$, where $\Sigma$ is the connect sum of an arbitrary number of Klein bottles and tori, which are distinguished by the minimal Maslov number of $\Sigma$; and, using the ideas of Section \ref{sec:Hamiltonian_monodromy}, we can give families of examples of such $S^1 \times \Sigma$ with non-trivial Hamiltonian Lagrangian monodromy group.



\section{Monotone Lagrangians in affine 3-folds and holomorphic annuli}\label{sec:annuli}

\subsection{Further constructions of monotone $S^1 \times \Sigma_g$ in $Y_{r,s}$}\label{sec:S^1xXi_g}

The constructions of Lagrangian surfaces in Section \ref{sec:building_blocks} relied on one-dimensional features: the existence of a sequence of positive Dehn twists on $M$ (the smooth fibre of $\Pi_r: X_r \to \C$) taking an exact Lagrangian, namely the vanishing cycle $b$, to a disjoint exact Lagrangian, namely the vanishing cycle $c$; and the existence of a sequence of \emph{positive} twists taking $c$ back to $b$. 

Given a sequence of positive Dehn twists displacing one of our Lagrangian surfaces (e.g.~$T_{k,l,m}$), and another one bringing the displaced copy back to the starting point, one could use similar ideas to construct interesting monotone Lagrangians in 3-dimensional Brieskorn--Pham hypersurfaces. 
Propositions \ref{prop:Phi_definition} and \ref{prop:Phi_Dehntwists} provide these sequences, possibly at the cost of passing to a larger $r$: they allow us to write down lots of constructions of embedded Lagrangian $S^1 \times \Sigma$ in $Y_{r,s}$, for sufficiently large $r$ and $s$, where $\Sigma \subset X_r$ is a Lagrangian surface from Section \ref{sec:building_blocks}, and  $S^1 \times \Sigma$ itself is fibred over an immersed $S^1$ in the base of $P_{r,s}$. 

We'll shortly give examples, and explain how to calculate the Maslov index of the $S^1$ factor. Of course, all of the constructions inside $\C^3$ also give monotone Lagrangians of the form $S^1 \times \Sigma_g$ in affine 3-folds, which can still be distinguished by the Maslov index of the $\Sigma_g$ factor. On the other hand, recall that by \cite[Theorem B]{Evans-Kedra}, given a monotone Lagrangian $S^1 \times \Sigma_g$ in $\C^3$, the $S^1$ factor must have Maslov index two. Of course this need not be the case for such Lagrangians in a general affine hypersurface -- indeed, we'll see examples where the $S^1$ factor takes arbitrary (even) Maslov index, including new Maslov two examples. 

Instead of presenting the simplest possible construction for a given topological type, we directly describe examples which are sophisticated enough to give all of the applications we have in mind. In particular, to get a setting in which the count of holomorphic annuli will readily be well-defined, we make slightly careful choices.
Concretely, our constructions will be in $Y_{r,s} = X_{3,3,r,s}$. For Floer-theoretic features, explored in Section \ref{sec:Floer_theory_properties}, it would typically be enough to work in $X_{2,4,r,s}$ (following on from Remark \ref{rmk:A_3}), which we will note when applicable.

We start with a variation on the genus $g$ Lagrangian surfaces in $X_r$ considered so far: the monotone Lagrangian surface $$\Xi_g (\{k_0, l_0, m_0\}, \ldots,\{k_g, l_g, m_g\})$$
described by Figure \ref{fig:Xi_g}, which in turn must be read with the conventions of Figures  \ref{fig:fibred_torus_Tklm2} and \ref{fig:fibred_linked_tori}. The surface $\Xi_g$ has genus $g$, as $S_{k_0, l_0, m_0}$ and $R_{k_1, l_1,m_1}$ get combined to form a single genus one component. 

\begin{figure}[htb]
\begin{center}
\includegraphics[scale=0.30]{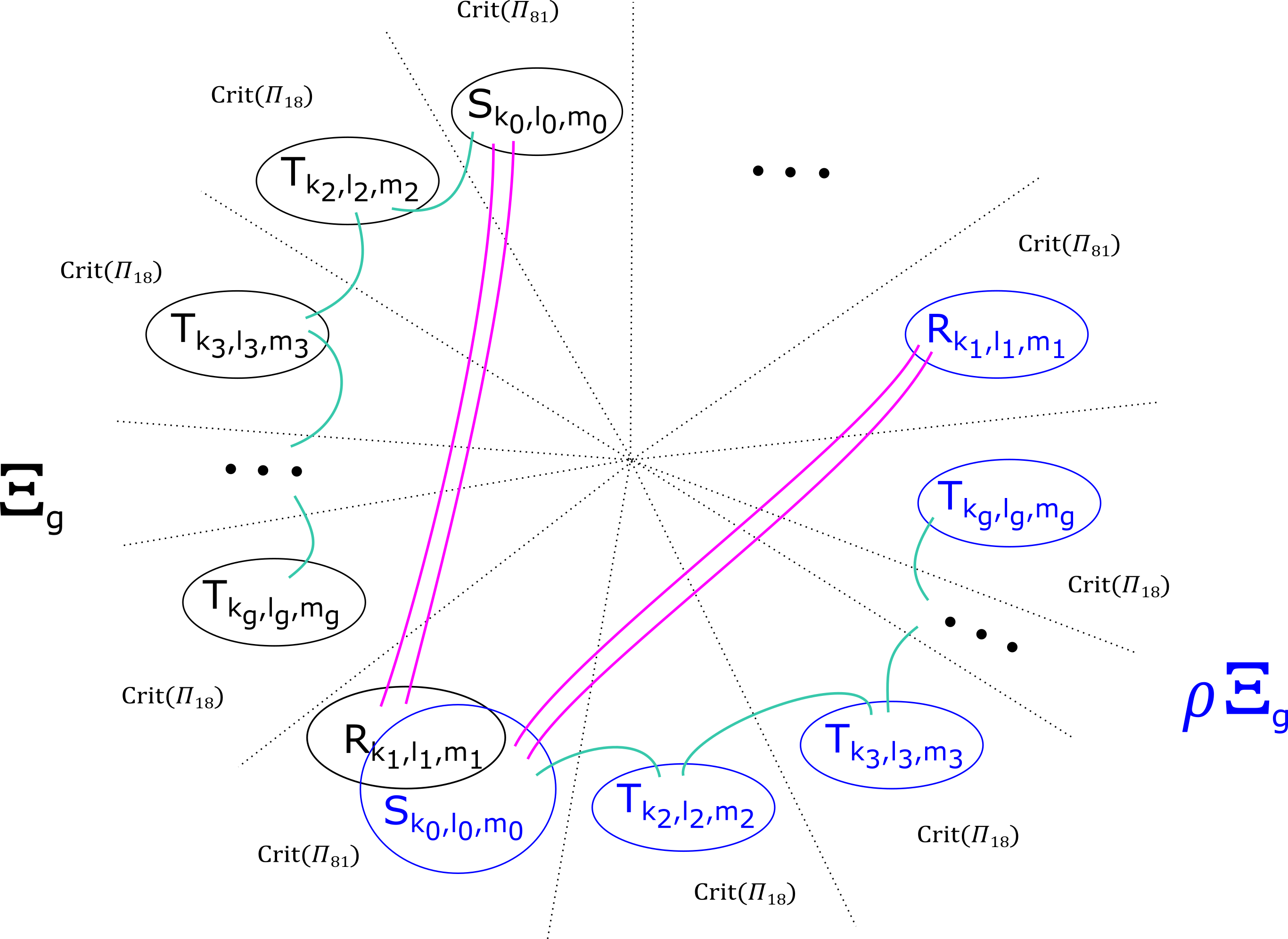}
\caption{ The genus $g$ monotone Lagrangian $\Xi_g (\{k_0, l_0, m_0\}, \ldots,\{k_g, l_g, m_g\})$ inside $X_r$, for $r \geq 243+36(g-1)$, together with $\rho(\Xi_g)$.}
\label{fig:Xi_g}
\end{center}
\end{figure} 

We take the immersed curves describing genus one components to each lie in $\text{Crit}(X_{81})$ or $\text{Crit}(X_{18})$ as indicated. Let $r'=r'(g)=  243 + 36 (g-1)$, the smallest multiple of 3 compatible with Figure \ref{fig:Xi_g}. Set $\eta_g = (81+18(g-1))/3$ and $\tilde{\eta}_g  = (162+18g)/3$. 
Note $\eta_g + \tilde\eta_g = r'/3$. 
Pick any $r \geq r'$.  We can apply Propositions \ref{prop:Phi_definition} and \ref{prop:Phi_Dehntwists}: there exist  compactly supported symplectomorphisms of $X_{r}$,  say $\rho$ and $\tilde{\rho}$, corresponding to positive rotations of the base of $\Pi_{r'}$ by $\eta_g$, respectively $\tilde{\eta}_g$, blocks of the form $\text{Crit}(X_{3})$ (i.e.~12 critical points). Using the notation $\Phi$ of Section \ref{sec:rotations}, $\rho = \Phi^{-\eta_g}$ and $\tilde{\rho} = \Phi^{-\tilde{\eta}_g}$. Both $\rho$ and $\tilde{\rho}$ are given by a product of negative Dehn twists in spheres which are matching cycles for $\Pi_{r'}$ (and so for $\Pi_r$: the symplectomorphism is extended from $X_{r'}$ to $X_r$ by the identity). Moreover, these matching cycles are themselves vanishing cycles for the singularity $x^3 + y^3 + z^r$ -- indeed, their ordered list is a distinguished collection of vanishing cycles for $x^3 + y^3 + z^{r'}$, say
$$V_1, \ldots, V_{4(r'-1)}, \ldots, V_1, \ldots, V_{4(r'-1)}$$
 repeated $3 \eta_g$ times in the case of $\rho$, and $3 \tilde{\eta}_g$ times in that of $\tilde{\rho}$. 

Note that $\rho(\Xi_g)$ is disjoint from $\Xi_g$, and that the intersection of the images under $\Pi_r$ of $\Xi_g$ and $\rho(\Xi_g)$ is precisely given by the intersection of the images of $S_{k_0, l_0, m_0}$ and $R_{k_1, l_1, m_1}$. 
Moreover, $\tilde{\rho} \rho (\Xi_g) = \Xi_g$, pointwise up to Hamiltonian isotopy.

We can use this to construct various families of embedded monotone Lagrangian $S^1 \times \Xi_g$ in $Y_{r,s}$ for $r \geq r'(g)$ and sufficiently large $s$. 
We will consider three different variations, with $S^1$ given by a path $\gamma_{i}$ as follows:
\begin{itemize}
\item
 $\gamma_0$ as given in Figure \ref{fig:S^1xXi_g_Maslov0};

\item $\gamma_2$ and $\gamma_4$ as given in Figure \ref{fig:S^1xXi_g_Maslov2};

\item $\gamma_6$, and $\gamma_{4n}$ and $\gamma_{4n+2}$ for $n \geq 2$, as given in Figure \ref{fig:S^1xXi_g_generalMaslov}. 

\end{itemize}

\begin{figure}[htb]
\begin{center}
\includegraphics[scale=0.30]{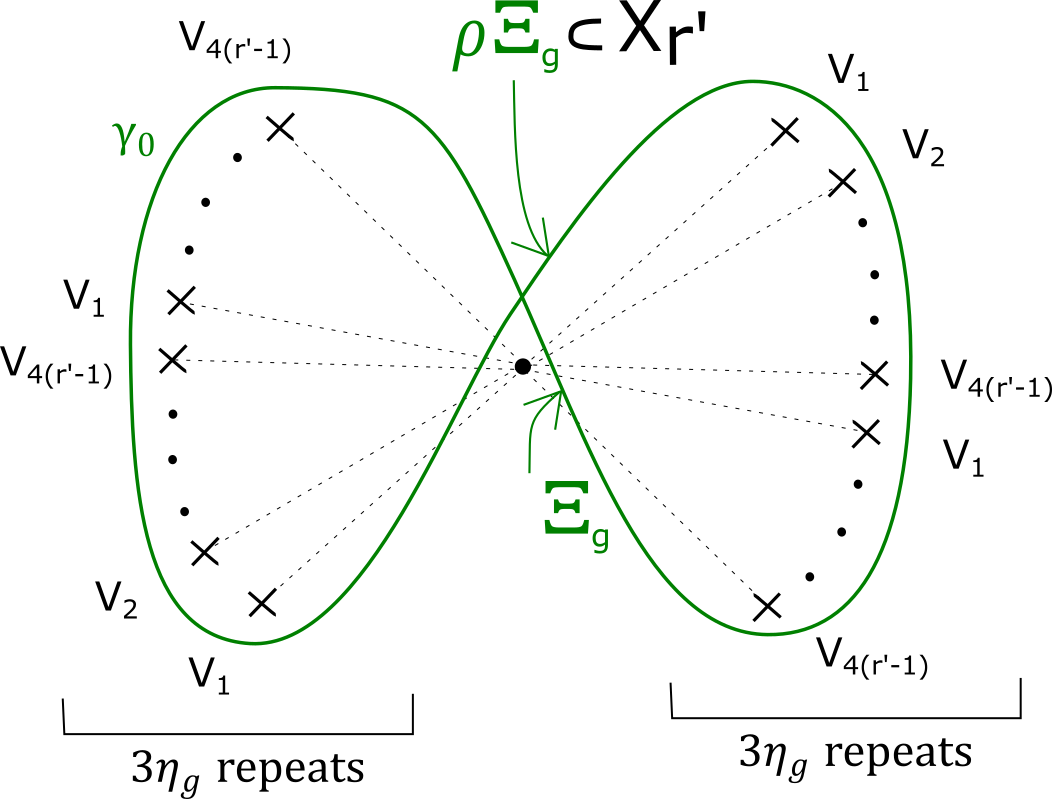}
\caption{ The monotone Lagrangian $\gamma_0 \times \Xi_g$, given by its image in (part of) the base of $P_{r,t}$.}
\label{fig:S^1xXi_g_Maslov0}
\end{center}
\end{figure} 

\begin{figure}[htb]
\begin{center}
\includegraphics[scale=0.30]{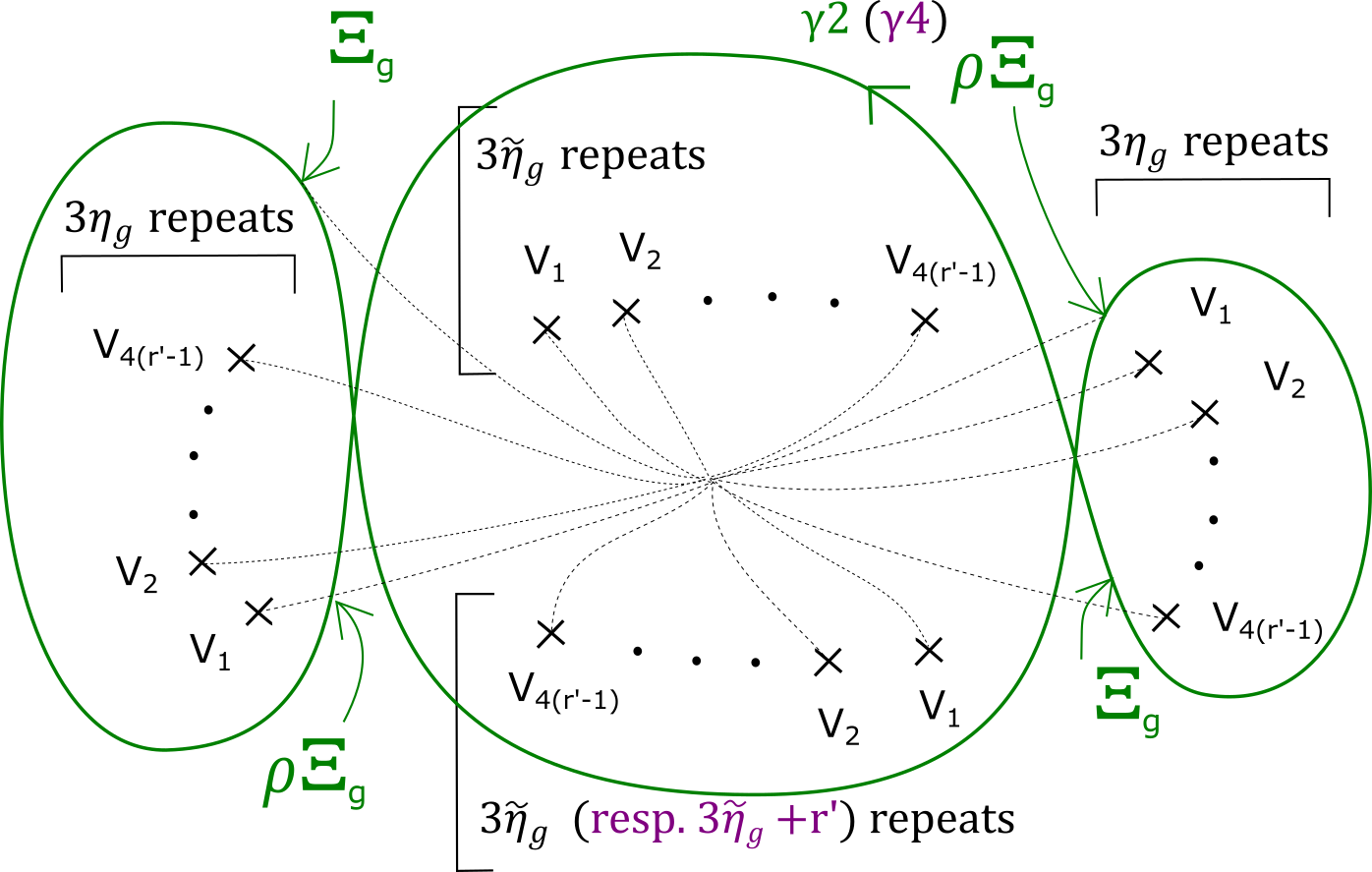}
\caption{ The monotone Lagrangian $\gamma_2 \times \Xi_g$, given by its image in (part of) the base of $P_{r,t}$.}
\label{fig:S^1xXi_g_Maslov2}
\end{center}
\end{figure}

\begin{figure}[htb]
\begin{center}
\includegraphics[scale=0.30]{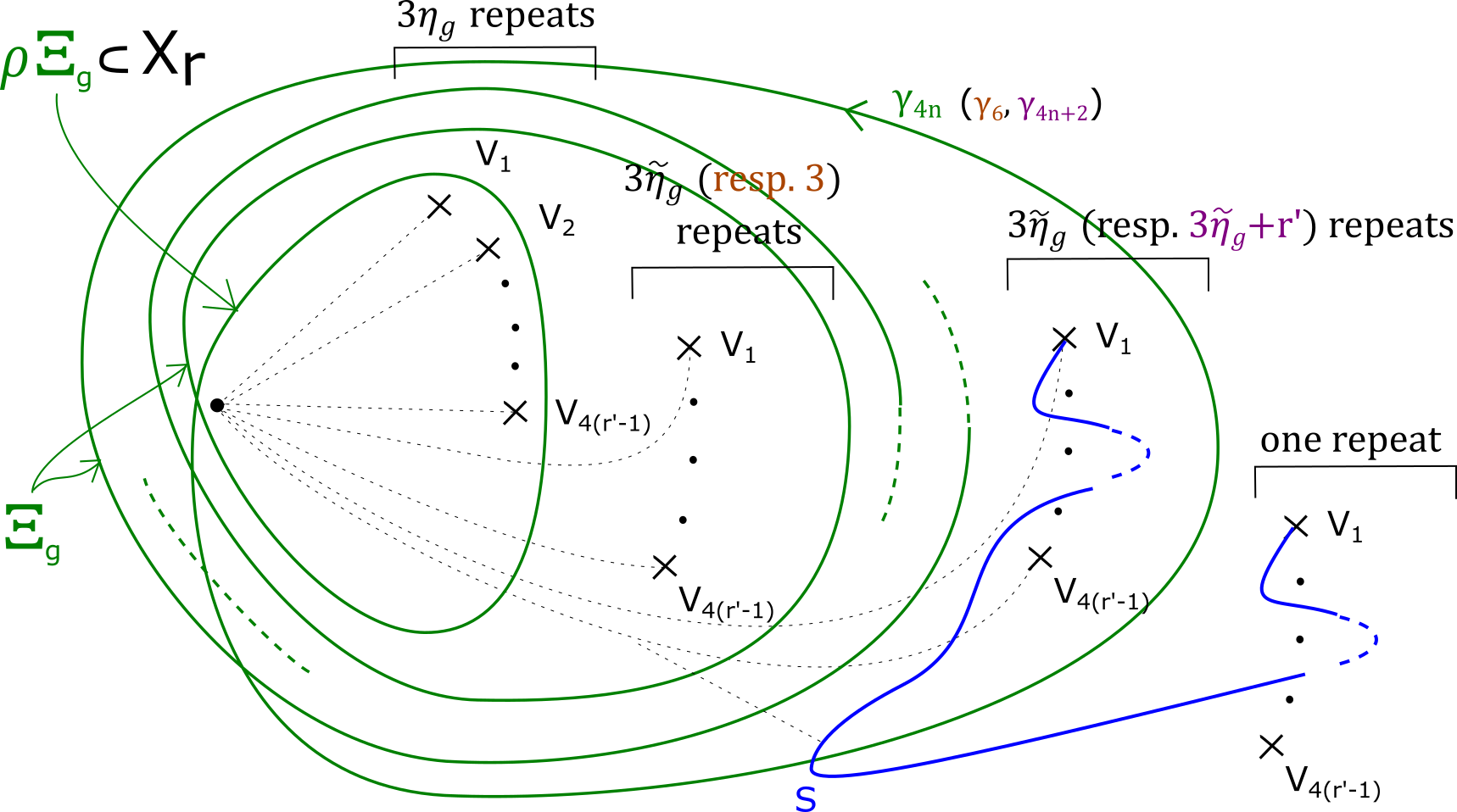}
\caption{ The monotone Lagrangian $\gamma_{2n} \times \Xi_g$, for $2n \geq 4$, given by its image in (part of) the base of $P_{r,t}$.
The labels for numbers of repeats apply to the sequence of vanishing cycles $V_1, \ldots, V_{4(r'-1)}$. For legibility we haven't labelled the Figure with the (total) winding numbers of the $\gamma_k$, which are as follows: the curve $\gamma_6$ has winding number two; for $n \geq 2$, the curves $\gamma_{4n}$ and $\gamma_{4n+2}$ have winding number $n$. (For $\gamma_8$, resp.~$\gamma_{10}$, the convention is that the two clusters of $3 \tilde{\eta}_g$ critical points (resp.~ the clusters of $3 \tilde{\eta}_g$ and  $3 \tilde{\eta}_g+r'$ points) are grouped together.)
Finally, the blue arc gives a matching path, corresponding to some vanishing cycle $S$, which will not in general be one of the $V_i$ and we will consider in Section \ref{sec:Floer_theory}.
}
\label{fig:S^1xXi_g_generalMaslov}
\end{center}
\end{figure} 

Our choice of indexing is explained by the following:

\begin{proposition}\label{prop:Maslov_index_3d}
For $n \geq 0$, the Maslov index of the lift of $\gamma_{2n}$ to a path $\gamma_{2n} \times \{ p \}$, given by the image under parallel transport of a point $ p \in \Xi_g$, is $2n$.
\end{proposition}

\begin{proof}
Let $D \subset \C$ be a disc containing a small segment of $\Pi_r(\Xi_g)$. 
Pick explicit representatives for $\rho$ and $\tilde{\rho}$, given by composing the `standard' representatives for Dehn twists in matching cycles in a Lefschetz fibration, as fibred symplectomorphisms (as before see \cite[Figure 16.3]{Seidel_book}), with support away from $\Pi^{-1}_r(D)$. Assume $\Pi_r(p)$ belongs to the segment in $D$. Let $\sigma_{2n}$ be the composition of copies of $\rho$ and $\tilde{\rho}$ given by the total monodromy about $\gamma_{2n}$. We know that $\sigma_{2n}$ is Hamiltonian isotopic to $\Phi^{t_{2n} r'}$, where $t_{2n} r'$ is the signed total number of repeats of the list of Dehn twists $\tau_{V_1} \tau_{V_2} \ldots \tau_{V_{4(r'-1)}}$. (Informally,  $\sigma_{2n}$ corresponds to $t_{2n}$ full, i.e.~ of angle $2\pi$, rotations of the base of $\Pi_r$.) 

Consider the Hamiltonian isotopy of $X_r$ induced by dragging the disc $D$ around by a rotation of $2\pi$ (with large compact support), say $\psi$. Inspecting the proof of Proposition \ref{prop:Phi_Dehntwists}, we see that the Hamiltonian isotopy taking $\psi \circ \sigma_{2n} (\Xi_g)$ to $\Xi_g$ can then be arranged to be relative to $\Pi^{-1}_r (\psi (D))$. We can use this to calculate the  Maslov index of $\gamma_{2n} \times \{ p \}$:
\begin{itemize}
\item the base contributes twice the winding number of $\gamma_{2n}$; 
\item the fibre contributes $2t_{2n}$ from the effect of $\psi$ on $p$ (recall that we arranged for a neighbourhood of $p$ to be fixed by $\sigma_{2n}$). 
\end{itemize}
This completes the proof.
\end{proof}

By adjusting the sizes of the different `lobes' (e.g.~taking the picture to be symmetric for the $n=0$ case), the lift of $\gamma_{2n}$ can also be arranged to bound a disc of arbitrary symplectic area. Thus the Lagrangians we have constructed, which in a slight abuse of notation we will denote $\gamma_n \times \Xi_g$, can be taken to be monotone.

\begin{remark}
One shouldn't  expect to be able to realise our construction $\gamma_2 \times \Xi_g$ into $\C^3$: we'll see in Proposition \ref{prop:Floer_coho_3d} that its Floer cohomology with a Lagrangian $S^3$ is non-zero. 
\end{remark}

\begin{remark}\label{rmk:variations_on_Xi_g} 
There are plenty of variations using Polterovich sums of other combinations of $T_{k,l,m}$, $R_{n,p,q}$ and $S_{u,v,w}$, or other tori constructed using the same ideas; and also Klein bottles, as considered in Section \ref{sec:non-orientable}. 
\end{remark}

\begin{remark}
Let $\Sigma$ be the Polterovich connected sum of a collection of tori and / or Klein bottles as before. 
We can use powers of $\Phi$ to write down a product of negative Dehn twists $\rho$ of $X_r$ such that not only  $\rho \Sigma \cap \Sigma = \emptyset$, but also $\Pi_r (\Sigma) \cap \Pi_r (\rho \Sigma) = \emptyset$. One can then contruct Lagrangians in $Y_{r,s}$ of the form $\gamma_{2n} \times \Sigma$, where $\gamma_{2n}$ is an immersed curve in the base of $P_{r,s}$, such that above the transverse intersection points the Lagrangian is given by $\Sigma \sqcup \rho {\Sigma}$. This may be useful in other circumstances. 
\end{remark}

\begin{remark}
If we allowed ourselves to work with more general classes of Liouville domains, we could make similar constructions by using the trick of `doubling' a Lefschetz fibration, i.e.~taking its double branched cover over a fibre as described in \cite[Section 18(a)]{Seidel_book}. This replaces a Lefschetz fibration 
 $\pi: U \to D$, with distinguished collection of vanishing cycles say $v_1, \ldots, v_r$, with a Lefschetz fibration $\pi: U \# U \to D$ with the same smooth fibre, and distinguished collection of vanishing cycles $v_1, \ldots, v_r, v_1, \ldots, v_r$. By construction, there are matching cycles, say $V_i$, between the critical values corresponding to the two $v_i$. Further, we have that $\tau_{V_1} \ldots \tau_{V_r }$ acts as a $\Z/2$ rotation on (large compact subsets of) the first and second copies of $U$ -- see \cite[Lemma 18.1]{Seidel_book}.
 One can also further exploit this by instead trebbling, etc the fibration. 
\end{remark}


\subsection{Partial compactifications of $X_r$ and $Y_{r,s}$}\label{sec:partial_compactifications}

\subsubsection{Partial compactifications}

Consider  $M = \{ x^3+y^3=1 \} $, the smooth fibre of $X_r$. Fix a partial compactification of $M$ corresponding to capping off one (and only one) of the punctures, between the curves $b$ and $d$, as in Figure \ref{fig:D_4compactified}. Call this $\bar{M}$. This carries a symplectic form which extends the one on a large compact subset of $M$ . Note that we could have arranged for the compactification to be arbitrarily `far out', i.e.~for the symplectic area of the cylinder enclosed between $b$ and $d$ in $\bar{M}$ to be arbitrarily large.

\begin{figure}[htb]
\begin{center}
\includegraphics[scale=0.50]{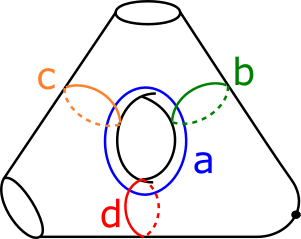}
\caption{ The compactification $\bar\Sigma$.}
\label{fig:D_4compactified}
\end{center}
\end{figure} 

Using Lemma \ref{lem:fibration_J}, we immediately see that this induces a partial compactification of $X_r$, with each of the fibres of $\Pi_r$ having one puncture capped off. Call it $\bar{X}_r$, and keep the same notation for the fibration, namely $\Pi_r: \bar{X}_r \to \C$. 
The symplectic form $\tilde{\omega}$ given by Lemma \ref{lem:fibration_J} extends to a symplectic form on $\bar{X}_r$, which we also denote $\tilde{\omega}$. Outside a large compact set, it is a product with respect to $\Pi_r$. Similarly, the almost complex structure $\tilde{J}_k$ given by Lemma \ref{lem:fibration_J} extends to one on $\bar{X}_r$, which is a product outside a large compact set in $X_r$.

Iterating, using Lemma \ref{lem:fibration_J} and Remark \ref{rmk:bifibred_J}, this induces in turn a partial compactication of $Y_{r,s}$,  say $\bar{Y}_{r,s}$,  given by partially compactifying each of the fibres of $P_{r,s}$. This can be equipped with a symplectic form $\tilde{\omega} $ and an almost complex structure $\tilde{J}$, both of which are products with respect to the bifibration $(P_{r,s}, \pi_r)$ outside a large compact set in $Y_{r,s}$. Note also that both projection maps are pseudo-holomorphic with respect to these choices.

In order to prevent pseudo-holomorphic curves from `escaping to infinity', throughout this section we restrict ourselves to almost-complex structures which agree with $\tilde{J}$ outside of a (possibly arbitrarily large) bounded set. 

The union of the `point at infinity' on each copy of $\bar{M}$ gives a divisor in $\bar{X}_r$, and in turn in $\bar{Y}_{r,s}$. We will call both of these `divisors at infinity' $D$. (With the choices we have made $D$ is naturally almost-complex.)

\subsubsection{Homology, first Chern class and monotonicity}
We have that
$$
H_2 (\bar{X}_r , \Z) = H_2 (X_r, \Z) \oplus \Z
$$
where the second term is generated by, say, $H$, the class of the annulus in $\bar{M}$ bounded by curves $b$ and $d$ and capped off by two Lagrangian thimbles ending on each of $b$ and $d$.
Moreover, $H_2 (\bar{Y}_{r,s}, \Z) = \Z\langle H \rangle$, and $\tilde{\omega}_{\bar{Y}_{r,s}} (H) = \tilde{\omega}_{\bar{X}_r}(H)$ can be arbitrarily large, depending on our choice of compactifications; in particular, neither of those symplectic forms is exact.

Notice that the trivialisation of $T  M = T \{ x^3 + y^3 =1 \}$ given in Figure \ref{fig:reference_Lagrangian_line2} and used for Maslov class computations extends to a trivialisation of $T \bar{M}$; in a suitable identification with a twice-punctured square with sides glued in pairs, the reference tangent lines have slope one. Further, this in turn readily induces trivialisations of $T \bar{X}_r$ and $T \bar{Y}_{r,s}$, extending those of $T X_r$ and $T Y_{r,s}$. In particular, $c_1(\bar{Y}_{r,s}) = 0$ and $c_1(\bar{X}_r)=0$, and all of our Maslov index computations are unchanged. (Note however that neither $\bar{X}_r$ nor $\bar{Y}_{r,s}$ are monotone, because of the class $H$.)

\subsection{Holomorphic annuli counts in 3 dimensions}\label{sec:annuli_counts_3d}

\subsubsection{Annuli counts for monotone $\gamma_{2n} \times \Xi_g$ for $ n\neq 0$}

Fix a monotone Lagrangian $\gamma_{2n} \times \Xi_g \subset Y_{r,s}$, say with monotonicity constant $\kappa$.
Assume in this subsection that $n \neq 0$. 
 As the Lagrangian is monotone but not exact, the Maslov class  $[\mu] \in H^1(S^1 \times \Xi_g, \R)$ is non-zero. 

Recall that from Weinstein's tubular neighbourhood theorem, Lagrangians sufficiently close to $\gamma_{2n} \times \Xi_g$ in $Y_{r,s}$ correspond to graphs of closed one-forms $\alpha \in \Omega^1(\gamma_{2n} \times \Xi_g)$. As the class $[\mu]$ is non-zero, such a graph will itself be a monotone Lagrangian in $Y_{r,s}$ if and only if $[\alpha] \in  H^1(S^1 \times \Xi_g, \R)$ is a multiple of $[\mu]$. We pick a representative for this carefully, as follows. 
The class $[\mu]$ corresponds to a homotopy class of smooth maps $f_\mu: \gamma_{2n} \times \Xi_g \to S^1$ such that $df_\mu \in \Omega^1 (\gamma_{2n} \times \Xi_g)$ satisfies $[df_\mu] = [\mu]$. As $n \neq 0$, $f_\mu$ is non-trivial on the first factor, and so has a representative with no critical points. Moreover, any two such can be interpolated by representatives with no critical points.
Fix a Weinstein tubular neighbourhood for $\gamma_{2n} \times \Sigma_g$. 
The graph of $\epsilon df_\mu$ gives a displacement of $\gamma_{2n} \times \Xi_g$ in the monotone direction, i.e.~the direction determined by $[\mu]$,  which is disjoint from the original, and itself monotone with monotonicity constant $\kappa' > \kappa$. Moreover, the above considerations also show that for fixed (and sufficiently close to $\kappa$) $\kappa'$, any two such disjoint monotone displacements are Hamiltonian isotopic, through other $\kappa'$ monotone displacements disjoint from the original.
 Let's call $L$ the original Lagrangian and $L'$ its monotone displacement.

We will fix a choice of monotone displacement which is also fibred with respect to $P_{r,s}$ and $\Pi_r$, given by pushing the $S^1$ factor (i.e.~$\gamma_{2n}$) in the base of $P_{r,s}$ off itself, say to $\gamma_{2n}'$, to get a parallel copy enclosing sligthly more signed area, and, for the $\Xi_g$ factor,  expanding lobes of $S^1$ components in the base as needed to increase the monotonicity constant by the same amount. (We then attach the tori fibred above these $S^1$s via Polterovich surgery using the same matching cycles as before to get the displacement of $\Xi_g$, say $\Xi'_g$.) While $\Xi_g'$ inside $X_r$ will generally intersect $\Xi_g$ (in particular whenever $g \geq 2$), $\gamma_{2n} \times \Xi_g$ and its displacement are disjoint because of the effect of the $\gamma_{2n}$ factor. (Note that the intersection points of $\gamma_{2n}$ and $\gamma_{2n}'$ are naturally in two-to-one correspondence with the self-intersection points of $\gamma_{2n}$. Suppose that $\Xi_g \sqcup h(\Xi_g)$ lies above one such point, where $h$ is the monodromy from the self-intersection point back to itself. Then we get $\Xi'_g \cup h(\Xi_g)$ and $\Xi_g \cup h(\Xi'_g)$ above the corresponding two points of $\gamma_{2n} \cap\gamma_{2n}'$; and for $\Xi'_g$ a sufficiently small displacement of $\Xi_g$, these are both disjoint unions. It then follows that  $\gamma_{2n} \times \Xi_g$ and its displacement are disjoint.)

Let $\alpha \in H_1 (L, \Z)$ be any primitive class of Maslov index zero. (We follow the standard convention that `primitive' implies non-zero.) Let $\alpha' \in H_1 (L', \Z)$ be the corresponding class under the natural isomorphism $H_1 (L, \Z) \cong H_1(L', \Z)$ induced by the displacement.
Let $A \in H_2 (\bar{Y}_{r,s}, L \sqcup L'; \Z)$ be a class such that $\partial A = \alpha \sqcup \alpha' \in H_1(L \sqcup L'; \Z) = H_1(L, \Z) \oplus H_1 (L', \Z)$. 

\begin{definition}\label{def:moduli_space}
Let $J$ be an almost complex structure on $\bar{Y}_{r,s}$. We define the moduli space
$$
\mathcal{M}_J(\bar{Y}_{r,s}; A, \alpha)  
$$
to consist of all $(j,J)$--pseudo-holomorphic maps $u: (\mathcal{A}, j) \to (\bar{Y}_{r,s}, J)$ such that
\begin{itemize}
\item $(\A, j)$ is a holomorphic annulus (of arbitrary modulus), with oriented boundary components $\partial_1$ and $\partial_2$ (these are intrinsically oriented but not ordered);
\item $u(\partial_1) \subset L$ and $u(\partial_2) \subset L'$;
\item $[u(\A)] =A, [u(\partial_1)] = \alpha$ and $[u(\partial_2)] =  \alpha'$.
\end{itemize}
See Figure \ref{fig:annulus}.

\begin{figure}[htb]
\begin{center}
\includegraphics[scale=0.40]{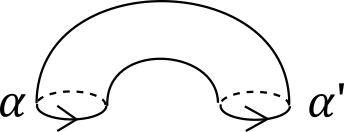}
\caption{ Boundary conditions for the holomorphic map $u$.}
\label{fig:annulus}
\end{center}
\end{figure} 

Let $\bar{\mathcal{M}}_J (\bar{Y}_{r,s}; A, \alpha)$ be the quotient of $\mathcal{M}_J(\bar{Y}_{r,s}; A, \alpha) $ by rotations and the $\Z/2$ involution of the annulus (which trades the two boundary components). 

\end{definition}

The abstract moduli space of annuli has one conformal parameter, namely the modulus of the annulus. It also has a one-dimensional family of automorphisms, given by the rotations and the $\Z/2$ involution mentioned above. By \cite[Theorem 1.2]{Liu_thesis}, $\bar{\mathcal{M}}_J (\bar{Y}_{r,s}; A, \alpha)$ has expected dimension zero.

\begin{proposition} \label{prop:regularity_monotone}
Let $\tilde{J}$ be our choice of almost-complex structure on $\bar{Y}_{r,s}$ from Section \ref{sec:partial_compactifications}. 
Let the pair $(A, \alpha)$ be any as above: $A \in H_2 (\bar{Y}_{r,s}, L \sqcup L'; \Z)$, $\alpha \in H_1(L;\Z)$ a Maslov zero primitive class, and $\partial A = \alpha \sqcup \alpha'$, where $\alpha' \in H_1 (L';\Z)$ is the image of $\alpha$ under the displacement. 

Then 
$\mathcal{M}_{\tilde{J}} ( \bar{Y}_{r,s}; A, \alpha)$ is empty 
for all but one choice of $(A, \alpha)$, for which $J$ is regular.  (We will see that this choice of $\alpha$ is of the form $(0,\tilde{\alpha}) \in H_1(S^1) \oplus H_1 (\Xi_g)$, and that this choice of $A$ satisfies $A \cdot D = 1$, where $D \subset \bar{Y}_{r,s}$ is the divisor at infinity.)

Moreover, for that choice of $(A, \alpha)$, $\mathcal{M}_{\tilde{J}} ( \bar{Y}_{r,s}; A, \alpha)$ has dimension one, and, assuming we are working with 
$$
\Xi_g = \Xi_g (\{k_0, l_0, m_0 \}, \{ k_1, l_1, m_1 \}, \ldots, \{ k_g, l_g, m_g \} ),
$$
the signed count of points in $\bar{\mathcal{M}}_{\tilde{J}} ( \bar{Y}_{r,s}; A, \alpha)$ is $h_n k_1$, where $h_n$ is the number of self-intersections of the $\gamma_{2n}$ curve in the base of $P_{r,s}$.

\end{proposition}

\begin{proof}
Suppose $u: \A \to \bar{Y}_{r,s}$ is a $(j, \tilde{J})$--pseudo-holomorphic map satisfying the conditions above. The map $P_{r,s} \circ u: \A \to \C$ is holomorphic with boundary components on $\gamma_{2n}$ and $\gamma_{2n}'$. 
Now note from standard complex analysis that such holomorphic annuli in $\C$ have boundary $c \gamma_{2n} \cup -c \gamma_{2n}'$, some integer $c$. As $\alpha$ and $\alpha'$ are \emph{equal} under the identification given by the small perturbation, $c=0$. 
It then follows that $P_{r,s} \circ u$ must be constant, with image lying in one of the intersection points of $\gamma_{2n}$ and $\gamma'_{2n}$. Thus $u(\A) \subset P_{r,s}^{-1}(p) \cong \bar{X}_r$, some $p \in \gamma_{2n}\cap \gamma'_{2n}$. 

Now consider the restriction $u: \A \to \bar{X}_{r}$. We can use similar considerations again: $\Pi_r \circ u: \A \to \C$ is a holomorphic map such that the images of the boundary components of $\A$ lie in $\Pi_r (\Xi_g)$ and $\Pi_r (\rho (\Xi_g'))$ respectively. In this case open mapping theorem type considerations applied to the configurations of Figures \ref{fig:fibred_linked_tori}  and \ref{fig:Xi_g}  show that the image of $\A$ must in fact lie in an intersection point of $\Pi_r (\Xi_g)$ and $\Pi_r (\rho (\Xi_g'))$. There are $h_n(2k_1+ l_1 + 1)$ such intersection points. Above each of these points, the fibre of $\Pi_r$ is $\bar{M}$, and there is a unique simple holomorphic annulus with boundaries on the restrictions of $\Xi_g$ and $\rho(\Xi'_g)$, which is just the `obvious' annulus enclosed between the vanishing cycles $b$ and $d$ in $\bar{M}$. 
(The covers of the simple annulus do not have primitive boundary classes.)

Above $h_n (k_1+l_1+1)$ of the intersection points of $\Pi_r (\Xi_g)$ and $\Pi_r (\rho (\Xi_g'))$, these annuli have boundaries of the form $\alpha \sqcup -\alpha'$, for some class $\alpha \in H_1(\Xi_g)$; these correspond to yellow and blue intersections in Figure \ref{fig:fibred_linked_tori}. Above the other $h_n k_1$ intersection points live annuli with boundaries of the form $\alpha \sqcup \alpha'$, corresponding to yellow and green intersections in Figure \ref{fig:fibred_linked_tori}; these are the ones we want. Note that this argument also gives the claimed uniqueness of $\alpha$.

Regularity follows from noticing that $\gamma_{2n} \pitchfork \gamma_{2n}'$ at the point $P_{r,s} \circ u (\A)$, and $\Pi_r (\Xi_g) \pitchfork \Pi_r (\rho (\Xi_g'))$ at the point $\Pi_r \circ u (\A)$, and applying one-dimensional regularity results. The signs agree at each point of the moduli space as the local configurations are identical.
\end{proof}

\begin{proposition}\label{prop:monotone_invariance}
The count of pseudo-holomorphic annuli is well-defined: for any pair $(A, \alpha)$ as above and any regular $J$, the signed count of points in $\bar{\mathcal{M}}_{J} ( \bar{Y}_{r,s}; A, \alpha)$ is the same as for $\bar{\mathcal{M}}_{\tilde{J}} ( \bar{Y}_{r,s}; A, \alpha)$.
\end{proposition}

\begin{proof}
As $J$ agrees with $\tilde{J}$, and so with a product, outside a very large but bounded set, we don't need to be concerned about pseudo-holomorphic curves running off to infinity in the horizontal direction. 
This means we `only' need to worry about two different possible phenomena: first, the fact that the abstract moduli space of holomorphic annuli has boundary; and, second, the possibility of disc bubbling. 

\emph{Excluding the boundary of the moduli space of annuli.} The abstract moduli space of holomorphic annuli has two boundary components: modulus zero and modulus infinity (see \cite[Figure 9]{Liu_thesis}). In the modulus zero case, the two boundary components of the annulus intersect; this is precluded in our case by the condition  $u(\partial_1) \subset L$ and $u(\partial_2) \subset L'$, as $L \cap L' = \emptyset $. In the modulus infinity case, one boundary component has shrunk to a point; this is precluded in our case by the conditions $[u(\partial_1)] = \alpha$ and $[u(\partial_2)] = \alpha'$ together with the assumption that $\alpha \neq 0 \in H_1(L, \Z)$ (and similarly for $\alpha'$).

\emph{Excluding disc bubbling.} It's enough to consider the pair $(A, \alpha)$ for which $\mathcal{M}_{J_0} ( \bar{Y}_{r,s}; A, \alpha)$ is non-empty. Note that the class $A$ is primitive, and intersects the `divisor at infinity' $D$ in exactly one point, with multiplicity one. There are two separate possibilities for bubbling, depending on where the intersection point with the divisor at infinity goes, as given in Figure \ref{fig:bubbling}.

\begin{figure}[htb]
\begin{center}
\includegraphics[scale=0.40]{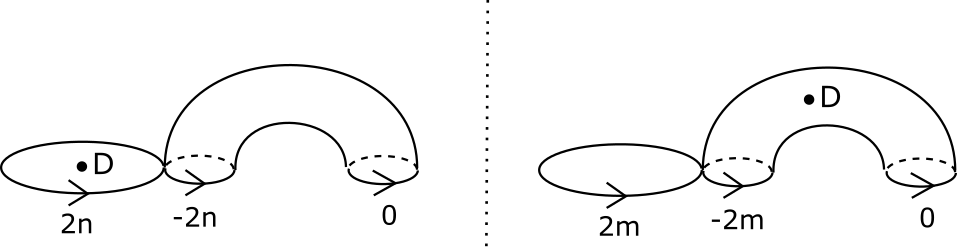}
\caption{ The two possibilities for bubbling, depending on the position of the intersection point with $D$, together with Maslov indices  of boundary curves ($n, m \in \Z$).}
\label{fig:bubbling}
\end{center}
\end{figure} 

In the first case (left-hand side of Figure \ref{fig:bubbling}), the pseudo-holomorphic disc cannot be constant, as its boundary lies on $L$ or $L'$. Moreover, as the multiplicity of the intersection of the disc with $D$ is one, it must be a simple curve -- following \cite[Section 3.2]{McDuff-Salamon} in the closed case. 
Thus it is regular for a generic $J$; after quotienting out by holomorphic automorphisms, by \cite[Theorem 1.2]{Liu_thesis}, the moduli space of such discs has expected dimension $2n$, and so $n \geq 0$. 
The annulus lies in $Y_{r,s}$, and has non-negative symplectic area, with oriented boundary components of Maslov indices $-2n$ and $0$. 
Now notice that the component of Maslov index $-2n$ can be filled in by a (homotopy class of a) disc in $Y_{r,s}$ with positive symplectic area (and boundary, with the orientation induced from the disc, of Maslov index $2n$). After the filling, we get a class in $\pi_2 (Y_{r,s}, L \, \text{or} \, L')$ which has positive symplectic area but Maslov index zero, a contradiction. This rules out the first possibility for bubbling.

The second case is similar but simpler. Analogously to before the annulus must be regular, and so $m \leq 0$. But by monotonicity of $L \subset Y_{r,s}$, it must then be that $m=0$ and the disc is constant. 
\end{proof}

\begin{remark}
To rule out e.g. the first bubbling case, alternatively, one could note that 
$$
H_2 (\bar{Y}_{r,s}, L; \Z) \cong \Z \langle H \rangle \oplus H_1 (L).
$$
We have that $\omega(\A) = \omega(H)$; on the other hand, if $\beta \in H_1(L)$ has Maslov index $2n$ and $\hat{\beta}$ is the lift of $\beta$ to $H_2 (Y_{r,s}, L; \Z)$, then $\omega(\hat{\beta}) = 2n \kappa$, where $\kappa$ is the monotonicity constant of $L$. In particular, a disc through $D$ with boundary $\beta$ would have symplectic area $\omega(H) + 2n\kappa$, so at least equal to that of $\A$ as $n$ is non-negative. This leaves no energy for the annulus in $Y_{r,s}$ -- a contradiction.  

\end{remark}

\subsubsection{Annuli counts for $\gamma_0 \times \Xi_g$}

In the case where $L$ is $\gamma_0 \times \Xi_g$, it does not have a disjoint monotone displacement (except possibly when $g=0$), because the Maslov index of $\gamma_0$ is zero. We need to refine the argument of the previous section to deal with this complication.

We make the following observation, the proof of which is immediate.

\begin{lemma}\label{lem:nonexact_deformation}
Consider $\gamma_0 \subset \C$. It has open neighbourhoods $\nu_{1/2} \subset \nu \subset \C$ such that:
\begin{itemize}
\item $\nu$ does not not cover the entirety of any of the components of $\C \backslash \gamma_0$ (think of $\nu_{1/2}$ and $\nu$ as small thickenings of $\gamma_0$);
\item there exists $\delta > 0$ such that for any point $p \in P_{r,s}^{-1} (\C \backslash \nu)$, there is a symplectic embedding of $B_\delta(0)$ centered on $p$ and whose image is disjoint from $P_{r,s}^{-1}(\nu_{1/2})$. 
\end{itemize}

\end{lemma}

Let $\gamma_0' \subset \C$ be a parallel displacement of $\gamma_0$ enclosing signed area $\epsilon >0$. As this can be taken to be arbitrarily small, assume that $\epsilon < \delta^2 /4$. Let $L'$ be given by $\gamma_0' \times \Xi_g$. 

Let $\tilde{\alpha} \in H_1 (\Xi_g)$ be any primitive Maslov zero class, and $\alpha = (0,\tilde{\alpha})$ be the corresponding class in $H_1(\gamma_0 \times \Xi_g)$; let $\alpha' \in H_1 (L'; \Z)$ be the image of $\alpha$ under the displacement. Let $A \in H_2 (\bar{Y}_{r,s}, L \sqcup L'; \Z)$ be such that $\partial A = \alpha \sqcup \alpha'$.

Given an almost complex structure $J$ on $\bar{Y}_{r,s}$, we define $\mathcal{M}_J(\bar{Y}_{r,s}; A, \alpha) $ and $\bar{\mathcal{M}}_J (\bar{Y}_{r,s}; A, \alpha)$ as before. 

\begin{proposition}\label{prop:regularity_invariance_nonmonotone}

We have the following analogues of Propositions \ref{prop:regularity_monotone} and \ref{prop:monotone_invariance}.

\begin{itemize}

\item[(a)] 

Let $\tilde{J}$ be our choice of almost-complex structure on $\bar{Y}_{r,s}$ from Section \ref{sec:partial_compactifications}, and let $(A, \alpha)$ be any pair as above. Then 
$\mathcal{M}_{\tilde{J}} ( \bar{Y}_{r,s}; A, \alpha)$ is empty 
for all but one choice of $(A, \alpha)$, for which $J$ is regular. 
Moreover, for that choice of $(A, \alpha)$, $\mathcal{M}_{\tilde{J}} ( \bar{Y}_{r,s}; A, \alpha)$ has dimension one, and 
the signed count of points in $\bar{\mathcal{M}}_{\tilde{J}} ( \bar{Y}_{r,s}; A, \alpha)$ is $k_1$.

\item[(b)]
The count of pseudo-holomorphic annuli is well-defined: for any pair $(A, \alpha)$ as above and any regular $J$, the signed count of points in $\mathcal{M}_{J} ( \bar{Y}_{r,s}; A, \alpha)$ is the same as for $\mathcal{M}_{\tilde{J}} ( \bar{Y}_{r,s}; A, \alpha)$.

\end{itemize}

\end{proposition}

\begin{proof}
For part (a), the proof of Proposition  \ref{prop:regularity_monotone} completely carries over: nowhere does it use monotonicity. Note that we in fact get something slightly stronger: as with Proposition   \ref{prop:regularity_monotone}, the argument shows that the count is non-zero for a unique choice of $A \in H_2 (\bar{Y}_{r,s}, L \sqcup L'; \Z)$ and $\alpha \in H_1(L; \Z)$ a Maslov zero primitive class such that $\partial A = \alpha \sqcup \alpha'$. 

More care needs to be taken for part (b), for which the proof of Proposition \ref{prop:monotone_invariance} needs refining.
As $L$ and $L'$ are disjoint, and $\alpha \neq 0$, we again do not need to worry about the boundary of the abstract moduli space of annuli. Disc bubbling, however, now requires more care. The casework of Figure \ref{fig:bubbling} is now further broken down into four cases, depending on the position of the marked point and whether the bubble is formed on the side of $L$ or $L'$. 

Note that the regularity considerations for simple curves used in the proof of Proposition \ref{prop:monotone_invariance} still hold, as do the virtual dimension counts of \cite[Theorem 1.2]{Liu_thesis}.

Assume first the bubble is formed on the side of $L$.

\begin{itemize}

\item If the bubble doesn't go through $D$, it must be constant, following the corresponding part of the proof of Proposition  \ref{prop:regularity_monotone}. 

\item If the bubble goes through $D$, we get a contradiction as before by using symplectic area considerations. 
\end{itemize}

Now assume that the bubble is formed on the side of $L'$. For concreteness, the two possibilities are given in Figure \ref{fig:bubbling_2}. 

\begin{figure}[htb]
\begin{center}
\includegraphics[scale=0.40]{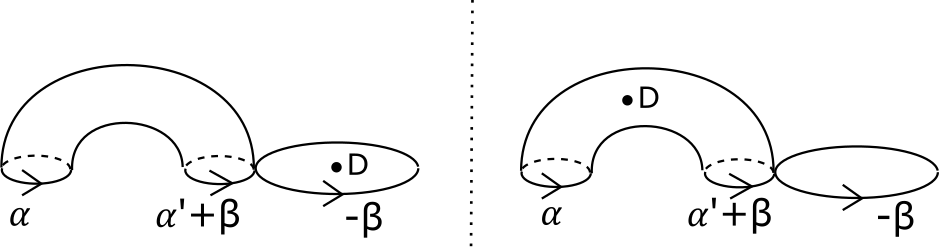}
\caption{The two possibilities for bubbling, depending on the position of the intersection point with $D$, together with the homology classes of the boundaries.}
\label{fig:bubbling_2}
\end{center}
\end{figure} 

Let us first rule out the right-hand side scenario in Figure \ref{fig:bubbling_2}. As before the annulus is regular and so has non-negative Maslov index. Thus the disc has non-positive Maslov index. With respect to the obvious decomposition, say $-\beta = (k \gamma_0', -\tilde{\beta}) \in H_1(\gamma_0' \times \Xi_g)$, where $-\tilde{\beta} \in H_1 (\Xi_g)$ has non-positive Maslov index, and $k$ is an integer. Thus the disc has symplectic area $k \epsilon + \kappa \mu(-\tilde{\beta}) \leq k \epsilon$, which in particular implies $k > 0$. Let $F=P^{-1}(\star)$ be the fibre above a point $\star$ in the middle of a lobe of $\gamma_0$ and $\gamma_0'$ (in particular, away from $\nu$). Now notice that if $-\hat{\beta} \in H_2 (Y_{r,s}, L'; \Z)$  has boundary $-\beta$, then for topological reasons we must have that $F$ intersects $-\hat{\beta}$ at least $k$ times, counted with signs. Moreover, by the monotonicity lemma for minimal surfaces, $-\hat{\beta}$ must have symplectic area at least $k \pi \delta^2$, a contradiction. (Readers will be familiar with this flavour of application of the monotonicity theorem from one of the standard proofs of Gromov non-squeezing, for instance as given in McDuff--Salamon \cite[Theorem 9.3.1]{McDuff-Salamon}.)

Let us now look at the left-hand side scenario. Similarly we must have $\beta = (k \gamma_0', \tilde{\beta}) \in H_1 (\gamma_0' \times  \Xi_g)$, where $k$ is a constant and $\tilde{\beta} \in H_1(\Xi_g)$ has non-positive Maslov index. We have $\alpha'+ \beta = (k \gamma_0', \tilde{\alpha} + \tilde{\beta})$. The symplectic area of the annulus is $k \epsilon + \kappa \mu(\tilde{\beta}) \leq k \epsilon$, so $k > 0$. Pick $\star \in \C$ as before, and let $C \in H_2 (Y_{r,s}, L \sqcup L'; \Z)$ be the unique class with boundary $\alpha \oplus \alpha' + \beta$. Again, topological considerations show that $C$ must intersect $P^{-1}(\star)$ a signed total of $k$ times, and the monotonicity lemma applies again to get a contradiction.
\end{proof}

\begin{remark}
For genus one, there are splittings of $L$ as $S^1 \times \Sigma_g$ other than the `obvious' $\gamma_0 \times \Xi_g$.
Assume that there is a compactly supported symplectomorphism $f$ of $Y_{r,s}$ taking $\gamma_0 \times \Xi_g$ to $\gamma_0 \times \Xi_g'$ and which does not respect the splitting. Then the above proof, using $f(P^{-1}(\star))$, etc, would show that the count of points in $\bar{\mathcal{M}}_{J} ( \bar{Y}_{r,s}; f( A), f(\alpha))$ is well-defined. On the other hand, as noted in the proof, we know that it is non-zero for a unique choice of $A \in H_2 (\bar{Y}_{r,s}, L \sqcup L'; \Z)$ and $\alpha \in H_1(L; \Z)$ Maslov zero primitive such that $\partial A = \alpha \sqcup \alpha'$. This gives a contradiction: $f$ must in fact have respected the splitting. 
\end{remark}

\subsubsection{Conclusion}
Start with a monotone Lagrangian $\gamma_{2n} \times  \Xi_g \subset Y_{r,s}$, where recall that  $$\Xi_g = \Xi_g ( \{ k_0, l_0, m_0 \}, \ldots, \{ k_g, l_g, m_g \} ).$$
Suppose you're given a compactly supported symplectomorphisms of $Y_{r,s}$. Choosing a compactification of $Y_{r,s}$ to $\bar{Y}_{r,s}$ which contains the whole domain, the previous section shows that the count of holomorphic annuli $\bar{\mathcal{M}}_{J} ( \bar{Y}_{r,s}, \gamma_{2n} \times \Xi_g; A, \alpha)$, for any regular $J$ and the one class $(A, \alpha)$ for which it's non-zero, is an invariant of $\gamma_{2n} \times \Xi_g$ which is unchanged under compactly supported symplectomorphisms. 
In particular, if we're also given $\Xi'_g = \Xi_g ( \{ k'_0, l'_0, m'_0 \}, \ldots, \{ k'_g, l'_g, m'_g \} )$ with $k_1 \neq k_1'$,
we see that the Lagrangians $\gamma_{2n} \times \Xi_g$ and $\gamma_{2n} \times \Xi'_g$ must be distinct up to compactly supported symplectomophisms. 
Putting everything together, we get the following theorem.

\begin{theorem}\label{thm:main_3d}
Fix $g$. For any sufficiently large $r$ and $s$, we can construct an infinite family of homologous monotone Lagrangian $S^1 \times \Sigma_g$ in $Y_{r,s} = \{ x^3 + y^3 + z^r + w^s =1 \}$, with fixed arbitrary even Maslov class and monotonicity constant, distinct up to compactly supported symplectomorphisms of $Y_{r,s}$. (This includes the exact case.)

\end{theorem}

\begin{remark}
The conclusion of Theorem \ref{thm:main_3d} also hold for well-behaved non-compactly supported symplectomorphisms, 
namely all of those which can be extended to the compactifications. 
\end{remark}

\subsection{Holomorphic annuli counts in 2 dimensions}\label{sec:annuli_counts_2d}

As annuli have Euler characteristic zero, by \cite[Theorem 1.2]{Liu_thesis} the expected dimension of a moduli space of annuli only depends on their Maslov class, and not on the dimension of the ambient space. This means that we can hope for well-defined  counts Maslov zero annuli in complex dimension two. (In complex dimension four or higher, one immediately runs into trouble because of the possible existence of discs with negative Maslov indices.) 

In particular, we're able to get a two-dimensional version of Theorem \ref{thm:main_3d}
in the case of tori.

\begin{theorem}\label{thm:main_2d}
For any sufficiently large $r$, we can construct an infinite family of homologous monotone Lagrangian tori in $X_r= \{ x^3 + y^3 + z^r  =1 \}$, with fixed arbitrary even Maslov class and monotonicity constant, distinct up to compactly supported symplectomorphisms of $X_r$. (This includes the exact case.)

\end{theorem}

\begin{proof}
We will just use the family of Lagrangian tori given by $T_{k,l,m}$ (see Figure \ref{fig:fibred_torus_Tklm2}), although completely analogous arguments can be carried out for  other families, with changes to the annuli count formulae. 
Our proof will follow from inspecting the arguments of Section \ref{sec:annuli_counts_3d} and noticing that they can be replicated, indeed with some simplifications. We briefly note what changes are to be made.

Let $T = T_{k,l,m}$.
As before, let us first consider the case where $T$ is monotone but not exact. 
Define moduli spaces  $\mathcal{M}_J (\bar{X}_r; A, \alpha)$ and $\bar{\mathcal{M}}_J (\bar{X}_r; A, \alpha)$  following Definition \ref{def:moduli_space}. For our favoured almost-complex structure $\tilde{J}$ from Section \ref{sec:partial_compactifications}, one gets an analogue to Proposition \ref{prop:regularity_monotone}, with the one difference being that for the sole choice of $(A, \alpha)$ such that  $\mathcal{M}_{\tilde{J}} (\bar{X}_r; A, \alpha)$ is non-empty (in which case it is as before one-dimensional), the signed count of points in $\bar{\mathcal{M}}_{\tilde{J}} (\bar{X}_r; A, \alpha)$ is $k_1$. 
To see that this count is a well-defined invariant as $J$ varies, let us re-visit the proof of Proposition \ref{prop:monotone_invariance}. The boundary of the abstract moduli space of holomorphic annuli gets avoided as before; and the arguments to exclude disc bubbling carry over, in fact with small simplifications: in this dimension, if a moduli space of index $2n$ discs is regular, then after quotienting out by holomophic automorphisms it has expected dimension $2n-1$, so $n > 0$.

Let us now turn to the case where $T$ is exact. We let $T'$ be a small \emph{non-exact} deformation of $T$. We construct this analogously to Lemma \ref{lem:nonexact_deformation}, but using $\Pi_r$ instead of $P_{r,s}$: we use a close parallel copy of the immersed $S^1$ in the base of $\Pi_r$ (which is now playing the role of $\gamma_0$), and don't touch the `meridional' $S^1$ (which is in a fibre). We're then able to follow the proof of Proposition \ref{prop:regularity_invariance_nonmonotone}; as in the previous case, the expected dimensions of spaces of holomorphic discs get adjusted in our favour. 
\end{proof}

Note that the assertions of Remark \ref{rmk:linking} about the `linking' of the tori $R_{k,l,m}$ and $S_{n,p,q}$ now readily follow by considering similar such holomorphic annuli counts. 

\subsection{Extensions and limitations}\label{sec:extensions_limitations}

Because of their reliance on partial compactifications, the proof of Theorems \ref{thm:main_3d} and \ref{thm:main_2d} will not survive under a general exact symplectic embedding of Liouville domains: instead, if one wants the same conclusion for e.g.~the Milnor fibre of a `larger' singularity, one would need it to have a compactification in which holomorphic annuli still appear.

In the two-dimensional case, by Remark \ref{rmk:A_3}, the conclusion of Theorem \ref{thm:main_2d} would also hold for $\{ x^2 + y^4 + z^r = 1 \}$ for sufficiently large $r$. 

More generally, we see that the constructions and arguments of Sections \ref{sec:S^1xXi_g} through \ref{sec:annuli_counts_2d} (and in particular Section \ref{sec:partial_compactifications}) can also be made whenever $x^3+y^3$, i.e.~the singularity $D_4$, is replaced by $x^3+xy^u$, i.e.~the singularity $D_{u+1}$ (while not Brieskorn--Pham, this is still weighted homogeneous, so Remark \ref{rmk:weighted_homogeneous} applies). 

On the other hand, suppose we try to replace $D_u$ by an even larger singularity in two variables. Suppose we're given two disjoint vanishing cycles on the associated Milnor fibre. One can check the following: if removing those vanishing cycles from the Milnor fibre disconnects it, then both of the resulting components (there can't be more than two) have positive genus. This means that the analogue of counting pseudo-holomorphic annuli would now be to count higher genus pseudo-holomorphic curves (with Maslov index zero, and two boundary components). While in dimension 3 these have expected dimension zero, at least naively one can't hope for a well-defined count, as we can't rule all of the possible configurations in the boundary of the abstract moduli space of such holomorphic curves: there is no a priori reason to be able to rule out an internal connecting node on the curve (as in \cite[Figure 8]{Liu_thesis}). 

\section{Floer-theoretic properties}\label{sec:Floer_theory_properties}

\subsection{Distinguishing Lagrangians using Floer theory}\label{sec:Floer_theory}
We'll show that the families of Lagrangians which we construct in $X_r$ and in $Y_{r,s}$ (in Section \ref{sec:S^1xXi_g}) can be distinguished using their Floer homology group with a fixed Lagrangian sphere. From one perspective, this is a weaker invariant than the pseudo-holomorphic annuli counts of Section \ref{sec:annuli}, as it shows that the Lagrangians in the family are different up to Hamiltonian isotopy, rather than arbitrary compactly supported symplectomorphisms; on the other hand, this invariant survives under e.g.~exact symplectic embeddings of Liouville domains. We'll also see that it allows us to distinguish some Lagrangians for which the holomorphic annuli counts were all the same, or for which we didn't define the count (for instance, genus $g$ Lagrangians in $X_r$).

For the standard almost-complex structure $J$, there are no non-constant holomorphic discs in $X_r$ with boundary on $T_{k,l,m}$, as an immediate consequence of the open mapping theorem. In particular, the count of Maslov index two discs with boundary on $T_{k,l,m}$ vanishes (of course this is automatic if $T_{k,l,m}$ is exact or has minimal Maslov number four or more, so we would in fact only have to worry about the case where a `longitude' of $T_{k,l,m}$ has Maslov index two).  This means that the Floer cohomology group of $T_{k,l,m}$ with e.g.~a Lagrangian sphere $L$ is well defined. We will take our Floer cohomology groups to have coefficients in $\C$, and decorate the Lagrangians with a choice of rank one local system, following the set-up in \cite[Chapter 2]{Seidel_book}. In the case where $T_{k,l,m}$ is exact, $HF(T_{k,l,m}, L)$ can be equipped with a $\Z$--grading; in general, it can be equipped with a $\Z/ 2|k-l|$--grading. 

\begin{proposition}\label{prop:floer_coefficients}
Then we can find Lagrangian spheres in $\{x^2+y^4+z^{18}=1\}$, say $S_1, S_2$ and $S_3$, which are vanishing cycles for the singularity $x^2+y^4+z^{18}$ (and thus given by matching cycles in the base of a Lefschetz fibration given by $\epsilon(x,y)+z$ for a small generic linear $\epsilon$), such that for any choice of rank one local system on $T_{k,l,m}$,
$$
\text{rk } HF(T_{k,l,m}, S_1) = m+2k+1; \qquad \text{rk } HF(T_{k,l,m}, S_2) = 2k+1; \qquad \text{rk } HF(T_{k,l,m}, S_3) = 2l.
$$
In particular, the Lagrangians $T_{k,l,m}$ and $T_{k', l', m'}$, constructed in $\{x^2+y^4+z^{18}=1\}$ using the same basic configurations (i.e.~configurations of critical points of type $BB$ and $BC$) cannot be Hamiltonian isotopic whenever $(k,l,m) \neq (k', l', m')$. 
\end{proposition}

\begin{proof}  Let $S_1$ be the matching cycle $S$ of Figure \ref{fig:fibred_torus_Tklm2}. (Following Remark \ref{rmk:A_3}, we drop the $d$ curve and work in $X = \{x^2+y^4+z^{18}=1\}$.) While $S$ isn't one of the distinguished collection vanishing cycles for $x^2+y^4+x^r$ which are given by matching cycles in Figure \ref{fig:distinguished_matching_cycles}, it is given by applying to one of them (a `type $a$' matching cycle) negative Dehn twists in a sequence of vanishing cycles which appear further down the list in the distinguished collection (these also correspond to `type $a$' matching cycles). In particular, $S$ itself in a vanishing cycle for $x^2+y^4+z^r$. 

The sphere $S=S_1$ intersects $T_{k,l,m}$ in $m+2k+1$ points, labelled $p_1, \ldots, p_m, q_{m+1}, \ldots, q_{m+2k+1}$ on the figure. We claim that the standard almost-complex structure is regular, and that for that choice there are no holomorphic discs between intersection points. 

First, we have that
$$
CF(S, T_{k,l,m}) = \C \langle p_1, \ldots, q_{m+2k+1} \rangle.
$$
Moreover, as a holomorphic disc in $X$ would project to one in $\C$ with the same boundary conditions, single variable complex analysis heavily constrains the possibilities for holomorphic discs: by inspection, the only possible differentials are from $p_i$ to $p_j$ with  $j>i$,  $p_i$ to $q_j$, or $q_i$ to  $q_j$ with $j>i$; see Figure \ref{fig:Tklm_disc} for a possible disc. 

\begin{figure}[htb]
\begin{center}
\includegraphics[scale=0.35]{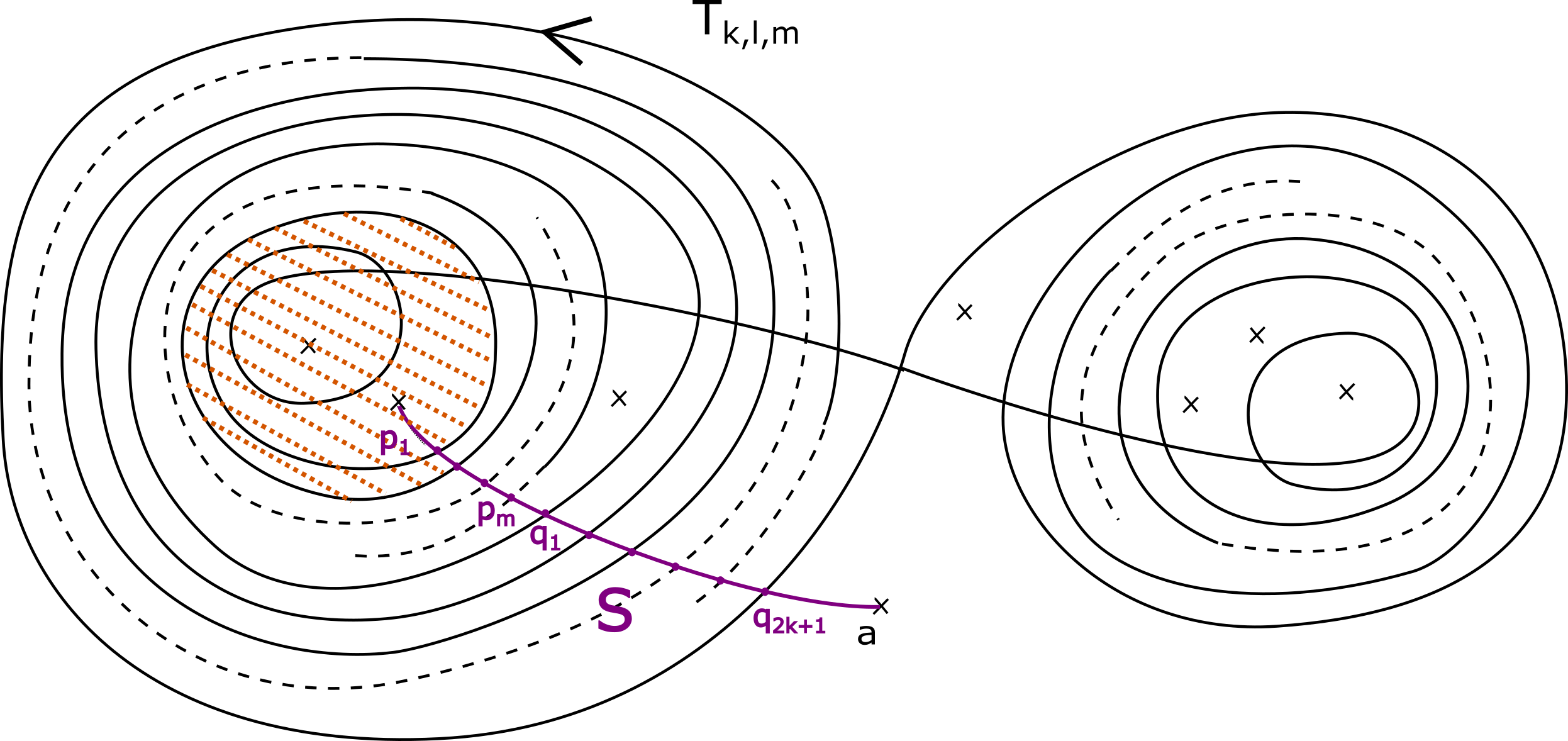}
\caption{Image in $\C$ of topological discs between $p_1$ and $p_2$. We've omitted some labels for legibility; these can be found in Figure \ref{fig:fibred_torus_Tklm2}.}
\label{fig:Tklm_disc}
\end{center}
\end{figure} 

On the other hand, the Maslov index calculations of Lemma \ref{thm:Maslov_indices} imply that these topological discs have either Maslov index zero, in the case of $p_i$ and $p_j$, or negative Maslov index, in the case of $p_i$ (or $q_i$) and $q_j$. Thus they must have zero or negative symplectic area, and cannot have holomorphic representatives (even the area zero ones are not constant). (Alternatively, one could note that by also by an application of the open mapping theorem, the holomorphic discs contributing to the differential would be the same for $T_{k,l,m}$ and $T_{k,k,m}$ -- this just uses the fact that the image of such a disc can't cross the inflection point between the two `lobes' of $T_{k,l,m}$ -- and then make use of the $\Z$--grading on $CF(S, T_{k,k,m})$.) This proves the claim for $S_1 = S$.

Going back to Figure \ref{fig:fibred_torus_Tklm2}, we can write down matching cycles $S_2$ and $S_3$, constructed similarly to $S_1$ but starting at a type $a$ critical point in two different $BB$, such that $S_2$ intersects $T_{k,l,m}$ in $2k+1$ points, and $S_3$ in $2l$ points; a completely analogous argument shows that with the standard almost-complex structure there are no holomorphic discs contributing to the differential on $CF(T_{k,l,m}, S_i)$, $i=1,2$, which completes the proof.
\end{proof}

\begin{remark}
If we just cared about determining the triple $(k,l,m)$, the Floer cohomology with $S_1$, together with the Maslov class and the homology class, would suffice -- the characterisation using only Floer groups will be useful later.
\end{remark}

There are clearly similar statements for $R_{n,p,q}$ and $S_{u,v,w}$  the triples $(n,p,q)$ and $(u,v,w)$ are determined by the rank of the Floer cohomology groups of $R_{n,p,q}$ and $S_{u,v,w}$ with some vanishing cycles for $x^3+y^3+z^r$.  Similarly for the parameters definining the Klein bottles of Section \ref{sec:non-orientable} (and plenty of other multi-parameter families of tori or Klein bottles one might construct analogously).

We also get the following corollary.

\begin{corollary}\label{cor:floer_coefficients_genus_g}
Fix any $r$ large enough such that the family $\Lambda_g (\{k_1, l_1, m_1 \}, \ldots, \{k_g, l_g, m_g \})$ can be realised in $X_{2,4,r}$. Then all of the parameters $(k_1, l_1, m_1), \ldots, (k_g, l_g, m_g)$ are determined by the ranks of Floer cohomology groups of $\Lambda_g$ with vanishing cycles for $x^2+y^4+z^r$. 

Similarly for other genus $g$ Lagrangians built using the same techniques, such as $\zeta_g$ of Section \ref{sec:Hamiltonian_monodromy} and $\Xi_g$ of Section \ref{sec:S^1xXi_g} (with vanishing cycles for  $x^3+y^3+z^r$), as well as further variations by taking a connected sum of a mixture of tori and Klein bottles. 
\end{corollary}

\begin{proof}
The Floer cohomology calculations of Proposition \ref{prop:floer_coefficients} are unaffected by the Polterovich surgeries: open mapping theorem type considerations still apply to show that there can be no Floer differentials for the standard almost-complex structure. 
\end{proof}

Note that 
the calculations will remain valid under an exact symplectic embedding of Liouville domains, e.g.~by the integrated maximum principle \cite[Lemma 7.5]{Seidel_book}.
This completes the proof of Theorem \ref{thm:floer_2d}.

Upgrading to complex dimension three, we get the following.

\begin{proposition}\label{prop:Floer_coho_3d}
Fix $n \in \Z_{\geq 0}$, $g \geq 1$, and $r, s$ large enough such that the family of Lagrangians $\gamma_n \times \Xi_g$ of Section \ref{sec:S^1xXi_g} can be constructed in $Y_{r-1,s-1}$. 

Then there are Lagrangian spheres in $Y_{r,s}$, which moreover are vanishing cycles for $x^3+y^3+z^r+w^s$, whose Floer cohomology groups with $\gamma_n \times \Xi_g$ recover the coefficients $\{k_0, l_0, m_0\},$ \ldots, $\{ k_g, l_g, m_g\}$ determining $\Xi_g$.
\end{proposition}

\begin{proof}
Let $S$ be one of the test Lagrangian 2-spheres used for determining the coefficients of $\Xi_g$ in $X_r$ (this uses the fact that $\Xi_g$ itself can be constructed in $X_{r-1}$, to have a `extra' critical value for the second end of the matching path for the sphere). As it's a vanishing cycle for $x^3+y^3 +z^r$, we can find a matching path in the base of $P_{r,s}$ which intersects $\gamma_n$ in precisely one point, and such that in the fibre above that point, the corresponding matching cycle, say $\mathbb{S}$, restricts to $S$, and $\gamma_n \times \Xi_g$ restricts to $\Xi_g$. It is given back in Figure \ref{fig:S^1xXi_g_generalMaslov} for the case of a general $n$ ($n=0,2$ are similar), and in turn uses the assumption on $s$ to get the `spare' critical value needed to close off that matching cycle.

As the images of $\mathbb{S}$ and $S^1 \times \Xi_g$ in the base of $P_{r,s}$ intersect transversally in a single point, using the standard almost complex structure, it's immediate that the Floer cohomology of $\mathbb{S}$ and $S^1 \times \Xi_g$ is precisely given by the Floer cohomology of $S$ and $\Xi_g$, and the result then follows from Corollary \ref{cor:floer_coefficients_genus_g}. 
\end{proof}

\begin{remark}
The conclusion of Proposition \ref{prop:Floer_coho_3d} applies much more broadly:
\begin{itemize}
\item to families of Lagrangians of the form $S^1 \times \Sigma$, where the surface $\Sigma$ is a connected sum of an arbitrary number of tori and Klein bottles which could be constructed using any of the variations of Remark \ref{rmk:variations_on_Xi_g}. (In particular, the bounds on $r$ and $s$ needed in order to get infinite families of $S^1 \times \Sigma_g$ can be significantly improved from the naive ones one would get using the numbers in Section \ref{sec:S^1xXi_g}.)

\item to families of Lagrangians of the form $L= \#_{i=1}^l S^1 \times \Sigma_i$, obtained by taking Polterovich connected sums of the $S^1 \times \Sigma_i$ with matching spheres. (As before, the surfaces $\Sigma_i$ are themselves the connected sum of an arbitrary number of tori and Klein bottles.) 
This connect sum procedure is the three-dimensional analogue of the constructions $\Lambda_g$, $\Xi_g$ etc, and the same argument applies regarding ranks of Floer cohomology. 
More precisely, just as in the two-dimensional case, consider a matching cycle $\mathbb{S}$ use to detect the parameters used to define $\Sigma_i$, and whose image in the base of $P_{r,s}$ only intersects the image of $S^1 \times \Sigma_i$ for one $i$. The open mapping theorem applies as before to show that discs contributing to the differential on $CF(L, \mathbb{S})$ are exactly the same as the ones contributing to $CF(S^1\times \Sigma_i, \mathbb{S})$. 
\end{itemize}
\end{remark}

The discussion of Floer cohomology groups in the 3-dimensional case has been conducted in $Y_{r,s}$, i.e.~$\{ x^3 + y^3 + z^r + w^s =1 \}$. 
However, if we only care about the conclusions of Theorem \ref{thm:floer_3d}, we claim that it's enough to take $\{ x^2 + y^4 + z^r + w^s =1 \}$, for sufficiently large $r$ and $s$, as ambient space: we only needed $Y_{r,s}$ in Section \ref{sec:S^1xXi_g} in order to construct $\Xi_g$ (Figure \ref{fig:Xi_g}), with the carefully arranged overlay of $S_{k_0, l_0, m_0}$ and $R_{k_1,l_1,m_1}$ above some self-intersection points of $\gamma_{2n}$, in order to get a non-trivial count of holomorphic annuli. Working simply with $T_{k,l,m}$, we can write down similar but simpler constructions in $\{ x^2 + y^4 + z^r + w^s =1 \}$, with the chief difference being that we get, above any self-intersection point of the base $S^1$, $T_{k,l,m} \sqcup \rho (T_{k,l,m}) \subset X_{2,4,r}$ for some $\rho$ such that $\Pi(T_{k,l,m}) \cap \Pi(\rho (T_{k,l,m})) = \emptyset$, where $\Pi:X_{2,4,r} \to \C$ is the same Lefschetz fibration as before.

This completes the proof of Theorem \ref{thm:floer_3d}

\subsubsection{Further diffeomorphism types}\label{sec:further_diffeo_types}
We briefly note how the ideas Sections \ref{sec:finite_order_automorphisms} and  \ref{sec:Hamiltonian_monodromy} might be further developed when the ambient manifold is (for instance) $X_{2,4,r,s}$ for sufficiently large $r$ and $s$. 
For concretness, let $\zeta_g (\{ k_1, l_1, m_1 \}, \{ k_2, l_2, m_2 \})$ be as in Section \ref{sec:finite_order_automorphisms} ($g$ odd as before), and $\rho$ the product of positive Dehn twists in vanishing cycles for $X_{2,4,r}$ such that $\rho$ fixes $\zeta_g$ setwise but acts on it by an order $(g-1)/2$ rotation pointwise. Now note that at the cost of enlarging $r$, we can use Propositions \ref{prop:Phi_definition} and \ref{prop:Phi_Dehntwists} to find another product of positive Dehn twists in vanishing cycles for $X_{2,4,r}$, say $\sigma$, such that $\sigma (\zeta_g) \cap \zeta_g = \emptyset$; and also $\tilde{\sigma}$, again a product of \emph{positive} Dehn twists in vanishing cycles, such that $\tilde{\sigma} \sigma (\zeta_g) = \zeta_g$, where the latter identity holds pointwise (after deforming by a suitable Hamiltonian isotopy). 

Taking $s$ to be very large, we can now use $\rho, \sigma$ and $\tilde{\sigma}$ to mimick the constructions of Section \ref{sec:S^1xXi_g} to get monotone Lagrangians in $X_{2,4,r,s}$ which are $\Sigma_g$ bundles over $S^1$ (as always immersed in the base of $P_{r,s}$), with monodromy given by $\rho$. (Recall that $X_\mathbf{a}$ is defined in general at the start of Section \ref{sec:preliminaries}). The base $S^1$ can be arranged to have arbitrary even Maslov index; and Floer cohomology with matching cycles in $Y_{2,4,r,s}$ allows us to distinguish such Lagrangians up to isotopy. We can also use matching cycles to take connected sums of several copies of such Lagrangians, and / or with Lagrangians of the form $\Sigma_g \times S^1$; and one can come up with variations replacing $\Sigma_g$ with certain connected sums of tori and Klein bottles. (Also, dropping monotonicity would readily allow for e.g.~the case where $g$ is even.)

For concretness, we record the one exact case, using the Maslov index calculation of Section \ref{sec:finite_order_automorphisms}.

\begin{proposition}\label{prop:further_diffeo_types}
Let $E$ be the $\Sigma_3$ bundle over $S^1$ with monodromy of order two described above, and let $L$ be an arbitrary connected sum of copies of $E$ and $ S^1 \times \Sigma_g $, where $g$ can vary. Then for all sufficiently large $r$ and $s$, there exists an infinite family of exact monotone Lagrangians in $Y_{2,4,r,s}$ which are diffeomorphic to $L$, homologous, and distinct up to Hamiltonian isotopy. This is preserved under exact symplectic embeddings of $Y_{2,4,r,s}$.
\end{proposition}

\begin{remark}
We can't hope to get non-zero holomorphic annuli counts for non-trivial $\Sigma$ bundles over $S^1$ without significantly refining our constructions:  the set-ups of Figures \ref{fig:Zeta_g} and \ref{fig:Xi_g} can't readily be amalgamated while preserving all the features one would need.
\end{remark}

\subsection{Floer cohomology for families of 2-dimensional Lagrangians and comparison with cluster mutations}\label{sec:no_mutations}

Aside from acting by symplectomorphisms, in the two-dimensional case a well understood technique for getting new Lagrangians from old ones is to use disc surgeries, also known as geometric mutations. 
More precisely, given a Lagrangian surface $L$ in a symplectic four-manifold together with a Lagrangian disc with boundary on it, one can construct a new Lagrangian surface $L'$ via so-called `disc surgery' on the original one \cite{MLYau_surgery}. If $L$ is monotone, then under suitable conditions $L'$ is too; similarly with exactness. 
This construction has been shown to have rich connections with the theory of cluster mutations, explored e.g.~in \cite{STW, Pascaleff-Tonkonog}, where the disc surgeries are refered to as (geometric) mutations. 

The torus case is particularly well understood; we follow the discussion in \cite{Pascaleff-Tonkonog}. Given a Lagrangian seed, i.e.~a  monotone Lagrangian torus $T$ together with a collection of Lagrangian discs with boundary on $T$
\cite[Definition 4.7]{Pascaleff-Tonkonog}, 
the authors explain how to iteratively perform mutations on the torus \cite[Definition 4.9]{Pascaleff-Tonkonog}. Moreover, the local model for a single mutation is given by passing from the Clifford to the Chekanov torus in $\C^2 \backslash \{ xy=1 \}$ (see \cite[Sections 4.5--4.7]{Pascaleff-Tonkonog}). Further, the Floer cohomology for this local model is known, see e.g.~the calculation in \cite[Proposition 11.8]{Seidel_lectures}, which closely follows \cite{Auroux-tduality}. In this local model, depending on the choice of rank one local systems, the two tori are either isomorphic or the Floer cohomology between them vanishes. (The matching of local systems to get non-zero Floer cohomology is what gives the wall-crossing formula.)  In the exact case, the same remains true of $T$ and its geometric mutation $T'$.

Yau also introduces disc surgeries on higher genus Lagrangian surfaces \cite{MLYau_surgery}; they are also included in the discussion in \cite[Section 2]{STW} (though the reader may wish to recall the caveats of e.g.~\cite[Sections 1.2.4 and 2.2]{STW} regarding iterations of mutations). Note that there isn't a single model for disc surgery in this case, as it depends on the isotopy class of the boundary of the disc on the Lagrangian. However, given the result on local systems and microlocal sheaves of \cite[Equation 1]{STW}, one might nonetheless expect that depending on the choice of rank one local system, an exact Lagrangian surface $L$ and its geometric mutation $L'$ are either isomorphic or have Floer cohomology zero.

In contrast, our constructions yield the following.

\begin{theorem}\label{thm:no_mutations}

Let $L$ be a  copy of $T_{k,l,m}$  with monotonicity constant $\kappa$, and, for $\lambda > 0$, let $L'$ be a copy of   $T_{k+\lambda, l+ \lambda, m}$ with monotonicity constant $\kappa'$, constructed using the same basic configurations in $X_{u,v,r}$, for some  $u,v $ such that $ 1/u + 1/v \leq 2/ 3$ and $r \geq 18$.

Let $(\alpha, \beta) \in (\C^\ast)^2$ parametrise choices of rank one local systems on $L$, with respect to the same basis for $H_1(L, \Z)$ as in Lemma \ref{thm:Maslov_indices}; similarly, we'll use $(\alpha', \beta') \in (\C^\ast)^2)$ for $L'$. Then we have that 

$$
\text{rk } HF((L, (\alpha, \beta)), (L', (\alpha', \beta')) =
\begin{cases}
2 \lambda & \text{ if } \beta = \beta' \\
2 \lambda & \text{ if } \beta = -\beta' \\
0 & \text{ otherwise } 
\end{cases}
$$
This property is preserved under exact symplectic embeddings of Liouville domains. Moreover, in the case where $u=v=3$, we can combine this with Theorem \ref{thm:main_2d} to get families of Lagrangian tori not related by compactly supported symplectomorphisms, in particular Dehn twists. 

\end{theorem}
 
\begin{remark}\label{rmk:no_mutations}
If instead $r \geq 18g$ (or greater), one can similarly go about calculating the Floer cohomology groups of e.g.~the genus $g$ Lagrangians $\Lambda_g (T_{k_1, l_1, m_1}, \ldots, T_{k_g, l_g, m_g})$ for varying values of $k_i, l_i$ and $m_i$. While we don't do this in detail here, it will be clear from the proof of Theorem \ref{thm:no_mutations} that one can obtain infinite families of homologous monotone genus $g$ Lagrangians (all with the same Maslov class and monotonicity constant) such that the Floer cohomology between them, for suitable rank one local systems, takes arbitrary large rank finite rank. 
\end{remark}

\subsubsection*{Discussion}
There is a general expectation that the smoothing of the singularity $x^u + y^v + z^r$ (i.e.~its Milnor fibre) should be mirror to a resolution of this singularity: Brieskorn--Pham  singularities fall into the framework of Berglund--Hubsch \cite{BH93}, and are their own Berglund--Hubsch transposes; some versions of mirror symmetry have been proved in \cite{Futaki-Ueda, Favero-Kelly}, or \cite{Chan-Ueda} in the case of $A_n$. 
In particular, one might expect the  (wrapped) Fukaya category of $X_{u,v,r}$ to be mirror to a category of coherent sheaves associated with the resolution of the singularity $x^u+y^v+z^r$. 
The mirror-symmetric counterpart to the statement about mutations is that given their Floer-theoretic properties, the family of tori we construct can't correspond to cluster charts on some (fixed) mirror variety, or more generally to structure sheaves of points of $(\C^\ast)^2$ affine charts inside such a mirror variety.
This contrast with what we understand of a number of families of tori in two-dimensional examples, e.g.~in the 
case of $A_n$ \cite{Auroux-tduality, Chan-Ueda, Lekili-Maydanskiy}, for $\C P^2$ \cite{Auroux-tduality, Vianna_CP2} or del Pezzo surfaces \cite{AKO, Vianna_delPezzo}, and arguably most remarkably for log CY surfaces \cite{GHK_birational, GHKK, Pascaleff}; cluster structures associated to Grasmannians have also recenty been used to construct families of exotic Lagrangian tori in them \cite{Castronovo}.

\begin{proof}
We'll again use fibredness over $S^1$ to apply tools from standard complex analysis. First, we note that irrespective of $\kappa$ and $\kappa'$, we can always pick Hamiltonian deformations of $L$ and $L'$ which are still fibred, and such that the intersection between them is naively minimal, as given in Figure  \ref{fig:Tklm_HF} in the case of $T_{2,2,0}$ and $T_{1,1,0}$: namely, we first arrange for there to be $2 \lambda$'s worth of fibrewise $S^1$ intersections, and then deform each of these $S^1$s in the fibre direction to get two intersection points (the local model in the fibre is the zero-section in $T^\ast S^1$ and a small Hamiltonian perturbation of it). 
Given a fibred $T_{k+ \lambda, l+ \lambda, m}$, one can draw a fibred $T_{k,l,m}$ by using its projection to $\C$ as a guide, running closely parallel to it and then `skipping' some twists in both the left and right-hand sides lobes. (This was done in  Figure \ref{fig:Tklm_HF}.)
 To see that one needn't worry about possible variations in $\kappa$ and $\kappa'$, note that one can first take a Hamiltonian deformation of $L'$ (in particular, preserving $\kappa'$) such that the relative sizes of the  area contribuations from the parts of the left and right lobes that get `skipped' are arbitrary.

\begin{figure}[htb]
\begin{center}
\includegraphics[scale=0.4]{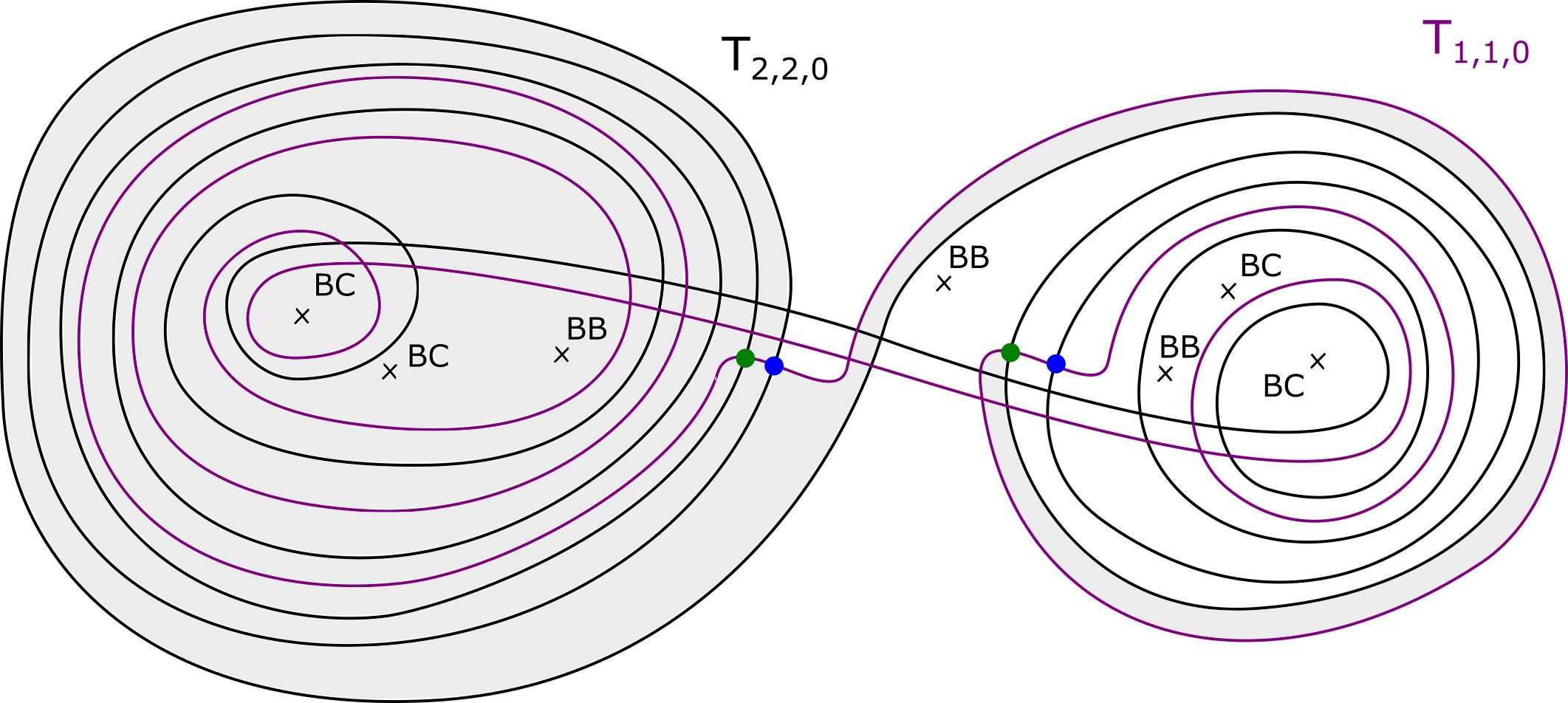}
\caption{Calculating $HF(T_{2,2,0}, T_{1,1,0})$. There are two intersection points above each of the marked dots, in each case coming from perturbing a fibrewise $S^1$. The colour of the dot encodes the orientation of the fibre type $b$ cycle for $T_{2,2,0}$, using the conventions of Figure \ref{fig:fibred_torus_Tklm2}. The shaded region is a potential projection of a holomorphic disc. 
}
\label{fig:Tklm_HF}
\end{center}
\end{figure}

With respect to these choices, $CF((L, (\alpha, \beta)), (L', (\alpha', \beta'))$ has $4 \lambda$ generators. Let's calculate the differential. Work with the standard $J$. 
Suppose there's a holomorphic disc $u: D \to X_{u,v,r}$ with boundary on $L$ and $L'$. We use the same trick as before: $\Pi \circ u: D \to \C$ must be a holomorphic map with boundary on $\Pi(L), \Pi(L')$. We claim that such holomorphic maps to $\C$ must be constant. 

Let us first consider potential holomorphic discs between two intersection points both on the left lobe. Start with the example of Figure \ref{fig:Tklm_HF}. 
Let $m_1, M_1$ be the intersection points above the left green dot, and $m_2, M_2$ be the ones above the left blue dot, ordered so that $\text{deg} (M_i) = \text{deg} (m_i) +1$.  
Using holomorphicity of $\Pi \circ u$, we see that  the only possible discs with non-constant projections have projections similar to the one in Figure \ref{fig:Tklm_disc}, which, given our ordering of $L$ and $L'$, would in our case give a differential from $m_1$ or $M_1$ to $m_2$ or $M_2$. 
(To rule out other discs, it's perhaps easiest to note that the boundary of the projection of any other potential holomorphic disc would differ from the one already described by an integer multiple of at least one of the two immersed $S^1$s in the base.) 

On the other hand, the  Maslov index calculations in the proof of \ref{thm:Maslov_indices} show that our potential disc would have index $-1$ if it were from $M_1$ to $M_2$, or from $m_1$ to $m_2$; and index zero from $M_1$ to $m_2$. 
If we generalise to work instead with $\lambda > 1$, we now see that the indices of potential holomorphic discs (still between intersection points both on the left lobe) would become more and more negative. 

The situation for holomorphic discs between intersection points both on the right lobe is completely analogous. Finally, let us consider holomorphic discs going between the two lobes. Holomorphicity of the projection map now implies that the point in right-hand lobe must project to the outermost green point; the shaded region in Figure \ref{fig:Tklm_HF} gives the potential holomorphic disc projection; now notice that such configurations have already been studied when thinking about the case with both intersection points in the left-hand lobe.

This mean for index reasons, we only need to worry about holomorphic discs whose image is contained entirely in a fibre of $\Pi$. There are $4 \lambda$ such discs, clearly regular: two in each of the fibres where $L$ and $L'$ intersect.
Tracking local systems, there are two possible configurations: the meridian in $T^\ast S^1$ and a push-off with the same orientation (green dot), in which case the contributions from the two discs cancel themselves out precisely when  $\beta = \beta'$; 
and the meridian in $T^\ast S^1$ and a push-off with the opposite orientation (blue dot), in which case the contributions from the two discs cancel themselves out precisely when  $\beta = - \beta'$.
\end{proof}

\begin{remark}
One could simplify the argument above, at least in the exact case (where the Lagrangians carry an absolute grading) at the cost of increasing $r$ and modifying the construction of $T_{k,l,m}$, for instance by adding two $BB$ type elementary configurations in each lobe to ensure that the indices of the intersection points were all further apart. 
\end{remark}

\begin{remark}
By constructing Lagrangians whose projections to $\C$ are disjoint, we can of course also construct (finite) families of Lagrangians all of which are Floer-theoretically disjoint.
\end{remark}

\subsection{Floer cohomology within families of 3-dimensional Lagrangians} \label{sec:floer_3d_family}

With suitable care, one would expect to be able to leverage the bifibration on $Y_{r,s}$, together with the fibred nature of the Lagrangians we construct, to calculate the Floer cohomology groups between at members of various families of Lagrangians from Section \ref{sec:S^1xXi_g}. We record the following observation for the exact case.

\begin{proposition}
Fix $g$. For suitably large $r, s$, there exists an infinite family of homologous, monotone exact Lagrangians of the form $S^1 \times \Sigma_g$ inside $Y_{r,s}$ such that for suitable choices of rank one local systems, the Floer cohomology between members of that family can take arbitrarily large rank.
\end{proposition}

As before, we only need $x^2+y^4+z^r+w^s$, and the conclusion persists if we take connected sums of such Lagrangians, or under inclusion of Liouville domains. 

\begin{proof}
This is similar to before. Let's look at the case of $\gamma_0 \times \Xi_g$, for two different choices of $\Xi_g = \Xi_g (\{ k_0, l_0, m_0 \}, \ldots, \{k_g, l_g, m_g \})$.  Considering a Hamiltonian perturbation of one of them whose projection to the base of $P_{r,s}$ is given by Figure \ref{fig:S^1xXi_g_Maslov0_HF}, we see that there are two collections of intersection points between the two Lagrangians (above each of the dots in the base), in natural one-to-one correspondence. Moreover, for each pair of such points, there are two holomorphic discs between them, whose contributions cancel for suitable choices of rank one local systems (this is merely a slightly unusual presentation of a standard configuration in $T^\ast S^1$; a neighbourhood of the zero section therein is immersed in the same of the fibration). Using the same circle of index and open mapping type considerations as before, the Floer cohomology calculation then boils down to finding 
the holomorphic discs in the fibres above each of the dots, which we already know from the proof of Theorem \ref{thm:no_mutations}.
\begin{figure}[htb]
\begin{center}
\includegraphics[scale=0.30]{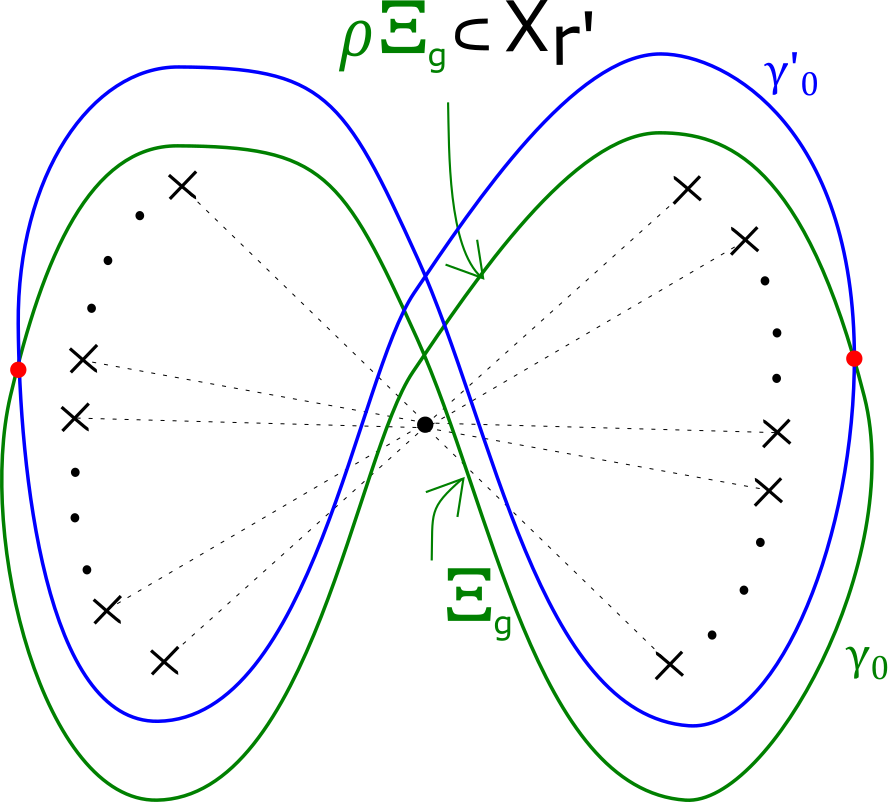}
\caption{Calculating the Floer cohomology between Lagrangians of the form $\gamma_0 \times \Sigma_g$. (Some labels are suppressed for legibility -- we're using the same configuration as Figure \ref{fig:S^1xXi_g_Maslov0}. The Hamiltonian perturbation of a Lagrangian $\gamma_0 \times \Sigma_g$ is fibred over the curve $\gamma_0'$. }
\label{fig:S^1xXi_g_Maslov0_HF}
\end{center}
\end{figure} 
\end{proof}


\section{Generalisations to higher dimensions}\label{sec:higher_dim}

Starting with our three-dimensional examples, we can iteratively use the monodromy tricks provided by Propositions \ref{prop:Phi_definition} and \ref{prop:Phi_Dehntwists}, as well as Polterovich surgery, to construct interesting families of monotone Lagrangians in higher-dimensional Brieskorn-Pham hypersurfaces.

More precisely, starting with the Lagrangians of Theorem \ref{thm:floer_3d} and iterating one dimension at a time, we can do the following:
\begin{enumerate}

\item Given a monotone Lagrangian $L$ in $X_1 = \{ z_0^{a_0} + \ldots + z_{m-1}^{a_{m-1}} = 1 \}$, get a monotone Lagrangian $S^1 \times L $ in $X_2 = \{ z_0^{a_0} + \ldots + z_{m-1}^{a_{m-1}} + z_m= 1 \} \cong \C^m$ by taking a product with  an embedded $S^1 \subset D \subset \C$ such that the Lefschetz fibration $z_m + \epsilon (z_0, \ldots, z_{m-1}): X_2 \to \C$ is essentially trivial above $D$. (The function $\epsilon (z_0, \ldots, z_{m-1})$ is a generic linear perturbation.) The $S^1$ factor has Maslov index two. 

\item 
Given a monotone Lagrangian $L$ in $X_1 = \{ z_0^{a_0} + \ldots + z_{m-1}^{a_{m-1}} = 1 \}$, for sufficently large $a_m$, get a monotone Lagrangian $S^1 \times L$ in $X_2' = \{ z_0^{a_0} + \ldots + z_{m-1}^{a_{m-1}} + z_m^{a_m}= 1 \}$ fibred over an immersed $S^1$ in the base of the Lefschetz fibration $z_m + \epsilon (z_0, \ldots, z_{m-1}): X_2 \to \C$,  analogously to the constructions of Section \ref{sec:S^1xXi_g}. The total monodromy about the immersed $S^1$ gives a symplectomorphism of the fibre $X_1$ which is, say, the $l$th power of the total monodromy $\varrho$ of $X_1$ viewed as the Milnor fibre of a singularity: using an auxiliary Lefschetz fibration of the form $z_{m-1} + \epsilon' (z_0, \ldots, z_{m-2}): X_1 \to \C$ and the notation of Section \ref{sec:rotations}, one would have $\varrho = \Phi^k$, meaning that in a large compact set, $\varrho$ is induced by a $2\pi$ rotation of the base of this fibration. (In particular, $\varrho^l$ acts as the identity on $L$.)

The Maslov index for this $S^1$ factor is calculated as in Proposition \ref{prop:Maslov_index_3d}: twice the total winding number of the immersed $S^1$ in the base, adjusted down by $2l$. In particular, the $S^1$ factor can take arbitrary even Maslov index. 

\item Using Polterovich surgery, for sufficiently large $a_m$, we can get any connected sums of the Lagrangians in (2), with the induced Maslov indices.

\end{enumerate}

The examples in (1) are clearly null-homologous. In (2) or (3) one can check that they are primitive in homology (with $\Z$ coefficients in the orientable case and $\Z/2$ otherwise) -- for instance, for any such Lagrangian we can find a matching cycle which intersects it in exactly one point. 
Let $(\dagger)$ the collection of possible pairs of diffeomorphism types and Maslov classes described by the iterative process above. (Note that for a given diffeomorphism type $L$, we can get as Maslov class any element of $H^1(L, \Z)$ which is not excluded by orientability-type considerations.)
 It is also clear that in situations (2) and (3) -- whenever $a_m$ is sufficiently large -- we can typically tell two examples apart by calculating their Lagrangians Floer cohomology with a fixed matching cycle. In particular,

\begin{proposition}\label{prop:floer_higher_dim}
Fix an $m$-dimensional  manifold  $L$ and $\mu \in H^1(L;\Z)$ such that  $(L, \mu) \in (\dagger)$, and any constant $\kappa > 0$. Then for sufficiently large $a_i$, there exist infinitely homologous monotone Lagrangian $L \subset \{ z_0^{2} + z_1^4 + z_3^{a_3}+ \ldots +z_m^{a_m}= 1 \} = X $, distinct up to Hamiltonian isotopy, with Maslov index $\mu$ and monotonicity constant $\kappa$. The bounds on the $a_i$ only depend on the diffeomorphism type of $L$. The class $[L] \in H_m (X)$ is primitive, where coefficients should be taken in $\Z$ or $\Z/2$ depending on orientability of $L$.

This property is preserved under exact symplectic embeddings. 
\end{proposition}

In particular, we get families of examples of monotone Lagrangians, including Lagrangian tori, which are primitive in homology but have arbitrarily high minimal Maslov number.

\subsection{Further diffeomorphism types}\label{sec:further_diffeo_higher_dim}

We briefly remark that with a little care, we can mimick the constructions of Section \ref{sec:finite_order_automorphisms} and \ref{sec:further_diffeo_types} in higher dimensions. Given a pair $(L, \mu) \in (\dagger)$, fix a monotone Lagrangian $L \subset  \{ z_0^{a_0} + \ldots + z_{m-1}^{a_{m-1}} = 1 \} = X$. Use this to construct a monotone Lagrangian $\tilde{L}$ in $ \{ z_0^{a_0} + \ldots + z_{m-1}^{2ka_{m-1}} = 1 \} = \tilde{X}$ as follows: first take $2k$ copies of $L$, arranged cyclically using the automorphism $z_{m-1} \mapsto \zeta z_{m-1} $ of $ \tilde{X}$, where $\zeta$ is a $(2k)$th root of unity; now perform Polterovich surgery on the copies of $L$ together with $2k$ (\emph{not} $2k-1)$ matching cycles, say $k$ of them corresponding to a fixed vanishing cycle, and $k$ of them corresponding to a second vanishing cycle, which is disjoint from the first one. 
By modifying the area enclosed by the chain of matching cycles we can arrange for the resulting Lagrangian to be monotone, and for there to exist products of positive Dehn twists on $\tilde{X}$, say $\rho, \sigma$ and $\tilde{\sigma}$, such that $\rho$ fixes $\tilde{X}$ setwise and acts as an order $k$ rotation pointwise; $\sigma(\tilde{L}) \cap \tilde{L} = \emptyset$; and $\tilde{\sigma} \sigma(\tilde{L}) = \tilde{L}$, where this identity holds pointwise  up to Hamiltonian isotopy. 
For sufficiently large $a_m$, this then allows us to construct monotone Lagrangians in   $ \{ z_0^{a_0} + \ldots + z_{m-1}^{a_{2km-1}} + z_m^{a_m} = 1 \} $ which are non-trivial $\tilde{L}$ bundles over $S^1$, with monodromy of order $k$. Further, as in Proposition \ref{prop:further_diffeo_types}, for fixed `soft' invariants these can often be distinguished up to Hamiltonian isotopy by the Floer cohomology with a fixed matching cycle.


\bibliography{bib}{}
\bibliographystyle{alpha}

\end{document}